\documentclass[a4paper,10pt]{article}
\usepackage{etex}
\usepackage[plainpages=false]{hyperref}
\usepackage{amsfonts,latexsym,rawfonts,amsmath,amssymb,amsthm, mathrsfs, lscape, calligra}
\usepackage{verbatim}
\usepackage{enumerate}

\usepackage[all]{xy}
\usepackage{graphicx,psfrag}
\usepackage{pst-node}
\usepackage{tikz-cd}
\usepackage{authblk}

\usepackage{array, tabularx, color}

\usepackage{setspace}

\newtheorem{thm}{Theorem}[section]
\newtheorem{cor}[thm]{Corollary}
\newtheorem{lem}[thm]{Lemma}

\newtheorem{prop}[thm]{Proposition}

\newtheorem{conj}[thm]{Conjecture}

\theoremstyle{remark}
\newtheorem{rmk}[thm]{Remark}

\theoremstyle{definition}
\newtheorem{defi}[thm]{Definition}

\numberwithin{equation}{section}

\def \C {\mathbb C}
\def \Z {\mathbb Z}

\def \F {\mathcal F}
\def \E {\mathcal E}
\def \V {\mathcal V}

\def \P {\mathbb P}

\def \T {\mathcal T}

\def \p {\partial}
\def \bp {\bar{\partial}}

\def \dVol {\text{dVol}}

\def \O {\mathcal{O}}
\def \l0 {\lim_{r\rightarrow 0}}

\def \Id {\text{Id}}
\def \Gr {Gr}
\def \Tr {\text{Tr}}
\def \Sing {\text{Sing}}
\def \S {\mathcal S}
\def \G {\mathcal G}
\def \I {\mathcal I}

\DeclareMathOperator{\rk}{rank}

\def \C {\mathbb C}
\def \Z {\mathbb Z}

\def \F {\mathcal F}

\def \P {\mathbb P}

\def \p {\partial}
\def \bp {\bar{\partial}}

\def \dvol {\text{dVol}}

\def \O {\mathcal{O}}
\def \l0 {\lim_{r\rightarrow 0}}

\def \E {\mathcal{E}}
\def \Tr {\text{Tr}}

\def \mod {\text{mod}}

\begin{document}
\title{Singularities of Hermitian-Yang-Mills connections and Harder-Narasimhan-Seshadri filtrations}
\date{\today}
\author{Xuemiao Chen\thanks{University of Maryland, xmchen@umd.edu}, Song Sun\thanks{UC Berkeley, sosun@berkeley.edu}}

\maketitle
\begin{abstract}
This is the first of a series of papers where we relate tangent cones of Hermitian-Yang-Mills connections at an isolated singularity to the complex algebraic geometry of the underlying reflexive sheaf, when the sheaf is locally modelled on the pull-back of a holomorphic vector bundle from the projective space. In this paper we shall impose an extra assumption that the graded sheaf determined by the Harder-Narasimhan-Seshadri filtrations of the vector bundle is reflexive. In general we conjecture that the tangent cone is uniquely determined by the double dual of the associated graded object of a Harder-Narasimhan-Seshadri filtration of an algebraic tangent cone, which is a certain torsion-free sheaf on the projective space. In this paper we also prove this conjecture when there is an algebraic tangent cone which is locally free and stable. 
\end{abstract}
\tableofcontents

\section{Introduction}
Let $(X,\omega)$ be an $n$ dimensional K\"ahler manifold, and $(E, H)$ be a Hermitian vector bundle over $X\setminus S$ for a closed set $S\subset X$ with locally finite real codimension four Hausdorff measure. A smooth unitary connection $A$ on $(E, H)$ is called an \emph{admissible Hermitian-Yang-Mills connection} on $X$ if the following two conditions hold
\begin{enumerate}[(1)]
\item $A$ satisfies the Hermitian-Yang-Mills equation
\begin{equation}\label{eqn1-1}
F^{0, 2}_A=0; \ \ \ \ \sqrt{-1}\Lambda_\omega F_A=n\mu\cdot \Id_E, 
\end{equation}
where $\mu\in \mathbb R$ is a constant. 
In the literature, (\ref{eqn1-1}) is also usually referred to as the \emph{Hermitian-Einstein equation} with  Einstein constant $n\mu$ --in this paper we will use both terminologies interchangeably;
\item $A$ has locally finite Yang-Mills energy, i.e. for any compact subset $K\subset X$, we have
\begin{equation}
\int_{K\setminus S} |F_A|^2 \frac{\omega^n}{n!}< \infty
\end{equation}
\end{enumerate}
In particular, $\bp_A$ defines a holomorphic structure on $E$ over $X\setminus S$. We denote the resulting Hermitian holomorphic vector bundle by $(\mathcal E, H)$. Then $A$ is the Chern connection associated to $(\mathcal E, H)$. Bando and Siu \cite{BS} proved that $\mathcal E$ \footnote{Strictly speaking here $\mathcal E$ should be the locally free sheaf generated by local holomorphic sections of $E$. In this paper, to make the notations simpler, we will not distinguish between a holomorphic vector bundle and the corresponding locally free sheaf.} naturally extends to a reflexive sheaf over  the whole $X$, and $H$ (hence $A$) extends smoothly  to the complement of the singular set of $\E$, which is a complex analytic subvariety of codimension at least three. 

There are several motivations for studying admissible Hermitian-Yang-Mills connections. First, from the complex geometric point of view, it is proved by Bando-Siu \cite{BS} that a polystable reflexive sheaf over a compact K\"ahler manifold always admits an admissible Hermitian-Yang-Mills connection, as a generalization of the Donaldson-Uhlenbeck-Yau theorem \cite{Donaldson1, UY} for holomorphic vector bundles. As a result, these connections have their relevance in algebraic geometry. Second, from gauge theoretic point of view, by \cite{Nakajima} (see also \cite{Tian})
 these admissible Hermitian-Yang-Mills connections naturally arise at the boundary of the moduli space of  smooth Hermitian-Yang-Mills connections with bounded Yang-Mills energy, as Uhlenbeck limits, and therefore they play an important role in understanding the structure of the compactified moduli space in gauge theory over higher dimensional  K\"ahler manifolds. The third motivation is that, in connection with gauge theory over $G_2$ manifolds,  singularities of admissible Hermitian-Yang-Mills connections in dimension three are expected to provide one possible model for  singularities of $G_2$ instantons (when the $G_2$ metric is close to the product of $S^1$ with a three dimensional Calabi-Yau metric) (see \cite{SE1, W, SW1, JW} for recent research along this direction).

Given an admissible Hermitian-Yang-Mills connection $A$, a natural and interesting question is to study the behavior of $A$ near a singular point $x\in S$. In this paper, we will always restrict to the special case when $S$ is discrete. This is largely due to technical reasons and we certainly hope this assumption will be removed in the future. So without loss of generality, we may assume that $X$ is the unit ball $B=\{|z|<1\}$ in $\C^n$, and $S=\{0\}$. We also always assume $n\geq 3$ in this paper, since the singularity is removable if $n\leq 2$. 
By the monotonicity formula of Price \cite{Price}, there exist \emph{tangent cones} of $A$ at $0$. These are obtained by pulling back the connection $A$ via dilations $z\mapsto \lambda z$ and then taking all possible Uhlenbeck limits as $\lambda\rightarrow 0$.  (see Section \ref{Tangent Cones}.)  It is  proved by Tian in \cite{Tian} that any such limit $A_\infty$ is an admissible Hermitian-Yang-Mills connection on $(\C^n, \omega_0)$ (here $\omega_0$ is the flat metric on $\C^n$) with vanishing Einstein constant, hence it defines a reflexive sheaf  on $\C^n$, which we denote by $\E_\infty$. Notice however, the \emph{uniqueness} of tangent cones is not a priori guaranteed, as in many other geometric analytic problems.

Our goal in this paper is to study the algebro-geometric meaning of the tangent cones in terms of the sheaf $\mathcal E$. To state the main result, we recall that given a torsion free coherent sheaf $\mathcal F$ over $\C\P^{n-1}$,  one can define a \emph{Harder-Narasimhan-Seshadri filtration} of subsheaves (c.f.  \cite{Kobayashi})
$$0=\mathcal F_0\subset \mathcal F_1\subset \mathcal F_2\subset \cdots \subset \mathcal F_l=\mathcal F$$
such that the quotients $Q_i=\mathcal F_i/\mathcal F_{i-1}$ are torsion free and stable, and furthermore, the slopes of $Q_i$ are decreasing. Such a filtration may not be unique in general but the double dual of the associated graded object 
$$(\Gr^{HNS}(\mathcal F))^{**}=\bigoplus_{i=1}^{l}Q_i^{**}$$
is canonical, i.e. it is uniquely determined by $\F$ up to isomorphism. 

\

For the convenience of later discussions we make a few notations and conventions which will be used throughout this paper.

\begin{itemize}
\item We will \emph{always} assume the Einstein constant $n\mu$ is zero. This does not affect the generality since we can always  achieve it simply by multiplying the Hermitian metric $H$ by a positive smooth function on $B$, and this does not change the tangent cone.
\item  We will choose local holomorphic coordinates  $\{z_1, \cdots, z_n\}$ centered at $0$ so that 
\begin{equation}
\omega=\omega_0+O(|z|^2)
\end{equation}
where
$$\omega_0=\sqrt{-1}\p\bp |z|^2 $$
is the flat K\"ahler metric. 
When we study tangent cones the effect caused by the difference is negligible. Also we emphasize that on the tangent cones the background K\"ahler metric is given by $\omega_0$. 
\item We denote by  $\omega_{FS}=\sqrt{-1}\p\bp \log |z|^2$ the Fubini-Study metric on $\C\P^{n-1}$.
\item On $\C^n$, we denote by $r=|z|$ the radial function for the flat metric $\omega_0$, and denote by $\p_r$ the unit radial vector field on $\C^n$. For $r>0$ we denote by $B_r$ the ball $\{|z|< r\}$ in $\C^n$, and denote $B_r^*=B_r\setminus\{0\}$; when $r=1$ we drop the subscript $r$. 
\item When we perform an integration we often omit the volume form in the formula, which is always to be understood as using the one associated to the obvious K\"ahler metric in the context. When we work over $B$ we will use the metric $\omega$ and when we work on the tangent cones we will use the metric $\omega_0$. 
\item For any open subset  $U\subset\C^n$ containing $0$, we shall denote by $\iota_U: U\setminus \{0\}\hookrightarrow U$ the obvious inclusion map, and $\pi_
U: U\setminus \{0\}\rightarrow \C\P^{n-1}$ the obvious projection map. When $U=\C^n$ we ignore the subscript and write $\iota=\iota_{\C^n}$,
and  $\pi=\pi_{\C^n}$.     
\end{itemize}

\begin{thm}\label{main}
Suppose $\mathcal E$ is a reflexive sheaf on $B$ with $0$ as an isolated singularity, such that 
\begin{itemize}
\item $\mathcal E$ is isomorphic to ${\iota_B}_*\pi_B^* \underline\E$ for some holomorphic vector bundle $\underline\E$ over $\C\P^{n-1}$;
\item $\Gr^{HNS}(\underline\E)$ is a reflexive sheaf.
\end{itemize}
  Then for any admissible Hermitian-Yang-Mills connection $A$ on $\E$ there is a unique tangent cone $A_\infty$ at $0$. More precisely,  the corresponding $\E_\infty$ is isomorphic to $\iota_*\pi^*\Gr^{HNS}(\underline\E)$, and $A_\infty$ is gauge equivalent to the natural Hermitian-Yang-Mills cone connection that is induced by the admissible Hermitian-Yang-Mills connection on $\Gr^{HNS}(\underline\E)$.
\end{thm}

\begin{rmk}
\begin{itemize}
\item The second assumption is due to a technical reason that will be explained in Section \ref{semistable} (see Remark \ref{rmk3-15}). We will remove this assumption and prove a stronger result in \cite{CS2}, which also includes the information on the bubbling set.
\item Under our assumption we also say $\E$ is \emph{homogeneous} near $0$, see Definition \ref{defi3-31}. We refer the readers to Section \ref{Tangent Cones} for the precise definition of a Hermitian-Yang-Mills cone. 
\item Roughly speaking, the theorem says that the tangent cone, a priori an analytically defined object, is indeed a purely algebro-geometric invariant of the reflexive sheaf $\mathcal E$. Viewed from another angle, we obtain the interesting fact that  the graded sheaf $\iota_*\pi^*(\Gr^{HNS}(\underline\E))$ can  be recovered from any admissible Hermitian-Yang-Mills connection on $\E$ which is homogeneous near $0$. 

\item When $\underline \E$ is the direct sum of polystable bundles, Theorem \ref{main} is also proved recently by Jacob-S\'a Earp-Walpuski \cite{JSW}, using pure PDE methods by working in a fixed holomorphic gauge. In this particular case we know $\E_\infty$ is isomorphic to $\E$, and \cite{JSW} furthermore proves polynomial rate of convergence to the tangent cone under the fixed holomorphic gauge. Our proof is based on a combination of PDE analysis and complex-geometric methods, and the main difficulty lies in the fact that in general we can not expect to obtain estimates in a fixed holomorphic gauge. The new technique allows us to prove a more general result when $\E_\infty$ is different from $\E$, and more importantly when $\E_\infty$ has singularities,  but the statement on the rate of convergence does not follow directly from our proof. In the general case it is also interesting to understand the rate of convergence, and we leave this for future study.
\end{itemize}
\end{rmk}

In general an isolated singularity of a reflexive sheaf is not necessarily homogeneous. Let $p: \widehat B\rightarrow B$ be the blow-up at the origin, then one can ask for extension of $p^*({\E)|_{\hat B\setminus \{0\}}}$ across the exceptional divisor $p^{-1}(0)\simeq \C\P^{n-1}$ over $\widehat B$ as a reflexive sheaf.  Given such an extension we denote by $\underline{\widehat \E}$ the restriction to $\C\P^{n-1}$, and we call $\underline{\widehat \E}$ an \emph{algebraic tangent cone} of $\E$ at $0$.\footnote{It would probably be better to call the corresponding sheaf $\iota_*\pi^*\underline{\widehat \E}$ on $\C^n$ the algebraic tangent cone. Our choice of notation makes the notations simpler in later discussion.}   In the homogeneous case above when $\E={\iota_B}_*\pi_B^*\underline \E$, it is easy to see that $\underline \E$ itself is an algebraic tangent cone. For general non-homogeneous singularities such an algebraic tangent cone serves as a homogeneous approximation of $\E$ itself, and is in general not unique. The next result says that if we can find an algebraic tangent cone $\underline{\widehat \E}$ which is locally free and stable, then the same result holds as in the homogeneous case, with $\underline\E$ replaced by $\underline{\widehat\E}$.   For a more detailed discussion and a conjectural picture in the general setting we refer the readers to Section \ref{Section3-3}.

\begin{thm} \label{thm2}
Suppose $\E$ is a reflexive sheaf on $B$ with isolated singularity at $0$, such that there is an algebraic tangent cone given by a stable vector bundle $\underline{\widehat \E}$. 
Then for any admissible Hermitian-Yang-Mills connection $A$ on $\E$ there is a unique tangent cone $A_\infty$ at $0$, and the corresponding $\E_\infty$ is isomorphic to $\iota_*\pi^*\underline{\widehat \E}$
, and $A_\infty$ is gauge equivalent to the natural Hermitian-Yang-Mills cone connection that is induced by the Hermitian-Yang-Mills connection on $\underline{\widehat \E}$.

\end{thm}

\

We finish the introduction with some discussion of the ideas involved in the proof of the above results. The key point is that in order to identify the limit Hermitian-Yang-Mills connection on the tangent cones, by a simple uniqueness theorem,  it suffices to first determine the underlying reflexive sheaf  which is a question of \emph{complex} geometric nature. In order to do this, we follow the basic principle that a reflexive sheaf can be understood by its holomorphic sections and we are lead to studying the behavior of holomorphic sections under rescaling. Motivated by the study on the algebraic structure of singularities of K\"ahler-Einstein metrics \cite{DS2012, DS2015, HS}, we prove a convexity result (see Proposition \ref{prop3-6}) 
for these holomorphic sections, which is a type of \emph{three-circle theorem}. The techniques developed in \cite{DS2015, HS} are robust and apply in greater generality, as long as one can obtain  a rigidity statement on the set of possible growth orders (which we can refer to as the \emph{spectrum})  of a homogeneous section on a tangent cone (which a priori may not  be unique). In our case this is  possible due to the fact that such growth orders are related to the slopes of reflexive sheaves on $\C\P^{n-1}$, which are always rational hence do not admit continuous deformations.

The convexity result implies in particular that any non-zero local holomorphic section $s$ of $\E$ always has a well-defined \emph{degree} $d(s)$ at $0$ (see (\ref{eqn3-1})) which is either a finite number in the spectrum, or is equal to $\infty$; moreover, if $d(s)<\infty$, then $s$ gives rise to homogeneous sections of degree $d(s)$ on \emph{all} the tangent cones.

 Roughly speaking, $d(s)$ measures the vanishing order of $s$ at $0$ with respect to the unknown Hermitian-Einstein metric. 
Notice a priori $d(s)$ can be $\infty$ in which case we would get the trivial zero section on the tangent cones, and then it would not provide useful information for us. One of the interesting aspects in the proof of our results lies in showing that $d(s)$ is always finite if $s$ is non-zero, and moreover, it can be a priori determined in terms of the slopes appeared in the successive quotients of the Harder-Narasimhan filtration of $\underline \E$ in the context of Theorem \ref{main}. This is quite different from the approach in  \cite{DS2015}, where one constructs local  holomorphic functions of finite degree by grafting holomorphic functions from the tangent cones and applying the \emph{H\"ormander construction}.  In our setting this construction does not work in any straightforward fashion, since we lack  H\"ormander's $L^2$ estimate in the absence of one-sided bound on the curvature of the Hermitian-Einstein metric; in fact, due to possible bubbling and the removable singularity theorem of Hermitian-Yang-Mills connections, one expects that in general not all the homogeneous sections on tangent cones may arise as limits. 

Instead, we go back to refine the PDE estimates for the Hermitian-Einstein equation on the original reflexive sheaf $\E$. The results of Bando-Siu \cite{BS} provide good control between two Hermitian-Einstein metrics, and with more delicate analysis (see Section 2.3)  to control $d(s)$ it suffices to construct a \emph{good} comparison Hermitian metric which is approximately Hermitian-Einstein in terms of the smallness of an $L^1$ integral on the mean curvature $\Lambda F$, and with respect to which one can understand the vanishing order of a holomorphic section $s$. This is easy in the case when $\underline \E$ is a direct sum of stable vector bundles since one can write down an explicit Hermitian-Einstein metric on $\E$ (roughly speaking, as a homogenous propagation of the corresponding Hermitian-Einstein metric on $\underline \E$),  and this observation has been used in \cite{JSW} to derive an $L^\infty$ bound, which has further lead to a precise decay rate of the convergence of the Hermitian-Yang-Mills connection to the tangent cone. 

In the general context of Theorem \ref{main} we make use of the important recent results of Jacob,  Sibley and Wentworth \cite{Jacob1, Jacob2, Sibley, SW} on the long time behavior of the Hermitian-Yang-Mills flow on compact K\"ahler manifolds. In the semistable case one obtains an approximately Hermitian-Einstein metric on $\underline \E$ which is good enough to tell us the degree $d(s)$. In the unstable case $\underline \E$ can only admit an approximately Hermitian-Yang-Mills connection which is not necessarily Hermitian-Einstein, in the sense that the mean curvature tensor is approximately block diagonal but not proportional to the identity matrix (the numbers appearing in the blocks are different ones given by the slopes of different pieces of the quotients from the Harder-Narasimhan filtration). One can try to compensate this deviation on $\E$ by making use of the fiber directions of the projection $\pi$. This solves the problem when the Harder-Narasimhan filtration is given by sub-bundles only; however, when the filtration has singularities, one can only perform this away from the singularities, and we need a more delicate choice of cut-off functions (see Lemma \ref{cut-off}). 

After understanding the degrees of holomorphic sections explicitly, one can start building non-trivial homomorphisms from the various subsheaves $\E_i$ appeared in the Harder-Narasimhan-Seshadri filtration of $\E$ to $\E_\infty$. The slope stability of the successive quotients in the Harder-Narasimhan-Seshadri filtration is then used to show that these homomorphisms must pass to isomorphisms between these quotients and direct summands of $\E_\infty$. An extra complication arises in that one also needs to deal with the \emph{Seshadri filtration} of a semistable sheaf. The latter is not a canonical object and we do not have an intrinsic characterization of it in terms of the Hermitian-Yang-Mills connection $A$. The difficulty is taken care of by refining the techniques in \cite{DS2015}. In our actual  proof of Theorem \ref{main} in Section \ref{section3-2}, to make the arguments clear we will treat cases of increasing generality,  see Sections 3.2.1-3.2.3.

A similar idea applies to prove Theorem \ref{thm2}. A good comparison metric is constructed out of the natural metric on the stable algebraic tangent cone $\underline{\widehat \E}$, using the natural relation between  $\E$ and $\underline{\widehat \E}$. The notion of algebraic tangent cone in our context does not seem to be well-known in the literature, so we give a slightly more detailed account of this. In the last subsection of this paper we discuss examples of reflexive sheaves which are not homogeneous. In particular, there are concrete non-trivial examples where Theorem \ref{thm2} can be applied.

\

\noindent \textbf{Acknowledgements:} We are grateful to Professor Jason Starr for helpful discussion concerning singularities of reflexive sheaves. We would like to thank Aleksander Doan, Simon Donaldson, Richard Thomas and Thomas Walpuski for useful comments. We thank the anonymous referee for helpful comments which greatly improves the presentation. This work was supported by a grant from the Simons Foundation (488633, S.S.).   S. S. is partially supported by an Alfred P. Sloan fellowship and NSF grant DMS-1708420.

\section{Preliminaries}

\subsection{Harder-Narasimhan-Seshadri filtration and canonical metrics}
In this section, we denote $X=\C\P^{n-1}$ endowed with the Fubini-Study metric $\omega_{FS}$ although the results apply to general compact K\"ahler manifolds. Recall a coherent sheaf $\F$ on $X$ is \emph{torsion free} if the natural map $\F\rightarrow \F^{**}$ is injective and \emph{reflexive} if the map is an isomorphism.  The singular set $\Sing(\F)$ is the set of points $x\in X$ where $\F_x$ is not free over $\O_{X,x}$. We know that $\Sing(\F)$ is always a complex analytic subset of $X$. It has complex co-dimension at least two if $\F$ is torsion free, and at least three if $\F$ is reflexive. A good nontrivial local example of a reflexive sheaf can be given by the sheaf $\iota_*\pi^*\underline \F$ on $\C^n$, where $\underline\F$ is a holomorphic vector bundle on $\C\P^{n-1}$, for example, the tangent bundle $\T_{\P^{n-1}}$ of $\C\P^{n-1}$.

The \emph{slope} of a coherent sheaf $\F$ is defined as 

\begin{equation}
\mu(\F):=\frac{2\pi  \int_X c_1(\F)\wedge\omega_{FS}^{n-2}}{\rk(\F)\int_{X}\omega_{FS}^{n-1}}\in \mathbb Q
\end{equation}
Here $c_1(\F)$ can be understood as the first Chern class of the \emph{determinant line bundle} of $\F$, which is always an integer,  and $\rk(\F)$ denotes the rank of $\F$.

\begin{defi}
A torsion free sheaf $\F$ is 
\begin{itemize}
\item \emph{semistable} if for all coherent subsheaves $\F'\subset \F$ with $rank(\F')>0$ we have $\mu(\F')\leq \mu(\F)$; 
\item  \emph{stable} if for all subsheaves $\F'\subset \F$ with $0<rank(\F')<rank(\F)$ we have $\mu(\F')< \mu(\F)$; 
\item \emph{polystable} if $\F$ is the direct sum of stable sheaves with equal slope;
\item \emph{unstable} if $\F$ is not semistable. 
\end{itemize}
\end{defi} 

The following definition is taken from Bando-Siu \cite{BS}
\begin{defi}
An \emph{admissible} Hermitian metric on $\F$ is a smooth Hermitian metric defined on $\F|_{X\setminus \Sing(\F)}$ such that the corresponding Chern connection $A$ satisfies $\int_{X\setminus \Sing(\F)}|F_A|^2\dVol_{\omega_{FS}}<\infty$,  and that $|\Lambda_{{\omega_{FS}}}F_A|$ is uniformly bounded on $X\setminus\Sing(\F)$; it is an \emph{admissible Hermitian-Einstein metric} if furthermore $\sqrt{-1}\Lambda_{\omega_{FS}}F_A=(n-1)\mu(\F)\Id$.
\end{defi}

By definition it follows that the Chern connection of an admissible Hermitian-Einstein metric is indeed an admissible Hermitian-Yang-Mills connection as defined in the introduction. Conversely, by \cite{BS} any admissible Hermitian-Yang-Mills connection on $\F$ defines a unique reflexive sheaf together with an admissible Hermitian-Einstein metric so that $A$ is the corresponding Chern connection. From now on, we will use the two terminologies interchangeably. We also drop the word ``admissible" when the meaning is clear from the context. 

The following theorem is proved by Donaldson and Uhlenbeck-Yau in the case of vector bundles, and later generalized by Bando-Siu to reflexive sheaves.

\begin{thm}[Donaldson-Uhlenbeck-Yau \cite{Donaldson1, Donaldson2, UY}, Bando-Siu \cite{BS}]\label{DUYBS}
A reflexive sheaf $\F$ over a compact K\"ahler manifold $(X, \omega)$ admits an admissible Hermitian-Einstein metric if and only if it is polystable.
\end{thm}

For later purpose we need the following

\begin{prop}[\cite{BS}, Proposition 3] \label{prop2-4}
Let $H$ be an admissible Hermitian-Einstein metric on a reflexive sheaf $\F$. Suppose $\mu(\F)\leq 0$, then any holomorphic section $s$ of $\F$ must be parallel with respect to the Chern connection. Furthermore, if $\mu(\F)<0$, then the only holomorphic section of $\F$ is the zero section. 
\end{prop}

This has a few consequences

\begin{cor} \label{cor2-5}
Let $\phi: \F_1\rightarrow\F_2$ be a non-trivial homomorphism between a stable reflexive sheaf $\F_1$ and a polystable reflexive sheaf $\F_2$ with $\mu(\F_1)=\mu(\F_2)$, then $\phi$ realizes $\F_1$ as a direct summand of $\F_2$.
\end{cor}
\begin{proof}
We view $\phi$ as a holomorphic section of $\mathcal{H}om(\F_1, \F_2)=\F_1^*\otimes \F_2$, which is a reflexive sheaf (see Proposition $5.32$ in \cite{Kobayashi}). By Theorem \ref{DUYBS} we know $\F_1$ and $\F_2$ admit Hermitian-Einstein metric, so we get an induced Hermitian-Einstein metric on $Hom(\F_1, \F_2)$ . On the other hand, $\mu(Hom(\F_1, \F_2))=\mu(\F_2)-\mu(\F_1)=0$. So by Proposition \ref{prop2-4} $\phi$ is parallel. In particular,  on the complement of $\Sing(\F_1)\cup \Sing(\F_2)$, $Ker(\phi)$ defines a parallel sub-bundle of $\F_1$, and $Im(\phi)$ defines a parallel sub-bundle of $\F_2$, and both $Ker(\phi)$ and $Im(\phi)$ admits induced Hermitian-Einstein metrics induced from $\F_1$ and $\F_2$ respectively. So by \cite{BS} they extend to polystable reflexive sheaves on $X$. By assumption we have $Ker(\phi)=0$, and $\F'=Im(\phi)$ is a direct summand of $\F_2$.  Hence $\phi: \F_1\rightarrow \F'$ is an isomorphism away from $\Sing(\F_1)\cup \Sing(\F')$, so extends as an isomorphism globally since $\F_1$ and $\F'$ are both normal. Indeed, by taking a locally free resolution of $\F_1^*\otimes \F'$ and taking its dual,  one obtains the sheaf exact sequence $0\rightarrow (\F')^*\otimes\F_1\rightarrow \G_1\rightarrow \G_2$ for locally free sheaves $\G_1$ and $\G_2$. $\phi^{-1}$ can  be naturally seen as a section of $\G_1$ away from $\Sing(\F_1)\cup\Sing(\F')$ which has complex codimension at least three, and it maps to zero in $\G_2$. So by the usual Hartogs's theorem $\phi^{-1}$ extends to a global section of $\G_1$ that maps to zero in $\G_2$, thus it defines a global homomorphism from $\F'$ to $\F_1$. Clearly it is the inverse of $\phi$. 
\end{proof}

\begin{cor} \label{cor2-6}
A stable reflexive  sheaf admits a unique Hermitian-Einstein metric up to constant rescalings. In general, any two Hermitian -Einstein metrics on a polystable reflexive sheaf determines the same Chern connection and the two metrics differ by a parallel complex gauge transform \footnote{Here a complex gauge transform means a complex linear isomorphism.} on the complement of singular set of the sheaf.
\end{cor}

\begin{proof}
Suppose $\F$ is polystable, and $H_1$ and $H_2$ are two Hermitian-Einstein metrics on $\F$, then by Proposition \ref{prop2-4} the identity map in $End(\F)$ is parallel, with respect to the Chern connection of the Hermtian metric $H_1^*\otimes H_2$. This implies that the Chern connections of $H_1$ and $H_2$ coincide.  Suppose $H_2(\cdot, \cdot)=H_1(g\cdot, \cdot)$ for a complex gauge transformation $g$ of $\F$ over $X\setminus\Sing(\F)$, then it follows that $g$ is holomorphic. Now by Proposition \ref{prop2-4} we conclude that $g$ is parallel with respect to the Hermitian-Einstein metric $H_1^*\otimes H_1$ on $End(\F)$. Hence it decomposes $\F$ into  the direct sum of eigenspace pieces, each of which is again polystable. If $\F$ is stable, then $g$ must be a multiple of identity. 
\end{proof}

\
Now we move on to discuss the case when $\F$ is not polystable. The following two results are well-known, see for example Page $174$ in \cite{Kobayashi}.

\begin{prop}
Suppose $\F$ is an unstable reflexive sheaf, then there is a \emph{unique} filtration by reflexive subsheaves
$$0=\F_0\subset \F_1\subset\cdots\subset \F_m=\F,$$
such that the successive quotient $Q_i:=\F_i/\F_{i-1}$ is torsion free and semistable, with $\mu(Q_{i+1})<\mu(Q_{i})$. 
\end{prop}

\begin{rmk}
The construction of \cite{Kobayashi} on Page 174 only states that $\F_i$ is torsion-free, but it is easy to see each $\F_i$ is indeed reflexive if $\F$ is reflexive.  By Proposition 5.22 in \cite{Kobayashi} a coherent subsheaf of a reflexive sheaf is reflexive if the corresponding quotient sheaf is torsion free. One then applies this fact inductively to $\F_i$ for $i=m, \cdots, 1$.  
\end{rmk}

The above filtration is called the \emph{Harder-Narasimhan filtration} of $\F$.  It follows that the associated graded object $\bigoplus_{i}Q_i$, which we denote by $\Gr^{HN}(\F)$, is also uniquely determined by $\F$. 

\begin{prop}
Suppose $Q$ is a semistable  torsion-free sheaf, then there is a filtration by  subsheaves
$$0=G_0\subset G_1\subset\cdots\subset G_q=Q, $$
so that the quotients $G_i/G_{i-1}$ are torsion free and stable, with $\mu(G_i/G_{i-1})=\mu(Q)$. 
\end{prop}

Such a filtration is usually referred to as a \emph{Seshadri filtration} of $Q$. Note that Seshadri filtration is in general \emph{not} unique; however, the associated graded sheaf $\bigoplus_{i} (S_i/S_{i-1})^{**}$ is nevertheless uniquely determined by $Q$. See \cite{BTT} for detailed discussion. 

Combining the above two results, given any reflexive sheaf $\F$, there is a double filtration by reflexive subsheaves
\begin{equation} \label{eqn2-2}
0=\F_0\subset \F_1\subset \cdots \F_m= \F
\end{equation}
and 
\begin{equation}
\F_{i-1}=\F_{i, 0}\subset \F_{i, 1}\subset \cdots \F_{i, q_i}=\F_{i}
\end{equation}
such that the successive quotients $\F_{i, j}/\F_{i, j-1}$ are torsion free and stable, and moreover the slope of these quotients is constant when $i$ is fixed, and strictly decreasing when $i$ increases. This is called the \emph{Harder-Narasimhan-Seshadri filtration} of $\F$, and we emphasize again that this filtration is not unique in general but the double dual of the associated graded object $(Gr^{HNS}(\F))^{**}$ is uniquely determined by $\F$. 

\

One can ask what is the analogue of a canonical Hermitian metric structure on a general $\F$. For semistable vector bundles on projective manifolds, the following is proved by Kobayashi \cite{Kobayashi}, using Hermitian-Yang-Mills flow. This is sufficient for our purpose, but we also mention that the result has been generalized to all compact K\"ahler manifolds by Jacob \cite{Jacob1}.

\begin{thm}[Kobayashi, Theorem $10.13$ in \cite{Kobayashi}] \label{thm2-10}
Suppose $\F$ is a semistable  vector bundle over $(X, {\omega_{FS}})$. Then $\F$ admits approximately Hermitian-Einstein metrics. Namely, for any $\epsilon>0$, there exists a Hermitian metric $H$ on $\F$ such that the associated Chern connection  $A$ satisfies
\begin{equation}
|\sqrt{-1}\Lambda_{{\omega_{FS}}}F_{A}-(n-1)\mu(\F)Id|_{L^\infty(X)}<\epsilon.
\end{equation}
\end{thm}

\ 

In the remainder of this section, we always assume $\F$ is locally free which may be unstable in general. This situation is more involved. Suppose the Harder-Narasimhan filtration of $\F$ is given as in  (\ref{eqn2-2}), and denote 
by $S(\F)$ the subset of $X$ where $Gr^{HNS}(\F)$ is not locally free. Given any Hermitian metric $H$ on $\F$, each $\F_{i}$ can be identified with a weakly holomorphic projection map $\pi_i \in W^{1,2}(\F^*\otimes \F)$ (see section $6$ in \cite{UY}), which is smoothly defined outside $S(\F)$ and satisfies the following:
\begin{itemize}
\item[(a).] $\pi_i=\pi_i^*=\pi_i^2$. This means that $\pi_i$ is a self-adjoint projection map. 
\item[(b).] $(\Id-\pi_i)\bp_{\F}\pi_i=0$. This condition is equivalent to $\F_i$ being a holomorphic sub-bundle outside $S(\F)$. 
\end{itemize}
In particular, $(\bp_{\F}\pi_i) \pi_i=0$, and taking adjoint we also have $\pi_i \p_\F \pi_i=0$, where $\partial_\F$ is the $(1, 0)$ component of the Chern connection on $End(\F)$ determined by the chosen metric $H$. Those simple observations will be used in the following calculations.

Now we define 
$$\psi^{H}=\sum_i (n-1)\mu_i (\pi_i-\pi_{i-1}), $$
 where $\mu_i=\mu(\F_i/\F_{i-1})$. Denote by $X_0=X\setminus \S(\F)$.  Then we have an orthogonal splitting over $X_0$ as
$$\F=\bigoplus_i Q_i, $$
where $Q_i:=\F_i/\F_{i-1}$ is naturally identified as a sub-bundle of $\F_i$, given by the orthogonal complement of $\F_{i-1}$ in $\F_i$. The splitting gives $\F$ another holomorphic structure $\bp_{S}$ outside $S(\F)$,  and this together with the fixed Hermitian metric defines a Chern connection which we deonte by $A_{(H,\bp_S)}$. 
\begin{rmk}\label{rmk2-11}
By definition, 
\begin{equation}\label{eqn2-15}
\bp_S=\sum_i (\pi_i-\pi_{i-1})\circ\bp_\F\circ (\pi_i-\pi_{i-1}).
\end{equation}
In particular, $\bp_\F=\bp_S-\sum_i(\pi_i-\pi_{i-1})\bp_\F \pi_i.$ 
\end{rmk}

Recall for each $i$, with respect to the orthogonal splitting $\F=\F_i\oplus \F_i^\perp$, the second fundamental form of $\F_i$ in $\F$ is a smooth section of $\Lambda_X^{1, 0}\otimes Hom(\F_i, \F_i^\perp)$ over $X_0$, whose adjoint is given by $\beta_i=-\pi_{i}\bp_\F\pi_i^{\perp}=\bp_\F \pi_i$, where $\pi_i^{\perp}$ denotes the projection map from $\F$ to $\F^{\perp}_i$. 

\begin{lem}\label{lem2-12}
The following estimates hold over $X_0$
\begin{enumerate}[(1)]
\item
\begin{equation}\label{eq2-5}
|\Lambda_{{\omega_{FS}}}\p_\F\beta_i|=|\Lambda_{{\omega_{FS}}}\p_\F\bp_\F\pi_i|\leq |\Lambda_{{\omega_{FS}}} F_{A_{(H,\bp_\F)}}+\sqrt{-1}\psi^H|+2|\beta_i|^2
\end{equation}
\item
\begin{equation}\label{sf}
|\p_\F\beta_i|\leq 2|\beta_i|^2+|F_{A_{(H,\bp_\F)}}|
\end{equation}
\end{enumerate}
\end{lem}
\begin{proof}
Since $\bp_\F=\bp_{S_i}-\beta_i$, where $\bp_{S_i}$ defines the split holomorphic structure with respect to the splitting $\F=\F_i\bigoplus \F_i^{\perp}$,  we have
\begin{equation*}
\begin{aligned}
F_{A_{(H,\bp_\F)}}
&=\p_\F\circ\bp_\F+\bp_\F\circ\p_\F\\
&=F_{A_{(H,\bp_{S_i})}}-\p_{S_i} \beta_i+\bp_{S_i} \beta^*_i-\beta_i\wedge\beta_i^*-\beta_i^*\wedge\beta_i. \\
\end{aligned}
\end{equation*}
Notice $\Lambda_{{\omega_{FS}}}\p_{S_i} \beta_i$ is a section of $Hom(\F_i, \F_i^{\perp})$ and therefore is perpendicular to the remaining terms, we have 
$$|\Lambda_{{\omega_{FS}}}\p_{S_i} \beta_i|\leq  |\Lambda_{{\omega_{FS}}} F_{A_{(H,\bp_\F)}}+\sqrt{-1}\psi^H|.
$$
Since $\p_\F=\p_{S_i}+\beta_i^*$, we have
$|\Lambda_{{\omega_{FS}}}\p_\F\beta_i|\leq |\Lambda_{{\omega_{FS}}} F_{A_{(H,\bp_\F)}}+\sqrt{-1}\psi^H|+2|\beta_i|^2.$
Similarly $|\p_{S_i}\beta_i|\leq |F_{A_{(H,\bp_\F)}}|$ and thus $|\p_\F\beta_i|\leq 2|\beta_i|^2+|F_{A_{(H,\bp_\F)}}|.$
\end{proof}

\

\begin{prop}\label{prop2-11}
There is a $K>0$ such that  for any $\epsilon>0$ and $\delta>0$, $\exists$ a smooth Hermitian metric $H$ on $\F$ such that the following holds for each $i$
\begin{enumerate}
\item $\sup_i\int_{X_0}|\beta_i|^2\leq\epsilon$;
\item $\sup_i\int_{X_0} |\Lambda_{\omega_{FS}}\p_\F\beta_i|\leq \epsilon$;
\item $\int_{X_0} |\sqrt{-1}\Lambda_{\omega_{FS}} F_{A_{(H, \bp_S)}}-\psi^{H}|
\leq \epsilon$;
\item $\|\Lambda_{{\omega_{FS}}} F_{A_{(H, \bp_{\F})}}\|_{L^{\infty}(X)}\leq K$;
\item $\sup_{X\setminus \S(\F)^{\delta}}(|\sqrt{-1}\Lambda_{\omega_{FS}} F_{A_{(H, \bp_S)}}-\psi^{H}|+\sup_i|\beta_i|+\sup_i|\Lambda_{{\omega_{FS}}}\p_\F\beta_i|)\leq\epsilon, $ 
where $S(\F)^{\delta}$ denotes the $\delta$-neighborhood of $S(\F)$. 
\item $\sup_{X\setminus \S(\F)^{\delta}} |F_{A_{(H, \bp_\F)}}|\leq K. $
\end{enumerate}
\end{prop}

\begin{proof}
This is indeed an easy consequence of \cite{Jacob2} and \cite{SW}. Starting from any initial Hermitian metric $H_0$ on $\F$, let $H_t$ be the family of Hermitian metrics on $\F$ evolving along the Hermitian-Yang-Mills flow.   We denote by $A_t$ the Chern connection of $H_t$ and $\pi^t_i$ the projection map determined by the Harder-Narasimhan filtration and the metric $H_t$. Then the following holds
\begin{itemize}

\item [(1).]
\begin{equation}\label{eqn2.10}
\lim_{t\rightarrow\infty}\int_{X_0} |\sqrt{-1}\Lambda_{\omega_{FS}} F_{A_t}-\psi^{H_t}|_{H_t}^2 = 0, 
\end{equation}
see Proposition 8 in \cite{Jacob2}.
\item[(2).]  \begin{equation}\label{eqn2.9}
 \lim_{t\rightarrow\infty} \int_{X_0} |\nabla_{A_t} \pi^t_i|_{H_t}^2=0, 
 \end{equation}
  see $(4.20)$ in \cite{Jacob2}.
\item[(3).] There exists a constant $K>0$ independent of $t$  such that 
\begin{equation}
|\sup_X \Lambda_{\omega_{FS}} F_{A_t}|_{H_t}\leq K
\end{equation}
see Lemma $(8.15)$ on page $220$ in \cite{Kobayashi}. 
\end{itemize}
We claim for $t$ large, $H_t$ satisfies the desired properties in the Proposition with the choice of $K$ as in the third item above. 
 In fact, let $\beta^t_i=\bp_\F\pi^t_i$. Then (\ref{eqn2.9}) implies 
 $$\lim_{t\rightarrow\infty} \int_{X_0} |\beta_i^t|^2=0.$$
 By Lemma \ref{lem2-12} and (\ref{eqn2.10}), we have
$$
\lim_{t\rightarrow\infty}\int_{X_0} |\Lambda_{{\omega_{FS}}}\p_\F \beta^t_i|=\lim_{t\rightarrow\infty}\int_{X_0} |\Lambda_{{\omega_{FS}}}\bp_\F (\beta^t_i)^*|=0.
$$
Then by Remark \ref{rmk2-11}, $\bp_\F=\bp_S -\sum_i (\pi^t_i-\pi^t_{i-1})\beta^t_i$, which implies 
$$
\begin{aligned}
\Lambda_{\omega_{FS}} F_{A_t}&=\Lambda_{\omega_{FS}} F_{(H_t, \bp_S)}-\sum_{i,j}\Lambda_{\omega_{FS}}  (\pi^t_i-\pi^t_{i-1})\beta^t_i\wedge (\beta^t_j)^*(\pi^t_j-\pi^t_{j-1})
\\
&-\sum_{i,j} \Lambda_{\omega_{FS}} (\beta^t_j)^*(\pi^t_j-\pi^t_{j-1})\wedge (\pi^t_i-\pi^t_{i-1})\beta^t_i
-
\sum_i \Lambda_{\omega_{FS}} \p_S [(\pi^t_i-\pi^t_{i-1})\beta^t_i]\\
&+\sum_i \Lambda_{\omega_{FS}} \bp_S [(\beta^t_i)^*(\pi^t_i-\pi^t_{i-1})].
\end{aligned}
$$
From this, we get 
$$
|\Lambda_{\omega_{FS}} F_{(H_t, \bp_S)}+\sqrt{-1}\psi^{H_t}| \leq |\Lambda_{\omega_{FS}} F_{A_t}+\sqrt{-1}\psi^{H_t}| +C(\sup_i|\beta^t_i|^2+\sup_i|\Lambda_{\omega_{FS}} \p_\E\beta^t_i|)
$$
Then we have
 $$\lim_{t\rightarrow\infty}\int_{X_0} |\Lambda_{\omega_{FS}} F_{(H_t, \bp_S)}+\sqrt{-1}\psi^{H_t}|=0. $$
The last two items follow from \cite{SW} (with a possibly larger constant $K>0$) where the analytic bubbling set of the Yang-Mills flow is identified with $\S(\F)$ as a set. More precisely, let $g_t=(H_0^{-1} H_t)^{\frac{1}{2}}$ and $\tilde{A}_t=g_t \cdot A_t \cdot g_t^{-1}$ which is integrable and unitary with respect to $H_0$, i.e. $(H_0, \tilde{A}_t)=(g_t^{-1})^*(H_t, A_t)$. Outside $S(\F)$, modulo unitary gauge transforms, $\tilde{A}_t$ converges smoothly to the natural admissible Hermitian-Yang-Mills connection $A_\infty$ on $(Gr^{HNS}(\F))^{**}$. Since the quantities we are interested are invariant under pulling back by $g_t^{-1}$, we get what we need.
\end{proof}

\begin{rmk}
 Although not needed in this paper, we mention that in \cite{Sibley} Sibley  proved the existence of $L^{p}$ approximate critical Hermitian metrics,  in the sense that for any $\delta>0$ and $1\leq p<\infty$ there exists a metric $H_{\delta}$ whose associated Chern connection $A_\delta$ satisfies
$$\|\sqrt{-1}\Lambda_{\omega_{FS}} F_{A_\delta}-\psi^{H_\delta}\|_{L^p(X)} <\delta.$$
\end{rmk}

\

Now consider a complex gauge transform away from $\S(\F)$ of the form
$$g=\sum_{i} f_i (\pi_i-\pi_{i-1}),  $$
where each $f_i$ is a smooth positive function. Denote $\beta=\bp_\F-\bp_S=-\sum_i(\pi_i-\pi_{i-1})\bp_\F\pi_i$.

\begin{lem}\label{complexgauge}
The following holds over $X_0$
\begin{equation*}
F_{(H, g\cdot\bp_\F)}=T_0 +T_1+T_2
\end{equation*}
with 
$$
\begin{aligned}
&T_0= F_{(H,g\cdot\bp_S)}=F_{(H, \bp_S)}-\sum_{i} \partial\bp\log (f_i^2) (\pi_i-\pi_{i-1});\\
&T_1=-(g\cdot\bp_S)(g \beta g^{-1})^*
+(g\cdot \bp_S)^*(g \beta g^{-1});\\
&T_2=-
g \beta g^{-1}\wedge(g \beta g^{-1})^*-(g \beta g^{-1})^*\wedge g \beta g^{-1}.
\end{aligned}
$$
Here 
$$
\begin{aligned}
(g\cdot\bp_S)(g \beta g^{-1})^*=&\sum_{i<j}\frac{f_i}{f_j} (\pi_j-\pi_{j-1})(\bp_S \beta^*)(\pi_i-\pi_{i-1})\\
&-2\sum_{i<j}\bp(\frac{f_i}{f_j})\wedge(\pi_j-\pi_{j-1})(\p_\F \pi_i) (\pi_i-\pi_{i-1}), 
\end{aligned}
$$
and $(g\cdot\bp_S)^*$ denotes the $(1,0)$ component of the Chern connection determined by $(H, g\cdot \bp_S)$.
\end{lem} 
\begin{proof}
By definition, 
$$ F_{(H,g\cdot\bp_S)}=(g\cdot\bp_S)^*\circ (g\cdot\bp_S)+(g\cdot\bp_S)\circ (g\cdot\bp_S)^*$$
Now the first part follows from this by plugging $g\cdot\bp_S=g\cdot\bp_\F-g\beta g^{-1}$. As for the second part,  
\begin{equation*}
\begin{aligned}
(g\cdot\bp_S)(g \beta g^{-1})^*
&=(\bp_S-\bp_Sg\cdot g^{-1})(g^{-1}\beta^* g)\\
&=\bp_S (g^{-1}\beta^* g)-[\bp_Sg\cdot g^{-1},g^{-1}\beta^* g]\\
&=g^{-1}\bp_S\beta^* g+2(\bp_S g^{-1})\wedge\beta^* g-2g^{-1}\beta^* \wedge\bp_S g
\end{aligned}
\end{equation*}
Plugging $\beta=-\sum_i(\pi_i-\pi_{i-1})\bp_\F\pi_i$ and using $\pi_{j}(\bp_S\beta^*)(\pi_i-\pi_{i-1})=0$ for $j\leq i$, we obtain the conclusion.
\end{proof}

\begin{cor}\label{splittingcurvature}
Over $X_0$, $|F_{(H,\bp_S)}| \leq |F_{(H,\bp_\F)}|+C\sup_i|\beta_i|^2$
\end{cor}
\begin{proof}
This follows by choosing $g=1$ in Lemma \ref{complexgauge} and applying Equation (\ref{sf}).
\end{proof}

We finish this subsection with a technical result which will be used in Section \ref{S3-5}. 

\begin{prop}\label{prop2-15}
Let $\G\subset\F$ be a saturated subsheaf and fix any smooth Hermitian metric $H$ on $\F$,  then there exists $\delta=\delta(\G)>0$ so that
$$\int_{X\setminus \Sing(\G)}|\bp\pi_\G|^{2+\delta}<\infty, $$
where $\pi_\G$ is the weakly holomorphic projection map defined by $\G$ with respect to $H$.
\end{prop}
\begin{proof}
By Hironaka resolution of singularities (see \cite{BS}),  there is a sequence of blow-ups $p_k: X_k\rightarrow X_{k-1}$ ($k=1, \cdots, N$) along smooth submanifolds of $X_{k-1}$ of codimension at least two, with $X_0=X$, such that $p=p_N\circ \cdots \circ p_1$ is  biholomorphic on the complement of $\Sing(\G)$, $p^{-1}(\Sing(\G))$ is a union $E=\bigcup E_j$ of simple normal crossing divisors (with possibly multiplicities),  and $p^*\G|_{X_N\setminus E} $ extends to a holomorphic sub-bundle of $p^*\F$. We denote the sub-bundle by $\widetilde\G$. Pulling back the given Hermitian metric on $\F$ to $p^*\F$, we obtain a corresponding smooth projection map $\pi_{\widetilde \G}$ defined by $\widetilde \G$. So 
$$\int_{X\setminus \Sing(\G)} |\bp \pi_\G|^{2+\delta}\omega_{FS}^{n-1}=\int_{X_N\setminus p^{-1}(\Sing(\G))} |\bp \pi_{\widetilde G}|_{p^*{\omega_{FS}}}^{2+\delta} \ p^*\omega_{FS}^{n-1}. $$
Let $\omega_k$ be a smooth K\"ahler metric on $X_k$, where $\omega_0=\omega_{FS}$. Then we can naturally view $\omega_k$ as a smooth real valued $(1, 1)$ form on $X_N$ which are K\"ahler metrics outside $E$. On $X_N\setminus E$ we have 
$$|\bp \pi_{\widetilde G}|^{2+\delta}_{\omega_{FS}}\leq (\Lambda_{p^*\omega_{FS}}\omega_N)^{1+\delta/2} |\bp\pi_{\widetilde G}|^{2+\delta}_{\omega_N}. $$
Notice 
$$ \Lambda_{p^*\omega_{FS}}\omega_N \frac{(p^*\omega_{FS})^{n-1}}{\omega_N^{n-1}}=\frac{(n-1)\omega_N\wedge (p^*\omega_{FS})^{n-2}}{\omega_N^{n-1}}$$
 is uniformly bounded, and $\bp\pi_{\widetilde G}$ is smooth on $X_N$. Therefore to prove the conclusion it suffices to show that we can find $\delta>0$ such that 
$$\int_{X_N\setminus E} (\Lambda_{p^*\omega_{FS}}\omega_N)^{\delta/2}\omega_N^{n-1}<\infty. $$
To prove this, we first notice on $X_N\setminus E$ we have
$$\Lambda_{p^*\omega_{FS}}\omega_N\leq \Pi_{k=1}^N (p_N\cdots p_{k+1})^*\Lambda_{p^*_{k}\omega_{k-1}}\omega_{k}.$$
 Now for each $k$, by fixing any smooth Hermitian metric on the corresponding line bundle associated to the exceptional divisor of $p_k$ and doing a local calculation, one can easily check that the $\Lambda_{p_{k}^*\omega_{k-1}}\omega_k \leq C|s_k|^{-2}$, where $s_k$ is the defining section for the exceptional divisor of $p_k$. Then we have
$$\Lambda_{p^*\omega_{FS}}\omega_N\leq C \Pi_{j} |\sigma_j|^{-2a_j}, $$
where $\sigma_j$ is the defining section of $E_j$ over $X_N$, and $a_j$ is a positive integer. 
Since $E$ is a union of simple normal crossing divisors, it is clear that we can find the desired $\delta>0$, again by estimating the corresponding integral locally near any point of $E$. 

\end{proof}

\subsection{Tangent cones}\label{Tangent Cones}
We first describe some generalities on Hermitian-Yang-Mills cones. Let $\underline A$ be an admissible Hermitian-Yang-Mills connection on $(\underline E, \underline H)$ over $(\C\P^{n-1}, \omega_{FS})$ with singular set $\underline \Sigma$. Let $\underline {\E}$ be the corresponding reflexive sheaf on $\C\P^{n-1}$, then the Einstein constant is $\mu=\mu(\underline\E)$. Recall we denote $\pi: \C^n\setminus\{0\}\rightarrow \C\P^{n-1}$ the natural holomorphic projection. Let $\Sigma=\pi^{-1}(\underline \Sigma)$. 
Define 
$$\E:=\iota_*\pi^*\underline \E.$$
Since $\pi$ is flat and $0$ is of codimension $n\geq 2$ in $\C^n$,  we know $\E$ is reflexive.  We consider the Hermitian metric $H:=|z|^{2\mu}\pi^*\underline H$ on $\C^n\setminus \Sigma$, and let $A$ be the corresponding Chern connection, then it follows that 
\begin{equation} \label{eqn2-1}
F_{A}=\pi^*F_{\underline A}+\sqrt{-1}\mu\pi^*\omega_{FS}\cdot \Id. 
\end{equation}

We first state a simple Lemma, whose proof follows easily from the fact that $(\C\P^{n-1}, \omega_{FS})$ is the symplectic reduction of $(\C^n, \omega_0)$ under the natural $S^1$ action.
\begin{lem} \label{lem2-16}
Let $\alpha$ be a two form on $\C\P^{n-1}$, then $\Lambda_{\omega_0}\pi^*\alpha=|z|^{-2}\pi^*(\Lambda_{\omega_{FS}}\alpha)$. 
\end{lem}
 It follows from Lemma \ref{lem2-16} and Equation (\ref{eqn2-1}) that $A$ is an admissible Hermitian-Yang-Mills connection on $\E$ (with respect to the K\"ahler metric $\omega_0$) with singular set $\Sigma$ and  vanishing Einstein constant. 
 
 \begin{defi}\label{defi2.19}
 We call such a Hermitian-Yang-Mills connection $(\E, H, A)$ a \emph{simple HYM cone}. When there is no confusion, sometimes we also use the notation $(\E, A)$ or simply $A$ for brevity.
 \end{defi}

We first note that different choices of $(\underline \E, \underline A)$ can give gauge equivalent simple Hermitian-Yang-Mills cones. For any $m\in \Z$, let $\underline A(m)$ be the Chern-connection on $\underline\E\otimes \O(m)$, where $\O(m)$ is endowed with the natural Hermitian metric whose Chern connection has curvature $-\sqrt{-1}m\omega_{FS}$, then it is easy to see that the Einstein constant of $\underline A(m)$ is $\mu_m=\mu+m$, and $\underline A(m)$ also gives rise to a simple HYM cone which is the same as the one given by $\underline A$. On the underlying sheaf, this is just the obvious fact that $\pi^*\O(m)$ is trivial on $\C^n\setminus\{0\}$, hence $\pi^*(\E\otimes \O(m))$ is isomorphic to $\pi^*\E$ and the metric then differs by a factor $|z|^{2m}$.

From the above discussion, in principle, given a simple HYM cone, there could be more than one way to modify the metric by a conformal factor so that the corresponding new Chern connection descends to be a connection over the projective space $\P^{n-1}$. One can ask whether the constructions above by tensoring with $\O(k)$ give all the possible ways to descend the modified connection. This is true indeed. As for this, notice there is a standard $S^1$ action on $\C^n$, given by $e^{i\theta}. z=e^{i\theta}z$. Parallel transport along the $S^1$ orbit determines a smooth section $P$ of the gauge group $\mathcal G$ of $A$. $P$ can be naturally viewed as a section of $End(\E)$, and using our definition it is easy to see that $P=e^{-2\pi \sqrt{-1}\mu}\Id$. It follows that $\mu$ is uniquely determined by $A$, modulo $\Z$. 
We fix such a choice of $\mu$. The conformal modification is to cancel the holonomy $P$ which implies that it has to be of the form $|z|^{2\mu+2k}$ for some $k$. As a result, once we have fixed $\mu$, the descending connection is also fixed. 

In sum, given two stable reflexive sheaves $\underline \E$ and $\underline \E'$, let $(\E, H, A)$ and $(\E', H', A')$ be the corresponding simple HYM cones constructed as in Definition \ref{defi2.19}, then we have
\begin{prop} \label{prop2-3}
$(\E, H, A)$ and $(\E', H', A')$ are gauge equivalent, i.e. there exists an holomorphic isomorphism $g: \E \rightarrow \E'$ so that $H=g^*H'$, if and only if $\underline \E$ and $\underline \E'$ differs by tensoring with a power of $\O(1)$. 
\end{prop}

For convenience we will simply call the matrix $e^{-2\pi\sqrt{-1}\mu}\Id$ the \emph{holonomy of $A$}. 
\begin{rmk}
If follows that $A$ is isomorphic to a pull-back connection from $\C\P^{n-1}$ if and only if the holonomy is trivial. In general $\mu$ does not have to be zero. For a simple example, we can take $\underline \E$ to be the tangent bundle of $\C\P^{n-1}(n\geq 3)$. It is well-known that $\underline\E$ is stable, with the obvious Hermitian-Einstein metric, and $\mu=\frac{n}{n-1}$. The corresponding simple HYM cone would have non-trivial holonomy and hence can not be a pull-back connection. 
\end{rmk}

\begin{defi}
A \emph{HYM} cone is a direct sum of simple HYM cones. 
\end{defi}

Similar to the above discussion, we can uniquely write a HYM cone $A$ as a direct sum of simple HYM cones $\bigoplus_{j}A_j$ such that each $A_j$ has distinct holonomy $e^{-2\pi\sqrt{-1}\mu_j}$.  We can similarly define the \emph{holonomy of $A$} as an element of $(S^1)^k\subset U(k)$, where $k=rank(E)$. It is uniquely determined by its eigenvalues (with multiplicities). The underlying sheaf $\E$ is also isomorphic to $\bigoplus_j \iota_*\pi^*\underline \E_j$ for reflexive sheaves $\underline \E_j$ over $\C\P^{n-1}$, with $\mu_j=\mu(\underline\E_j)$, and the corresponding Hermitian-Einstein metric on $\E$ can be written as $H=\bigoplus_j |z|^{2\mu_j} \pi^*\underline H_j$ for  Hermitian-Einstein metrics $\underline H_j$ on $\underline \E_j$. So it is clear that $\mu_j \in (k!)^{-1}\Z$ for all $j$.

\

Next we give an intrinsic characterization of a HYM cone. This is also observed in \cite{JSW}. 

\begin{thm} \label{thm2-6}
Let $A$ be an admissible Hermitian-Yang-Mills connection on $(\C^n, \omega_0)$ with vanishing Einstein constant and with singular set $\Sigma$, then $A$ is gauge equivalent to a HYM cone if and only if $\p_r\lrcorner F_A=0$ holds on $\C^n\setminus \Sigma$.
\end{thm}

\begin{proof}
The ``only if" direction follows easily from definition, so it suffices to prove the ``if" direction. Notice a priori we are not assuming $\Sigma$ is $\C^*$-invariant. We let $\Sigma'=\{\lambda. x|x\in \Sigma, \lambda\in \C^*\}$, then since $\Sigma$ is of complex codimension at least three,  $\Sigma'$ is of complex codimension at least two. We can use parallel transport along the $S^1$ orbit with respect to $A$ to define a smooth section $P$ of the gauge group $\mathcal G$ over $\C^n\setminus\Sigma'$.
We claim $P$ is covariantly constant, when viewed naturally as a section of $End(E)$. Notice this is a local property. To see this, we fix a point $z\in \C^n\setminus\{0\}$, and locally we can choose a trivialization of $E$ under which we can write $d_{A}=d+A_0$ for a $\mathfrak u(l)$-valued $1$-form $ A_0$, where $l$ is the rank of $E$. Modifying by an element of $\mathcal G$ we may assume locally around $z$ that $ A_0(\p_r)=A_0(J\p_r)=0$. Notice since $\p_r\lrcorner F_{A}=0$ and $F^{0, 2}_{A}=0$, we also have $(J\p_r)\lrcorner F_{A}=0$. Now we claim that $A_0$ is invariant under the local $\C^*$ action, so we can write $A=\pi^*\underline A$ for a locally defined unitary connection $\underline A$ on $\C\P^{n-1}$. Indeed, locally near $z$, we can choose a basis of vector fields $(e_1, \cdots e_{2n-2}, \p_r, J\p_r)$ so that $(e_1, \cdots, e_{2n-2})$ are $\C^*$ invariant. By assumption, we have $F_{A}(\p_r, e_i)=F_{A}(J\p_r, e_{2n-2})=0$, i.e. $\p_r(A_0(e_i))=J\p_r(A_0(e_i))=0$ for any $i$. In particular, $A_0(e_i)$ is $\C^*$ invariant. Since $e_i$'s  are all chosen to be $\C^*$ invariant, $A_0$ is $\C^*$ invariant. It then follows that parallel transport along the $\C^*$ action orbit commutes with the co-variant derivative $d_{A}$, hence propagating along the $S^1$ orbit we obtain $d_{A}P=0$.

Using $P$ we obtain a parallel splitting of $(E, A)$ over $\C^n\setminus \Sigma'$ into the direct sum of Hermitian-Yang-Mills connections. Since $\Sigma'$ has complex codimension at least two, by \cite{BS} each direct summand extends to an admissible Hermitian-Yang-Mills connection on $\C^n$ which defines a polystable stable reflexive sheaf respectively. Moreover, on each piece the holonomy $P$ is given by multiplication by $e^{-2\pi \sqrt{-1}\mu}$ for some $\mu$. From this, we know know as Proposition \ref{prop2-3} that each piece is indeed a direct sum of simple HYM cones. It also follows from the above argument that $\Sigma$ is indeed $\C^*$ invariant. 
\end{proof}

Now we apply the above discussion to our setting of taking tangent cones of admissible Hermitian-Yang-Mills connections. Let $(\E, H, A)$ be an admissible Hermitian-Yang-Mills connection on $(B, \omega)$ with isolated singularity at the origin. For  $\lambda>0$, we denote by $A_\lambda$ the pull-back of $A$ to $B_{\lambda^{-1}}\setminus\{0\}$ by the map $z\mapsto \lambda z$. Clearly $A_\lambda$ is Hermtian-Yang-Mills with respect to the metric $\lambda^{-2}\cdot \lambda^*\omega$. 

Using the monotonicity formula of Price  (Proposition $5.1.1.$ in \cite{Tian}), i.e. $e^{ar^2} r^{4-2n} \int_{B_r} |F_A|^2$ is decreasing as $r\rightarrow 0$ for some constant $a$, we know that there exists a fixed constant $C$ so that 
$$
\int_{B_{\lambda^{-1}}} |F_{A_\lambda}|^2 \leq C<\infty.
$$ 
Now over any fixed compact subset of $\C^n$, $A_\lambda$ has uniformly bounded $L^2$ norm of the curvature. Since over $B^*$, a fixed unitary vector bundle is always isomorphic to the pulled-back of a unitary vector bundle over $S^{2n-1}$, we can assume that the underlying unitary bundle is scaling invariant.  By using Uhlenbeck compactness theorem (see theorem $5.2$ in \cite{UY} and Theorem 2.2.1 in \cite{Tian} ) applied on an exhaution of $\C^n$ by any compact subsets  \footnote{Here, the limit is taken on an exhaustion given by compact subsets of $\C^n \setminus \tilde{\Sigma}$ first and then do a diagonaliozation argument to get the limit over $\C^n$.}, it follows that by passing to a subsequence $\{\lambda_{i}\}\rightarrow 0$, we may find a closed subset $\tilde\Sigma\subset \C^n$ which contains $0$ (the \emph{bubbling set}) and has  finite codimension 4 Hausdorff measure, a sequence of smooth gauge transformations $g_i$ over $B_{\lambda_i^{-1}}\setminus \tilde\Sigma$,  a smooth Hermitian-Yang-Mills connection $A_\infty$ on $(\E_\infty, H_\infty)$ over $(\C^n\setminus \tilde\Sigma, \omega_0)$ with vanishing Einstein constant and smooth unitary isomorphisms $P: (\lambda_i^* \E, \lambda_i^*H)\rightarrow (\E_\infty, H_\infty)$ over $B_{\lambda_i^{-1}} \setminus \tilde{\Sigma}$, such that $(P^{-1})^*g_i(A_{\lambda_i})$ converges smoothly to $A_\infty$ on any compact subset $K\subset \C^n\setminus \tilde\Sigma$ (\cite{Tian}, Proposition 3.1.2). 

Furthermore, for the limit connection $A_\infty$, on any fixed compact subset $K$ of $\C^n\setminus \tilde\Sigma$, there is a constant $C>0$ such that  over any compact set $K\subset \C^n$
$$\int_{K\setminus\tilde\Sigma}|F_{A_\infty}|^2\leq C,$$ 
hence by the removable singularity theorem of Bando-Siu \cite{BS}, $A_\infty$ is indeed an admissible Hermitian-Yang-Mills connection which defines a reflexive sheaf $\mathcal E_\infty$. By \cite{BS}, $H_\infty$ can be extended over $\C^n \setminus Sing(\E_\infty)$. Abusing notation, we still denote it as $H_\infty$. Notice that a priori we may obtain different limits $(\E_\infty, A_\infty)$ along different subsequences.

\begin{defi} 
Any such limit $(\E_\infty, A_\infty)$ is called a \emph{tangent cone} of $A$ at $0$.
\end{defi}

Using the monotonicity formula again Tian (\cite{Tian} Lemma 5.3.1) proved that $\p_r\lrcorner F_{A_\infty}=0$ so by Theorem \ref{thm2-6} we obtain the fact 

\begin{prop}
Any tangent cone is a HYM cone. 
\end{prop}

\begin{rmk}
In \cite{Tian}, Page 263 it is claimed that any tangent cone is always the pull-back of a Hermitian-Yang-Mills connection on $\C\P^2$. This is \emph{not} precisely correct, since a HYM cone is not necessarily simple. Compare also Proposition 6.1 in \cite{JSW}. 
\end{rmk}

\subsection{PDE estimates}\label{PDE estimate}

 Let $\overline B^*=\overline B\setminus\{0\}\subset\C^n$. Again we always assume $n\geq 3$. For a function $g$, we denote $g^+=\max\{g,0\}$. The following Lemma is crucial for us. Here $\Delta=\Delta_{\omega}$ is the analyst's Laplacian operator.
 
\begin{lem}\label{lem2-19}
Suppose $g\in C^{2}(B^{*})\cap C^0(\overline B^*)$ with $\int_{ B^*}|g^+|^{\frac{n}{n-1}}<\infty$ and $f$ is a non-negative function on $B^*$. If on $B^*$ we have
\begin{equation} \label{eqn2-16}
\Delta g(z) \geq -{|z|^{-2}}{f(z)}, 
\end{equation}
 then for all $z\in B^*$ the following hold,
 \begin{enumerate}[(1).]
\item For all $z\in B^*$,  $$g(z)\leq |g|_{L^{\infty}(\p B)}+\int_{B^*} G(z,w) |w|^{-2} f(w)dw, $$
 where $G(z, w)$ is the (positive) Green's function for $-\Delta$ on $B$. The inequality is only meaningful when the right hand side is finite. 
\item For all $z\in B^*$, 
\begin{equation*}
g(z)\leq C_0(|g|_{L^\infty(\p B)}+\sup_{|w-z|\leq |z|/2}|f(w)|+(-\log |z|) \sup_{r\in (0, 1]} r^{1-2n}\int_{\p B_r}|f|), 
\end{equation*} 
where $C_0$ depends only on $n$.
\item If $|f|_{L^\infty(B^*)}<\infty$, then
$$\limsup_{z\rightarrow 0} \frac{g(z)}{-\log |z|}\leq C_0\limsup_{r\rightarrow 0} r^{1-2n}\int_{\p B_r}|f|.$$
\end{enumerate}
\end{lem}

\begin{proof}
We first solve the Dirichlet problem $\Delta h=0$, $h|_{\p B}=g|_{\p B}$, then $|h|_{L^\infty(\overline B)}\leq |g|_{L^\infty(\p B)}$. So we can reduce to the case that $g|_{\p B}=0$.
Fix any $z\in B^*$, for $\epsilon<|z|/4$, we choose a cut-off function $\chi_\epsilon$ supported in $B\setminus B_\epsilon$, and equal to $1$ on $B\setminus B_{2\epsilon}$, with $|\nabla \chi_\epsilon|\leq C\epsilon^{-1}$ and $|\nabla^2 \chi_\epsilon|\leq C\epsilon^{-2}$. For $\tau>0$ we denote $$g_\tau=\frac{1}{2}(\sqrt{g^2+\tau^2}-\tau+g),$$ then $g_\tau=0$ on $\p B$ and one can check 
$$\Delta g_\tau(z)\geq \frac{1}{2}\frac{\sqrt{g^2+\tau^2}+g}{\sqrt{g^2+\tau^2}}\Delta g\geq -\frac{1}{2}\frac{\sqrt{g^2+\tau^2}+g}{\sqrt{g^2+\tau^2}}|z|^{-2}f(z)\geq -|z|^{-2}f(z). $$
Using  Green's representation formula we have
\begin{eqnarray*}
g_\tau(z)=\int_{B^*} 2\nabla_w G(z,w) \nabla \chi_{\epsilon}(w)g_\tau(w)+G(z,w)\Delta\chi_\epsilon(w) g_\tau(w) -G(z,w) \chi_{\epsilon}(w)\Delta  g_\tau(w)\\
\end{eqnarray*}
Let $\tau\rightarrow 0$, we get 
$$g^{+}(z)\leq \int_{B^*} 2\nabla_w G(z,w) \nabla \chi_{\epsilon}(w)g^+(w)+G(z,w)\Delta\chi_\epsilon(w) g^+(w) +G(z,w) \chi_{\epsilon}(w) |w|^{-2}f(w).$$
Now let $\epsilon \rightarrow 0$, we claim the first two terms tend to zero. We only prove this for the second term and the first term can be dealt with similarly.  We use the well-known fact that 
$|G(z, w)|\leq C|z-w|^{2-2n}$ and $|\nabla G(z, w)|\leq C|z-w|^{1-2n}$ for $|z|\leq 1/2, |w|\leq 1/2$(see for example Theorem 4.13 in \cite{Aubin}).

We have 
\begin{eqnarray*}
&&|\int_{B^*}G(z, w)\Delta \chi_\epsilon(w)g^{+}(w)|\\
&\leq& C\epsilon^{-2} \int_{B_{2\epsilon}\setminus B_\epsilon}|z-w|^{-2n+2} |g^{+}(w)|\\
&\leq & C\epsilon^{-2}Vol(B_{2\epsilon})^{1/n}|z|^{-2n+2}(\int_{B_{2\epsilon}\setminus B_\epsilon} |g^{+}(w)|^{\frac{n}{n-1}})^{\frac{n-1}{n}}, 
\end{eqnarray*}
and the last term tends to zero, since $\int_{ B^*}|g^+|^{\frac{n}{n-1}}<\infty$. So we obtain
$$
g^{+}(z)\leq \int_{B} G(z,w) {|w|^{-2}}{f(w)} dw.
$$
This finishes the proof of (1).

\

Now we prove (2).  It suffices to estimate the integral $\int_B |z-w|^{2-2n}|w|^{-2}f(w)$. We divide this into two parts. 
When $|w-z|\leq |z|/2$ we have 
\begin{equation}\label{eqn2-17}
\int_{|w-z|\leq |z|/2} |z-w|^{2-2n}|w|^{-2}f(w)\leq  C\sup_{|w-z|\leq |z|/2} f(w).
\end{equation}
When $|w-z|\geq |z|/2$, we have 
$|w|\leq 3|z-w|$. Then
$$
\begin{aligned}
\int_{|w|\geq \frac{|z|}{2}, |z-w|\geq \frac{|z|}{2}} |z-w|^{2-2n}|w|^{-2}f(w)&\leq C\int_{|w|\geq \frac{|z|}{2}} |w|^{-2n}f(w)\\
&\leq C(-\log |z|) \sup_{r\in (0, 1]} r^{1-2n}\int_{\p B_r}|f| 
\end{aligned}
$$
We also have 
$$
\begin{aligned}
\int_{|w| \leq |z|/2} |z-w|^{2-2n}|w|^{-2} f(w)
&\leq C |z|^{2-2n} \int_{|w|\leq |z|/2}|w|^{-2}f(w)\\
&\leq C \sup_{r\in (0, 1]} r^{1-2n}\int_{\p B_r}|f|.
\end{aligned}
$$
Combining the estimates above, we easily get the conclusion.

Now to see (3) we first take $z\rightarrow 0$ in (2) to get 
$$\limsup_{z\rightarrow 0} \frac{g(z)}{-\log |z|}\leq C_0\sup_{r\in (0, 1]} r^{1-2n}\int_{\p B_r}|f|.$$
Now we notice \eqref{eqn2-16} is invariant under rescaling, so the same estimate holds on the ball $B_s$ for all $s\in (0, 1]$. So we can take the $\limsup$ on the right hand side of the above equation. 
\end{proof}

Let $H$ be an admissible Hermitian metric on a reflexive sheaf $\E$ defined on $\overline B$.  

\begin{thm}[\cite{BS}] \label{thm2-24}
For any holomorphic section $s$ of $\E$,  we have $\log^{+}|s|^2$ belongs to $H^1_{loc}$, and the following inequality holds  in weak sense
\begin{equation} \label{eqn2-4}
\Delta \log^{+}|s|^2\geq -2\frac{\langle \Lambda_{\omega}Fs, s\rangle}{|s|^2}\geq -2|\Lambda_{\omega}F|.
\end{equation}
\end{thm}
In particular, by Moser iteration,  $|s|\in L^\infty_{loc}$. Moreover, if $|s|$ is in $L^2(B)$, then we have 
\begin{equation} \label{eqn2-5}
|s|_{L^\infty(B_{1/2})}\leq K_1 |s|_{L^2(B)}
\end{equation}
for $K_1$ depending on $|\Lambda_{\omega}F|_{L^\infty(B)}$.

\

Now suppose $\E$ has an isolated singularity at $0$. let $H$ and $H'$ be two admissible Hermitian metrics on $\E$, then $\Tr_H(H')$ and $\Tr_{H'}(H)$ are both the norms of the identity section of $End(\E)$ with respect to the two admissible Hermitian metrics $H^*\otimes H'$ and $(H')^*\otimes H$ respectively. Applying (\ref{eqn2-4}) and Lemma \ref{lem2-19},
 we obtain 
\begin{eqnarray} \label{cor2-26}
&&\limsup_{z\rightarrow 0} \frac{|\log \Tr_H(H')(z)|+|\log \Tr_{H'}(H)(z)|}{-\log |z|}\nonumber \\&\leq& C_0\limsup_{r\rightarrow 0} r^{1-2n}\int_{\p B_r}r^2(|\Lambda_{\omega}F_H|+|\Lambda_{\omega}F_{H'}|).
\end{eqnarray}

Notice the left hand side bounds the ratio between the metrics  $H$ and $H'$. Indeed, we have 
$$H'\leq e^{|\log \Tr_{H} H'|} H, H\leq e^{|\log \Tr_{H'}H|}H'.$$ This can be seen by diagonalizing $H'$ ( resp. $H$) using $H$ (resp. $H'$) respectively and a direct comparison in terms of eignevalues. In particular, if both $H$ and $H'$ are Hermitian-Einstein with vanishing Einstein constant, then there is a constant $C>0$ such that $C^{-1}H'\leq H\leq CH'$. This has been observed in \cite{JSW}. 

\

For our later purposes we also need to deal with more general classes of Hermitian metrics which may not be admissible. The following Lemma makes it convenient to use Lemma \ref{lem2-19}. 

\begin{lem} \label{lem2-30}
Suppose $\E$ has an isolated singularity at $0$, and  $H$, $H'$ be two smooth Hermitian metrics on $\E|_{\overline B^*}$ such that for some $\delta\in (0, 1]$, 
$$|F_H|+|F_{H'}|\in L^{1+\delta}(B).$$ 
Then 
$$g^{+}\in L^{\frac{n}{n-1}(1+\delta)}(B)$$
where $g$ denotes either $\log Tr_{H}H'$ or $\log Tr_{H'}H$. In particular, by Lemma \ref{lem2-19}, if we further assume $r^2(|\Lambda_\omega F_H|+|\Lambda_\omega F_{H'}|)\in L^\infty(B^*)$, then (\ref{cor2-26}) continues to hold. 
\end{lem}

\begin{proof}
The argument essentially follows from the proof of Theorem 2 in \cite{BS}.  Fix any complex subspace $V\subset \C^n$ of dimension $n-2$, and denote by $p: B\rightarrow B\cap V$ the orthogonal projection.  Let $\chi: \C^2\rightarrow[0, 1]$ be a cut-off function which is equal to 1 for $|z|\leq 1/100$ and equal to zero for $|z|\geq 2/100$. For each $t\in V$ with $0<|t|\leq 1/2$, $\chi$ defines a natural cut-off function on $p^{-1}(t)$. Since $p^{-1}(t)$ is a complex subspace, and $\E$ is a holomorphic vector bundle over $p^{-1}(t)$, we can apply the above discussion to $p^{-1}(t)$ and obtain 
$$\Delta_t g\geq -C (|F_{H}|+|F_{H'}|), $$
where $\Delta_t$ is the Laplacian operator on $p^{-1}(t)$. Multiplying both sides by $\chi^2(g^+)^{\delta}$, and integrating by parts on $p^{-1}(t)$ we obtain ($\nabla^t$ denotes the derivative on $p^{-1}(t)$)
\begin{eqnarray*}
&&\int_{p^{-1}(t)} |\nabla^t (\chi (g^+)^{\frac{\delta+1}{2}})|^2
\\&\leq& C(\int_{p^{-1}(t)} \chi^2 (g^+)^\delta (|F_H|+|F_{H'}|)+\int_{p^{-1}(t)} |\nabla^t \chi|^2 (g^{+})^{1+\delta}
\\&&+\int_{p^{-1}(t)} \chi |\nabla^t\chi|(g^+)^{\delta}|\nabla^t g^{+}|), 
\end{eqnarray*}
where the constant $C$ depends on $\delta$. 
Notice $\nabla^t\chi$ is supported outside the ball $|z|\leq 1/100$, and $H$ and $H'$ are both smooth away from zero, so the last two terms are uniformly bounded independent of $t$. For the first term on the right hand side we can use Young's inequality, and obtain for any $\epsilon>0$, a number $C(\epsilon)>0$ such that 
\begin{eqnarray*}
&&\int_{p^{-1}(t)} |\nabla^t (\chi (g^+)^{\frac{\delta+1}{2}})|^2
\\&\leq& \epsilon \int_{p^{-1}(t)} \chi^2 (g^+)^{\delta+1}+C(\epsilon) \int_{p^{-1}(t)}\chi^2 (|F_H|+|F_{H'}|)^{\delta+1}+C
\end{eqnarray*}
Using the Poincar\'e inequality on the unit ball in $\C^2$, and choosing $\epsilon$ sufficiently small, we conclude that 
$$\int_{p^{-1}(t)} |\nabla^t (\chi (g^+)^{\frac{\delta+1}{2}})|^2+\int_{p^{-1}(t)} \chi^2 (g^+)^{\delta+1}\leq C(\epsilon) \int_{p^{-1}(t)}\chi^2 (|F_H|+|F_{H'}|)^{\delta+1}+C. $$
Integrating this along $V$, and noticing that the inequality is uniform for all choices of complex subspaces $V$, one sees that $(g^{+})^{\frac{\delta+1}{2}}\in W^{1, 2}(B)$. Then by Sobolev embedding theorem, we get $g^+\in L^{\frac{n}{n-1}(\delta+1)}(B)$. 
\end{proof}

\section{Proof of the main results}

\subsection{Convexity for holomorphic sections}\label{Convexity for holomorphic sections}
We first assume $(\E, H, A)$ is a HYM cone on $\C^n$.  Let $\Sigma$ be the  singular set of $\E$, and $\underline \Sigma$ be the corresponding singular set of $\underline \E$ in $\C\P^{n-1}$. In the following, we will use $\nabla$ to denote the covariant derivative determined by $A$.

\begin{defi}
A holomorphic section $s$ of $\E$ over the ball $B$ is called \emph{homogeneous} if  $s$ satisfies $\nabla_{\p_r}s=dr^{-1}s$  on the locally free part of $\E$, for some constant $d$. We call $d$ the \emph{degree} of $s$. 
\end{defi}

It follows from the definition directly that for a homogeneous sections $s$ of degree $d$, $\nabla_{Jr\p_r}s=\sqrt{-1} d s$ 
since $\nabla_{\p_r}s=\mu r^{-1}s$. In particular, for fixed $z$ and any $t\in \C^*$, one has
\begin{equation}\label{eqn3.1}
|s|^2(tz)=|t|^{2d} |s|^2(z)
\end{equation}
which can be seen by taking derivatives with respect to $t$ for both sides. The following lemma will be used later. 

\begin{lem}\label{lem3-2}
If $s$ is a homogeneous holomorphic section of $\E$ with degree $\mu$, then there is a \emph{simple} HYM cone $\E'=\iota_*\pi^*\underline{\E}'$  which is a direct summand of $\E$, so that $s=\iota_*\pi^*\underline s$ for some holomorphic section $\underline s$ of $\underline{\E}'$ on $\C\P^{n-1}$. Moreover, the holonomy of $\E'$ is $e^{-2\pi\sqrt{-1}\mu}\Id$, and the Einstein constant of $\underline{\E}'$ is equal to $\mu$. 
\end{lem}
\begin{proof}
Since $\nabla_{\p_r}s=\mu r^{-1}s$, we have $\nabla_{Jr\p_r}s=\sqrt{-1}\mu s$. For any $z\in \C^n\setminus\Sigma$ such that $s(z)\neq 0$, it follows that $s(z)$ is an eigenvector of the holonomy $P$ with eigenvalue $e^{-2\pi\sqrt{-1}\mu}$. So $s$ is a section of a simply HYM cone direct summand $\E'$ of $\E$. Then using the discussion of Section 2.2 we can show that $\E'$ is isomorphic to the pull-back of a reflexive sheaf $\underline{\E}'$ on $\C\P^{n-1}$ with Einstein constant $\mu$, and $s$ descends to a holomorphic section $\underline s$ of $\underline{\E}'$. More precisely, we multiply the original metric with the conformal factor $|z|^{-2\mu}$ and with respect to the new connection, we know that $s$ is parallel along the natural $S^1$ orbits, thus descending to be a section of $\underline \E'$ over the projective space.
\end{proof}  

 Let $\Gamma$ be the set of all possible degrees of  non-zero homogeneous sections of $\E$. Then by Lemma \ref{lem3-2}, we know
 \begin{equation} \label{rigidity}
 \Gamma\subset (\text{rank}(\E)!)^{-1}\Z.
 \end{equation}
  Also $\Gamma$ is bounded below since for all $i$,  $H^0(\C\P^{n-1}, \underline\E_i(m))=0$ if $m$ sufficiently negative.

Notice by Theorem \ref{thm2-24}, a holomorphic section $s$ of $\E$ satisfies that $|s|^2$ is locally bounded across $\Sigma$. Since $\Sigma$ has locally finite codimension four Hausdorff measure, it  does not contribute to the calculation of  the $L^2$ norm of $|s|$. In the following,  all the integrals can be understood as integrating on the complement of $\Sigma$. 

\begin{lem}
Suppose $s_1$ and $s_2$ are two homogeneous holomorphic sections of $\E$ with degree $d_1$, $d_2$ respectively. If $d_1\neq d_2$, then for any $S^1$ invariant subset $Z$ of $B$, 
$$\int_{Z}\langle s_1, s_2\rangle =0.$$
\end{lem}

\begin{proof}
Since $s_i$ is homogeneous, we have $\nabla_{Jr\p_r}s_i=\sqrt{-1}d_is_i$. Hence along the $S^1$ orbit of any $z\in B\setminus\Sigma$, we have
$$\frac{d}{d\theta}\langle s_1, s_2\rangle=\sqrt{-1}(d_1-d_2)\langle s_1, s_2\rangle.$$
So the integral of $\langle s_1, s_2\rangle$ over any $S^1$ orbit is zero. By Fubini theorem we obtain the conclusion. 
\end{proof}

\begin{lem} \label{lem3-4}
Given a holomorphic section $s$ of $\E$ over $B$ with $\int_{B}|s|^2<\infty$, we have an orthogonal decomposition over $B\setminus \Sigma$
$$s=\sum_{d\in \Gamma}s_d, $$
where each $s_d$ is homogeneous of degree $d$, and the convergence is understood as in $L^2(B\setminus\Sigma)$ and $C^\infty_{loc}(B\setminus\Sigma)$. 
\end{lem}

\begin{proof}
First of all suppose $(\E, A)$ is the direct sum of simple HYM cones $(\E_j, A_j)$, then on $B\setminus \Sigma$ we can naturally write $s=\sum s_j$, where $s_j$ is a holomorphic section of $\E_j|_{B\setminus\Sigma}$. The normality of $\E_j$ implies $s_j$ is indeed a holomorphic section of $\E_j$. Clearly, this is also an $L^2$ orthogonal decomposition. Therefore, without loss of generality we may assume $A$ is a simple HYM cone. 

Suppose $\E=\pi^*\underline\E$ and $H=|z|^{2\mu}\pi^*\underline H$.  Then locally choose a small open set $U\subset\C\P^{n-1}$ such that $\underline \E|_U$ is free and admits a trivialization by holomorphic sections $\underline s_j(j=1, \cdots, m:= \text{rank}(\E))$. On $\pi^{-1}(U)$, we can write $s=\sum_{j=1}^m f_j(z) \pi^*\underline s_j$ for some holomorphic functions $f_j$ on $\pi^{-1}(U)$. We can perform Laurent series expansion along the fibers of $\pi$ and write $f_j=\sum f_{j, e}$, where each $f_{j, e}$ is homogeneous of degree $e$ under the $\C^*$ action. So on $\pi^{-1}(U)$ we have an expansion $s=\sum s_d$ into direct sum of homogeneous sections, which is $L^2$ orthogonal over $\pi^{-1}(V)$ for any $V\subset U$. In particular, such an expansion is independent of the choice of the local trivialization $\{\underline s_j\}$. This implies each $s_d$ is indeed a global holomorphic section on $B\setminus\Sigma$, which also extends to the entire $B$ by Hartogs's theorem. 

\end{proof}

\begin{prop}\label{prop3-5}
Given a holomorphic section $s$ of $\E$ over $B$ with $\int_{B}|s|^2<\infty$, we have
$$\int_{B_{1/4}}|s|^2 \cdot \int_{B}|s|^2\geq (\int_{B_{1/2}}|s|^2)^2. $$
Moreover, the equality holds if and only if $s$ is homogeneous.

\end{prop}

\begin{proof}
We simply apply the above Lemma and write $s=\sum_{d\in\Gamma}s_d$. Notice for each $d$, by changing variable and using Equation \ref{eqn3.1}
$$\int_{B_r}|s_d|^2=r^{2d+2n}\int_{B}|s_d|^2. $$
Then the conclusion follows from the Cauchy-Schwarz inequality.
\end{proof}

Now we move on to the general case when $(\E, A)$ is an admissible Hermitian-Yang-Mills connection on $B$ with isolated singularity at $0$. The above discussion can be applied to all the tangent cones $(\E_\infty, A_\infty)$ of $A$.

\begin{prop} \label{prop3-6}
For any $\lambda\notin (\text{rank}(\E)!)^{-1}\Z$, we can find $j_0=j_0(\lambda)\in \Z_{\geq0}$ such that for all $j\geq j_0$, if a holomorphic section $s$ of $\E$ defined on $B_{2^{-j}}$ satisfies
$$\int_{B_{2^{-j-1}}} |s|^2 \geq  2^{-2\lambda-2n} \int_{B_{2^{-j}}} |s|^2 , $$
then 
$$\int_{B_{2^{-j-2}}} |s|^2 > 2^{ -2\lambda-2n} \int_{B_{2^{-j-1}}} |s|^2 . $$
\end{prop}

\begin{proof}
Otherwise there exists $\lambda\neq (rank(\E)!)^{-1}\Z$, a subsequence $j_i\rightarrow\infty$, a holomorphic section $s_i$ over $B(i):=B_{2^{-j_i}}$ for each $i$, such that 
$$\frac{1}{\text{Vol}(2^{-1}B(i))}\int_{2^{-1}B(i)} |s_i|^2 =1,$$
but
$$\frac{1}{\text{Vol}(B(i))}\int_{B(i)} |s_i|^2 \leq  2^{2\lambda},$$
and
$$ \frac{1}{\text{Vol}(2^{-2}B(i))}\int_{2^{-2}B(i)} |s_i|^2 \leq 2^{-2\lambda}. $$
Passing to a further subsequence, we may assume the sequence of connections $A_i$, given by the pull-back of $A$ by the map $u_i: z\mapsto 2^{-j_i}z$ converges to a tangent cone $(\E_\infty, A_\infty)$ on $\C^n$. Let $\tilde\Sigma$ be the set where the convergence is not locally smooth. Then away from $\tilde \Sigma$,  the standard elliptic estimate for $L^2$ holomorphic sections implies that by possibly passing to a further subsequence $u_i^*s_i$ \footnote{Later in this section we will often talk about convergence of a sequence of holomorphic sections over $B_{2^{-j}}$ to a holomorphic section on the tangent cones, and we often omit the pull-back notation, which we hope will not cause confusion. } converges in $C^\infty_{loc}(B\setminus\tilde\Sigma)$ to a limit holomorphic section $s_\infty$ of $\E_\infty$. From the integral condition $s_i$ satisfies, by doing change of variables, we know $s_\infty$ satisfies
$$\frac{1}{\text{Vol}(B)}\int_{B} |s_\infty|^2 \leq 2^{2\lambda}, $$
and
$$\frac{1}{\text{Vol}(B_{1/4})} \int_{B_{1/4}} |s_\infty|^2 \leq 2^{-2\lambda}. $$
On the other hand, by (\ref{eqn2-5}), we have uniform bound of $|u_i^*s_i|$ on $B_{1/2}$.  Since $\tilde\Sigma$ has locally finite codimension four Hausdorff measure, $s_\infty$ extends across $\tilde \Sigma$ (see Lemma 3 in \cite{Shiffman}), and
$$\frac{1}{\text{Vol}(B_{1/2})} \int_{B_{1/2}} |s_\infty|^2 =1. $$
Now we apply Proposition \ref{prop3-5} to conclude that $s_\infty$ must be homogeneous with 
 degree  $\lambda$. This contradicts with (\ref{rigidity}). 
\end{proof}

\begin{cor} \label{cor3-7}
For any nonzero holomorphic section $s$ of $\E$ defined in a neighborhood of $0$, then the following is well-defined 
\begin{equation}\label{eqn3-1}
d(s):=\frac{1}{2}\lim_{r\rightarrow 0}\frac{\log \int_{B_r} |s|^2}{\log r}-n \in (k!)^{-1}\Z_{\geq 0}\cup\{+\infty\},
\end{equation}
where $k=rank(\E)$. If $d(s)<\infty$, by passing to a subsequence $s$ gives rise to limit homogeneous holomorphic sections of degree $d(s)$ on all the tangent cones. 
\end{cor}
\begin{proof}
Fix any $s$ as above. Let 
$$b_j=\frac{\log \int_{B_{2^{-j}}}|s|^2}{-\log 2}$$ 
and $a_j=b_{j+1}-b_j$. By Proposition \ref{prop3-6}, for any $\lambda \notin (rank(\mathcal E)!)^{-1} \mathbb{Z}$, there exists $j_0=j_0(\lambda)$ so that for $j\geq j_0$, if 
$$a_j \leq 2\lambda+2n$$
then 
$$a_{j+1} \leq 2\lambda+2n.$$
From this, we know $\lim_j a_j$ exists in $\mathbb{R} \cup \{\pm \infty\}$, which implies 
$$
\lim_j \frac{\log\int_{B_{-j}} |s|^2}{-j\log 2}=\lim_j \frac{b_j}{j}=\lim_j \frac{\sum_{i=1}^{j-1} a_{i}}{j}
$$
is also well-defined in $\mathbb{R} \cup \{\pm \infty\}$. From this, using that for any $r\in (0,1]$, there exists a unique $j$ so that $r\in (2^{-j}, 2^{-j-1}]$ and $s$ is uniformly bounded, we know $d(s)$ is well defined. The fact that $|s|$ is bounded near $0$ implies $d(s)$ is nonnegative.  If $d(s)<\infty,$ i.e. $\lim b_j$ is finite, by repeating the argument in Proposition \ref{prop3-6}, a subsequence of the normalized sequence given by $s$ will converge to a homogeneous holomorphic section of some tangent cone $\E_\infty$, with degree equal to $d(s)$. In particular, we know $d(s)\in (k!)^{-1}\Z_{\geq 0}$. 
\end{proof}

In the following, we will call $d(s)$ the \emph{degree} of $s$.  
 
\

\

\subsection{Proof of Theorem \ref{main}} \label{section3-2}

Now we can start proving Theorem \ref{main}. In order to make the main ideas of the proof clear, we will present the proofs under assumptions of increasing generality. Let $(\E, H, A)$ be an admissible Hermitian-Yang-Mills connection on $B$ with isolated singularity at $0$. We assume $\E$ is isomorphic to ${(\iota_B)}_*\pi_B^*\underline \E$ for some vector bundle $\underline \E$ over $\C\P^{n-1}$. Let $(\E_\infty, H_\infty, A_\infty)$ be \emph{a} tangent cone of $A$ at $0$. Without loss of generality we may assume it is given by a rescaling limit coming from a subsequence of the particular sequence $\{2^{-j}\}_{j=0}^{\infty}$. We always fix the corresponding gauges to realize this convergence, which is smooth away from a closed subset $\tilde\Sigma$ of $B$ with Hausdorff codimension at least four, so that we can talk about the convergence of holomorphic sections of $\E$ to a holomorphic section of $\E_\infty$.

\subsubsection{Stable case}
In this subsection we assume $\underline\E$ is stable. By the Donaldson-Uhlenbeck-Yau theorem $\underline \E$ admits a Hermitian-Yang-Mills connection $(\underline H, \underline A)$ with Einstein constant $\mu=\mu(\underline \E)$. As in Section 2.2 this gives rise to a  simple HYM cone $(\E,  \widehat H, \widehat A)$ with holonomy $e^{-2\pi\sqrt{-1}\mu}\Id$. In this case Theorem \ref{main} reduces to the following

\begin{thm} \label{thm3-8}
 $(\E_\infty, A_\infty)$ is a simple HYM cone and is isomorphic to $(\E, \widehat A)$.
\end{thm}

To see this we consider  $\F:=Hom(\E, \E)=\E^*\otimes\E$. It has an admissible connection induced by $(\widehat H, \widehat A)$ on $\E^*$ and $(H, A)$ on $\E$. This connection is not exactly Hermitian-Yang-Mills since $(\widehat H, \widehat A)$ is only Hermitian-Yang-Mills with respect to $\omega_0$. However, since we have $\omega=\omega_0+O(|z|^2)$, and it is easy to see $|F_{\widehat A}|_{\omega_0}=O(|z|^{-2})$, we have 
\begin{equation}
|\Lambda_{\omega} F_{\widehat A}(z)|\leq C,  \ \ \ \ z\in B_{1/2}^*
\end{equation}
Notice $\F$ also has a natural section $s$ given by the identity map. By definition and (\ref{cor2-26}) we see that $d(s)=0$, and $s$ gives rise to a non-trivial limit homogeneous degree zero section $s_\infty$ of the HYM cone $(\F_{\infty}:=\E^*\otimes \E_\infty, \widehat H^*\otimes H_\infty,  A^*\otimes A_\infty)$.  Lemma \ref{lem3-2} implies that there is a simple HYM direct summand, say $\E_\infty'$ of $\E_\infty$ with holonomy $e^{-2\pi\sqrt{-1}\mu}\Id$, such that $s_\infty$ is a degree zero section of $\E^*\otimes \E_\infty'$. In particular, $s_\infty$ induces a homomorphism $\underline s_\infty: \underline \E\rightarrow \underline \E_\infty'$. Since $\underline \E$ is stable, and both sheaves have the same slope $\mu$, by Corollary \ref{cor2-5} we conclude that $\underline s_\infty$ must be an isomorphism. Hence $(\E_\infty, A_\infty)$ is a simple HYM cone and $\underline\E_\infty$ is isomorphic to $\underline\E$. By Corollary \ref{cor2-6} we see $\underline A_\infty$ must be isomorphic to $\underline A$. This finishes the proof of Theorem \ref{thm3-8}.

\

As mentioned in the introduction this case has already been proved by \cite{JSW} using PDE method. The point is that in this case (or more generally, when $\underline \E$ is the direct sum of stable vector bundles), one can as above construct a HYM cone on $\E$ and use the inequality in Theorem \ref{thm2-24} to obtain an $L^\infty$ bound between the unknown Hermitian metric and the Hermitian metric on the HYM cone.  This initial $L^\infty$ bound allows \cite{JSW} to go further to obtain higher regularity and decay estimates. 

In the case when $\underline \E$ is not a direct sum of stable vector bundles, and more seriously when the Harder-Narasimhan-Seshadri filtration of $\E$ is not given by sub-bundles, as our main result Theorem \ref{main} shows, $\E_\infty$ is not even isomorphic to $\E$ as reflexive sheaves. Hence one expects an essential difficulty by a direct PDE argument. This is also reflected in the fact that our proof above also needs to be refined and this is what we shall elaborate in the following subsections.

\

\

\subsubsection{Semistable case} \label{semistable}
In this subsection we assume $\underline\E$ is semistable and $\Gr^{HNS}(\underline\E)$ is reflexive. Let $0=\underline\E_0\subset \underline\E_1\subset\cdots\subset\underline\E_q=\underline\E$ be \emph{one} Seshadri filtration of $\underline \E$. We recall that Seshadri filtration may not be unique, but the corresponding graded sheaf $\oplus_{j=1}^q (\E_j/\E_{j-1})^{**}$ is unique up to isomorphism. Denote $\E_i={\iota_B}_*\pi_B^*\underline\E_i$. Then Theorem \ref{main} in this case reduces to the following

\begin{thm}\label{thm3-9}
$(\E_\infty, A_\infty)$ is a simple HYM cone. Moreover, $\underline\E_\infty$ is isomorphic to $\Gr^{HNS}(\underline\E)$, and $A_\infty$ is gauge equivalent to the natural Hermitian-Yang-Mills cone connection on $\iota_*\pi^*(Gr^{HNS}(\underline\E))$.
\end{thm}

We first begin with a non-technical description of what we are going to do in this section. There are two points that need to be taken care of 
\begin{itemize}
\item Calculate the degree of the pull-back sections from $\underline \E$. This is done by comparing the approximate Hermitian-Einstein metric on $\underline \E$ with the unknown Hermitian-Einstein metric. See Lemma \ref{lem3-11} below.

\item Given a linearly independent set of $N$ such pull-back sections with the same degree $\mu$, we want to get a set of $N$ limiting $L^2$ othornormal homogeneous sections of the tangent cone with degree equal to $\mu$. This requires us to do $L^2$-orthogonalization on each rescaled ball. For this reason, we need to introduce a modified  notion of degree function, see Proposition \ref{prop3-14}. It turns out that the new limiting sections have the same degree as the original ones by a refined analysis of the degree function we defined in Section \ref{Convexity for holomorphic sections} (See Corollary \ref{cor3.14} below).
\end{itemize}
Given the above, using those sections, we can represent the chosen subsheaves using a sequence of global resolution of the subsheaves $\E_r$ of $\E$ given by the chosen sections on each rescaled ball, and study the convergence of those resolutions. It turns out that these can allow us to build isomorphisms we need and finish the proof.

Now we start to prove the main theorem in detail. Using the asymptotic Riemann-Roch formula for reflexive sheaves (Theorem $1.1.24$ on Page $21$ in \cite{Lazarsfeld}) and standard vanishing theorems,  by tensoring with a large power of $\mathcal O(1)$, we may without loss of generality assume the following holds for all $i=1, \cdots, q$, 

\begin{itemize}
\item Each $\underline \E_i$ and $\underline \E_i/\underline\E_{i-1}$ is generated by its global sections; 
\item For all $i$, the following sequence is exact
$$0\rightarrow H^0(\C\P^{n-1}, \underline \E_{i-1})\rightarrow H^0(\C\P^{n-1}, \underline \E_i)\rightarrow H^0(\C\P^{n-1}, \underline \E_i/\underline \E_{i-1})\rightarrow 0. $$
\end{itemize}

By Theorem \ref{thm2-10}, for any $\epsilon>0$ we can find a Hermitian metric $\underline H_\epsilon$ on $\underline \E$ such that $$\|\sqrt{-1}\Lambda_{\omega_{FS}}F_{\underline A_\epsilon}-\mu\Id\|_{L^\infty(\C\P^{n-1})}<\epsilon$$ with $\mu=\mu(\underline \E)$. Let $H_\epsilon=|z|^{2\mu}\pi^*\underline H_\epsilon$. Then $|F_{(H_\epsilon,\bp_{\E})}|=O(r^{-2})$,  so $F_{H_{\epsilon,\bp_\E)}} \in L^{2}(B^*)$. Furthermore, for all $r\in (0, 1]$,  we have 
$$r^2\sup_{\p B_r}|\Lambda_{\omega_0} F_{(H_\epsilon,\bp_\E)}|<\epsilon.$$
Again it is easy to see 
$$r^2\sup_{\p B_r}|\Lambda_{\omega} F_{(H_\epsilon,\bp_\E)}|<\epsilon+Cr^2.$$

\begin{lem} \label{lem3-11}
For any $s={\iota_B}_*\pi_B^{*}\underline s$, where $\underline s\in H^0(\C\P^{n-1}, \underline \E)$, we have $d(s)=\mu$.
\end{lem}

\begin{proof}
By definition, there is a constant $C(\epsilon)>0$ such that for all $r\in (0, 1)$
\begin{equation}\label{eqn}
C(\epsilon)^{-1}r^{2n+2\mu}\leq \int_{B_r} |s|^2_{H_\epsilon}\leq C(\epsilon) r^{2n+2\mu}. 
\end{equation}
Applying  Lemma \ref{lem2-30} with $\delta=1$, we see (\ref{cor2-26}) holds with $H'=H_\epsilon$, so 
$$\limsup_{z\rightarrow 0}\frac{|\log \Tr_{H_\epsilon}{H}(z)|+|\log \Tr_{H}{H_\epsilon}(z)|}{-\log |z|}\leq C_0\epsilon,$$
where $C_0$ is a constant \emph{independent} of $\epsilon$. In particular, we have 
$$
(2C_0)^{-1} |z|^{\epsilon}\leq H_\epsilon \leq 2C_0 |z|^{-\epsilon} H_\epsilon
$$
near the origin, which together with Equation (\ref{eqn}) implies 
$$d(s)\in [\mu-C_0\epsilon, \mu+C_0\epsilon]$$ for all $\epsilon>0$. By letting $\epsilon$ go to $0$, we obtain $d(s)=\mu$. 
\end{proof}

At this point one can try to argue as in the stable case to build maps from $\E_i(i=1, \cdots, q)$ to $\E_\infty$, and aim for the desired isomorphism between $\iota_*\pi^*(\Gr(\underline\E))$ and $\E_\infty$. This works in a straightforward way in the case when $q=2$, but in general it becomes very complicated especially when some factor $\underline \E_i/\underline \E_{i-1}$ appears with multiplicity bigger than one in $\Gr(\underline \E)$. To overcome this we need more involved arguments. 

Let 
$$HG_i:=\{s={\iota_B}_*\pi_B^{*}\underline s : \underline s \in  H^{0}(\C\P^{n-1}, \underline \E_{i})\}, $$
and denote $HG:=HG_q$. 
This defines a filtration 
\begin{equation}
0=HG_0\subset HG_1 \subset\cdots\subset HG_{q}=HG
\end{equation}
We know each non-zero  $s\in HG$ gives rise to non-trivial homogeneous sections of the limiting sheaves of all the tangent cones, but a priori two different elements in $HG$ may yield the same limit. To deal with issue we need to refine the discussion in Section 3.1.

To make the idea clear, we start with a fixed nonzero holomorphic section $\sigma\in HG$. Denote
$$\sigma^j=\frac{\sigma|_{B_{2^{-j}}}}{\|\sigma\|_j}, $$
where 
$$\|\sigma\|_j:=\sqrt{\frac{1}{\text{Vol}(B_{2^{-j}})}\int_{B_{2^{-j}}}|\sigma|^2}. $$
Given $s\in H^0(B_{2^{-j}}, \E)$, we denote by $p^j s$  the $L^2$ orthogonal projection of $s$ to the orthogonal complement of $\sigma^j$ on $B_{2^{-j}}$. The following is a straightforward analogue of Proposition 3.11 in \cite{DS2015}, and we include a proof here for the convenience of readers. 

\begin{prop}\label{prop3-12}
Given any $\lambda\notin(rank(\E)!)^{-1}\Z$, we can find $j(\lambda, \sigma)$ such that  for any $j\geq j(\lambda, \sigma)$ and $s\in H^{0}(B_{2^{-j}},\E)$, if $\|p^{j+1} s\|_{j+1}\geq 2^{-\lambda}\|p^{j}s\|_{j}$, then $\|p^{j+2}s\|_{j+2}> 2^{-\lambda}\|p^{j+1}s\|_{j+1}$. 
\end{prop}

\begin{proof}
Otherwise, there exists $\lambda \notin (rank(\E)!)^{-1}\Z$,  a subsequence $j_i\rightarrow\infty$, and a section $s_i\in H^0(B_{2^{-j_i}}, \E)$ for each $i$,  satisfying the following:
\begin{equation}
\|p^{j_i+1} s_i\|_{j_i+1}\geq 2^{-\lambda}\|p^{j_i}s_i\|_{j_i},
\end{equation}
but
\begin{equation}
\|p^{j_i+2} s_i\|_{j_i+2}\leq 2^{-\lambda}\|p^{j_i+1}s_i\|_{j_{i+1}}.
\end{equation}
We can normalize so that $\|p^{j_i+1}s_i\|_{j_{i}+1}=1$. As in the proof of Proposition \ref{prop3-6}, by taking a further subsequence we may assume $\{p^{j_i}s_i\}_i$ converge to holomorphic sections in some tangent cone (under the pull-back maps $z\mapsto 2^{-j_i}z$), namely, 
\begin{equation}
\tau=\lim_{i\rightarrow\infty} p^{j_i}s_i, \tau'=\lim_{i\rightarrow\infty} p^{j_i+1} s_i,\tau''=\lim_{i\rightarrow\infty}p^{j_i+2} s_i
\end{equation} 
satisfying 
$$\frac{1}{\text{Vol}(B)}\int_B|\tau|^2\leq 2^{2\lambda}, \frac{1}{\text{Vol}(B_{1/2})}\int_{B_{1/2}}|\tau'|^2\leq 1, \frac{1}{\text{Vol}(B_{1/4})}\int_{B_{1/4}}|\tau''|^2\leq 2^{-2\lambda}.$$
Notice  $\tau'$ and $\tau''$ are defined on $B_{1/2}$ and $B_{1/4}$ respectively. Now we have the following $L^2$ orthogonal decomposition 
\begin{equation}\label{eqn3-9}
p^{j_i}s_i|_{B_{2^{-{(j_i+1)}}}}=p^{j_i+1} s_i+c_{i}\sigma^{j_i+1}
\end{equation}
for some constant $c_i$. 
By taking a further subsequence, we may assume $c_{i}\sigma^{j_i+1}$ converges to $c\sigma^{\infty}$ and we get 
\begin{equation}
\tau|_{B_{\frac{1}{2}}}=\tau'+c \sigma^{\infty}.
\end{equation}
Then $\tau'$ is $L^2$ orthogonal to $\sigma^{\infty}$ on $B_{\frac{1}{2}}$ and $\tau$ is $L^2$ orthogonal to $\sigma^{\infty}$ on $B$. As $\sigma^{\infty}$ is homogeneous, $\tau|_{B_{\frac{1}{2}}}$ is $L^2$ orthogonal to $\sigma^\infty$ on $B_{1/2}$. More precisely, by writting $\tau=\sum_d s_d$ as a sum of $L^2$ orthogonal homogeneous holomorphic sections, by Equation \ref{eqn3.1} we have 
$$
\int_{B_{\frac{1}{2}}} (\tau, \sigma_\infty)=\int_{B_{\frac{1}{2}}} (s_{d_0}, \sigma_\infty)=2^{-n} \int_{B}(s_{d_0}, \sigma_\infty)=2^{-n} \int_{B} (\tau, \sigma_\infty)
$$
where $d_0=d(\sigma_\infty)$. This implies $c=0$ and thus $\tau|_{B_{\frac12}}=\tau'$. Thus 
$$\|p^{j_i}s_i\|_{j_i+1}\geq 1.$$  
We have 
$$\frac{1}{\text{Vol}(B_{1/2})} \int_{B_{1/2}} |\tau'|^2=\frac{1}{\text{Vol}(B_{1/2})} \int_{B_{1/2}} |\tau|^2\geq 1$$
as the proof of Proposition \ref{prop3-6}.
Similarly $\tau|_{B_{\frac{1}{4}}}=\tau''$. So we have 
 $$(\int_{B_{1/2}}|\tau|^2)^2\geq \int_B|\tau|^2\cdot \int_{B_{1/4}}|\tau|^2. $$
By Proposition \ref{prop3-5}, $\tau$ must be a homogeneous section on the tangent cone of degree $\lambda$, which contradicts our choice of $\lambda$. This finishes the proof. 
\end{proof}

Now given $s\in H^0(B, \E)\setminus \C\langle \sigma\rangle$. Using the above Proposition, similar to the proof of Corollary \ref{cor3-7},  we obtain
 
\begin{prop}
\label{prop3-14}
We have
\begin{itemize}
\item[1.] The following is well-defined
\begin{equation}
d_{\sigma}(s):=\lim_{j\rightarrow\infty}\frac{\log\|p^{j}s\|_{j}}{-j\log2}\in (k!)^{-1}\Z\cup \{\infty\}, 
\end{equation}
where $k=rank(\E)$. 
\item[2.] If $d_{\sigma}(s)<\infty$, then the sequence $\{\sigma^j, p^js/\|p^js\|_j\}$ converges to an $L^2$ orthonormal set of homogeneous sections on the tangent  cones with degree $d(\sigma)$ and $d_\sigma(s)$ respectively. 
\end{itemize}
\end{prop}
 
Since $\| p^{j}s\|_{j}\leq \| s\|_{j}$, we get $d_\sigma(s)\geq d(s)$. Up to this point it is still a straightforward analogue of the general result in \cite{DS2015}. Now in our case the new point is that the converse inequality also holds for $s\in HG\setminus\C\langle\sigma\rangle$, based on the following crucial observation

\begin{prop}\label{prop3-11}
For all $\epsilon>0$ sufficiently small, there exists $j(\epsilon)$ large enough so that for any $j\geq j(\epsilon)$ and for any $s\in HG$, we have 
\begin{equation}\label{equ3-1}
2^{\mu-\epsilon}\|s\|_j \leq \|s\|_{j-1}\leq 2^{\mu+\epsilon}\|s\|_j.
\end{equation}
\end{prop}

\begin{proof}
To see the first inequality, we assume $\epsilon>0$ is sufficiently small so that $\mu-\epsilon\notin (rank (\E)!)^{-1}\Z$. Let $j(\epsilon)$ be the number $j_0(\mu-\epsilon)$ given in Proposition \ref{prop3-6}, i.e. by letting $\lambda=\mu-\epsilon$ in Proposition \ref{prop3-6}. If for some $j\geq j(\epsilon)$ and  some $s\in HG$, the first inequality does not hold, i.e.,  
$$2^{\mu-\epsilon}\|s\|_j>\|s\|_{j-1}. $$
Then it follows from Proposition \ref{prop3-6} that the same inequality 
$$2^{\mu-\epsilon}\|s\|_k\geq \|s\|_{k-1} $$
holds for all $k\geq j$. In particular, this implies $d(s)\leq \mu-\epsilon$. Contradiction.

For the second inequality, fix any norm on $HG$, and let $S^N$ be the unit sphere in it, where $N+1=\dim HG$. Since the inequality is invariant under multiplying $s$ by constant and $S^N$ is compact, it suffices to show that given any $s_0\in S^N$, there is an open neighborhood $U$ of $s_0$ in $S^N$, such that (\ref{equ3-1}) holds for all $s\in U$.  Notice since $d(s_0)=\mu$,  by Corollary \ref{cor3-7}, we can find $j(\epsilon, s_0)$ such that for all $j\geq j(\epsilon, s_0)$ the following holds
$$ \|s_0\|_{j-1}\leq 2^{\mu+\epsilon/2}\|s_0\|_j. $$
We may choose $j(\epsilon, s_0)>j_0=j_0(\lambda)$ where $\lambda=\mu+\frac{\epsilon}{2}$ and $j_0$ is given as in Proposition \ref{prop3-6} and we may assume $\mu+\epsilon/2\notin (rank(\E)!)^{-1}\mathbb{Z}$ for $\epsilon$ sufficiently small. 
Now by continuity we can find an open neighborhood $U$ of $s_0$ such that for all $s\in U$, we have 
$$ \|s\|_{j(\epsilon, s_0)-1}\leq 2^{\mu+\epsilon}\|s\|_{j(\epsilon, s_0)}. $$
So by Proposition \ref{prop3-6} again we see for all $s\in U$ and $j\geq j(\epsilon, s_0)$
$$\|s\|_{j-1}\leq 2^{\mu+\epsilon}\|s\|_{j}. $$
Now we can choose $j(\epsilon)$ by compactness of $S^N$. 
\end{proof}

\begin{cor}\label{cor3.14}
For all $s\in HG\setminus \C\langle \sigma\rangle$,  $d_\sigma(s)=\mu$. 
\end{cor}
\begin{proof}
 First, by Proposition \ref{prop3-11} for any $\epsilon>0$ sufficiently small we can choose $j$ large so that for any $s\in HG$, we have 
\begin{equation} \label{eqn3-12}
2^{\mu-\epsilon}\|s\|_j \leq \|s\|_{j-1}\leq 2^{\mu+\epsilon}\|s\|_j. 
\end{equation}
 \textbf{Claim:} There is a constant $C>0$ independent of $\epsilon$ and $j$ such that for all $s\notin\C\langle\sigma\rangle$, 
$$ \|p^j s\|_j\geq (1-C\epsilon) \|p^{j-1}s\|_j\geq (1-C\epsilon) 2^{-\mu-\epsilon}\|p^{j-1}s\|_{j-1}. $$

\

 Given this \textbf{Claim}, we can choose $\epsilon>0$ sufficiently small so that $(1-C\epsilon)2^{-\mu-\epsilon}=2^{-\mu-\epsilon'}$ with $(rank(\E)!)^{-1} \mathbb{Z}\cap (\mu, \mu+\epsilon']=\emptyset$. Then we choose $j$ bigger than $j(\mu+\epsilon', s_1)$, and apply Proposition \ref{prop3-12} and Proposition \ref{prop3-14} to get $d_\sigma(s)\leq \mu+\epsilon'$. As mentioned above,  $d_\sigma(s)\geq \mu$, which then implies $d_\sigma(s)=\mu$ since $\epsilon$ can be made arbitrarily small and so is $\epsilon'$.
 
\

Now we prove the \textbf{Claim}. It suffices to show that for $\sigma, \tau\in HG$ with  $\|\sigma\|_{j-1}=\|\tau\|_{j-1}=1$ with $\langle \sigma, \tau\rangle_{j-1}=0$, we have 
$|\langle \sigma, \tau\rangle_{j}|\leq C\epsilon$ for a constant $C>0$ independent of $j$ and $\epsilon$. This can be easily seen using  (\ref{eqn3-12}) and the elementary fact that
$$Re \langle \sigma, \tau\rangle_j= \frac{1}{2}(\|\sigma+\tau\|_{j}^2-\|\sigma\|_j^2-\|\tau\|_j^2). $$
\end{proof}

Arguing by induction it is straightforward to obtain a basis $\mathcal B$ of $HG$, which can be written as the union 
$$\mathcal B=\bigcup_{r=1}^q \mathcal B_r, $$
with $$\mathcal B_r=\{\sigma_{r, 1}, \cdots, \sigma_{r, s_r}\}\subset HG_r, $$
such that the following hold

\begin{itemize}
\item
$\mathcal B_r$ descends to a basis of $HG_r/HG_{r-1}$;
\item
For any fixed $(r, t)$, let $p^j\sigma_{r, t}$ be the $L^2$ projection on $B_{2^{-j}}$ to the orthogonal complement of the linear span $\C\langle \bigcup_{(q, s)<(r, t)}\sigma_{q, s}\rangle$, where $(q, s)<(r, t)$ means either $q<r$, or $q=r$ and $s<t$. If $(r, t)=(1, 1)$ then we do not do projectio. Then after passing to subsequences $p^j\sigma_{r, t}/\|p^j\sigma_{r, t}\|_{j}$ converges (again, under the natural pull-back map) to homogeneous holomorphic sections $\zeta_{r, t}$ of degree exactly $\mu$ on all the tangent cones, that are orthogonal to all $\zeta_{q, s}$ with $(q, s)<(r, t)$. 
\end{itemize}

Here the induction is done by adding more sections and repeating the argument for Proposition \ref{prop3-14} and Corollary \ref{cor3.14}.

\

For each fixed $(r, t)$, we denote by $\sigma^j_{r, t}$ the $L^2$ projection of $\sigma_{r, t}$ to the orthogonal complement of $HG_{r-1}$  on $B_{2^{-j}}$. In particular, $\sigma^j_{1, t}=\sigma_{1, t}$ for all $1\leq t\leq s_1$ and all $j$. 

 We denote 
$$M_{r}^j:=\sup_{1\leq t\leq s_r} \|\sigma_{r, t}^j\|_j. $$
Then from the above discussion we know for each $(r, t), 1\leq t\leq s_r$, the sequence of sections $(M_{r}^j)^{-1}\sigma_{r, t}^j$ converge (by passing to subsequences) to homogeneous holomorphic sections $\sigma_{r, t}^\infty$ on the tangent cones which, if non-zero, is of degree exactly $\mu$. Moreover,  there is at least one $t$ such that $\sigma_{r, t}^\infty\neq 0$.

For each $j$, the elements $\{\frac{1}{M_{r}}\sigma^j_{r, t}|1\leq t\leq s_r\}$ define an obvious homomorphism $\tau_{r}^j: \O^{\oplus s_r}\rightarrow \E$ over $B_{2^{-j}}$ as $\tau_r^j(a_1,\cdots, a_{s_r})=\sum_t a_t \frac{1}{M_{r}}\sigma^j_{r, t}$.  For each $j$, by global generation property, $\tau_r^j$ induces a map from $\E_r/\E_{r-1}$ to $\E/\E_{r-1}$ and we denote this by $\phi^j_r$. The goal is to show $\phi^j_r$ also converges to a homomorphism from $(\E_r/\E_{r-1})^{**}$ to  $\E_\infty$ that descends over the projective space. 
  Recall that we have denoted by $\tilde \Sigma$ the locus where the convergence to the tangent cone $\E_\infty$ is not smooth. Denote by $\tilde\Sigma'$ the union of $\tilde\Sigma$ and the locus where $\oplus^q_{r=1} \E_r/\E_{r-1}$ is not locally free. It is also a closed subset of $\C^n\setminus \{0\}$ with Hausdorff codimension at least $4$.

We shall prove the following statements by induction on $r$. Notice Theorem \ref{thm3-9} is a direct consequence of $(a)_r$ in the following while $(b)_r$ is used for the inductive argument.

\begin{itemize}
\item[$(a)_r$:] There is a simple HYM cone direct summand $\S_r$ of $\E_\infty$ with holonomy $e^{-2\pi \sqrt{-1}\mu}$ with $\S_r\cap \S_{k}=0$ for all $k<r$ (we make the convention that $\mathcal S_0=0$), such that the map $\tau_{r}^\infty$ induces a homogeneous isomorphism $\phi_{r}^\infty: \E_r/\E_{r-1}\rightarrow\S_r$.
\item[$(b)_r$:] On $B\setminus \tilde \Sigma'$, under the fixed gauge we use to identify $\E$ with $\E_\infty$,  the sub-bundle $\E_r$ converges smoothly to $\bigoplus_{k\leq r}\S_k$. In other words, under this convergence, for any fixed point $x\in B\setminus \tilde\Sigma'$,  unit vectors of $\E_r$ over $2^{-j}x$  naturally converge to unit vectors of $\bigoplus_{k\leq r}\S_k$ over $x$. 
\end{itemize}

First consider the case $r=1$.  We claim $\tau^\infty_{1}$ descends to a homomorphism $\phi_{1}^\infty$ from $\E_1$ to $\E_\infty$. To see this, we first define the corresponding vector bundle homomorphism on $B\setminus \tilde\Sigma'$. For any $x\in B\setminus \tilde \Sigma'$ and any $\xi$ in the fiber $\E_1(x)$ (as a vector space, not to be confused with the sheaf stalk), we can write 
$$\xi=\sum_s a_s \sigma_{1, s}(x)$$ for $a_s\in \C$, then we define 
$$\phi_{1}^\infty(\xi)=\sum_s a_s\sigma^\infty_{1, s}(x)$$ as a vector in the fiber $\E_\infty(x)$. To see this is well-defined, suppose 
$$\sum_s a_s\sigma_{1, s}(x)=0,$$ then it follows that for any $j$ we have 
$$\sum_s  \frac{1}{M_{1}^j} a_s\sigma^j_{1, s}(2^{-j}x)=0,$$ hence $\sum_s a_s\sigma^\infty_{1, s}(x)=0$ because by definition we have smooth convergence at $x$. It is also easy to see that $\phi_{1}^\infty$ is holomorphic. Now we can view $\phi_{1}^\infty$ as a holomorphic section of the reflexive sheaf $(\E_1)^*\otimes \E_\infty$ over $B\setminus \tilde\Sigma'$. By Hartogs's theorem (as in the proof of Corollary \ref{cor2-5}, using lemma 3 in \cite{Shiffman} ) we know that $\phi_{1}^\infty$ extends to a holomorphic section over the whole $B$, and since it is induced by the correspondence of homogeneous sections pulled back from the projective space, it naturally descends over the projective space. More precisely, write $\E_\infty$ as the direct sum of simple HYM cones of different holonomy, since $d(\sigma_{1, s})=d(\sigma^\infty_{1, s})=\mu$ by Lemma \ref{lem3-11}, Corollary \ref{cor3-7} and Lemma \ref{lem3-2}, the image of $\phi_1^\infty$ must be contained in the simple HYM cone which is a direct summand of $\E_\infty$ with holonomy $e^{-2\pi\sqrt{-1}\mu}$ which we denote by $\S_1$, then $\phi_1^\infty$ induces a homomorphism $\underline \phi_1^\infty: \underline\E_1\rightarrow\underline \S_1$, with $\mu(\underline \E_1)=\mu(\underline \S_1)=\mu$ by Lemma \ref{lem3-2}. Since $\underline \E_1$ is stable and reflexive it follows that $\underline\phi_1^\infty$ is an isomorphism onto a direct summand of $\underline\S_1$. For simplicity we still denote the later by $\S_1$. This proves $(a)_1$.

For any $x\in B\setminus \tilde\Sigma'$, we can choose $r_1$ sections $\gamma_1, \cdots, \gamma_{r_1}$ in the span of $\sigma^\infty_{1, s}$ that are orthonormal at $x$ (as vectors in the fiber of the bundle $\S_1$), where $r_1=\rk(\E_1)$. By definition each $\gamma_l$ is the limit of holomorphic sections $\gamma_{l}^j$ of $\E_1$ over $B_{2^{-j}}$ (for a subsequence of $\{j\}$). It follows that $\gamma_{l}^j (1\leq l\leq r_1)$ generate the fiber of $\E_1$ over the point $2^{-j}x$, and are approximately orthonormal. Then it is clear that the fiber $(\E_1)_{2^{-j}x}$ converges to $(\S_1)_x$.
  This proves $(b)_1$.

Now suppose we have established both $(a)_1, \cdots, (a)_{r-1}$ and $(b)_1, \cdots, (b)_{r-1}$, and we want to prove $(a)_r$ and $(b)_r$.  We write $\E_\infty=(\bigoplus_{k<r} \S_k)\bigoplus \V_r$.  We claim $\tau_{r}^\infty$ still induces a well-defined homomorphism  
$$\phi_r^\infty: \E_r/\E_{r-1} \rightarrow \E_\infty/(\bigoplus_{k<r}\S_k)\simeq \V_r.$$
 Indeed,  we can  first define a vector bundle homomorphism on $B\setminus \tilde\Sigma'$ as follows. Given a vector  $\xi$ in the fiber of $\E_r$ over $x\in B\setminus \tilde\Sigma'$, we can find a global section $u\in HG_r$ in the span of $\{\sigma_{r, t}, 1\leq t\leq s_r\}$,  with $u(x)-\xi\in (\E_{r-1})_x$. Writing 
 $$u=\sum_{1\leq t\leq s_r}a_{r, t} \sigma_{r, t}+\sum_{k< r, 1\leq t\leq s_k} a_{k, t} \sigma_{k, t}$$ for some constants $a_{r, t}, a_{k, t}\in \C$. Let $u^\infty$ be the (pointwise) projection of $\sum_{1\leq t\leq s_r}a_{r, t}\sigma^\infty_{r, t}$ to $\E_\infty/(\bigoplus_{k<r}\S_k)\simeq \V_r$, and we define $\phi_{r}^\infty(\xi)=u^\infty(x)$. Then $\phi_{r}^\infty$ is well-defined on $B\setminus \tilde\Sigma'$. Indeed, if $u(x)=0$, then for all $j$ we know $\frac{1}{M_{r}^j}\Sigma_{1\leq t\leq s_r} a_{r, t} \sigma^j_{r, t}(2^{-j}x)$ is in the fiber of $\E_{r-1}$ at $2^{-j}x$. Then by $(b)_{r-1}$ we know that  over $B\setminus \tilde\Sigma'$ the limit $\sum_{1\leq t\leq s_r} a_{r, t}\sigma^\infty_{r, t}$ must be contained in $\bigoplus_{k<r}\S_k$. It also follows that $\phi_r^\infty$ factors through $\E_r/\E_{r-1}$ which we still denote by $\phi_{r}^\infty$. As before we can view $\phi_{r}^\infty$ as a holomorphic section of the reflexive sheaf $(\E_r/\E_{r-1})^*$ over $B\setminus\tilde\Sigma'$, hence by Hartogs's theorem it extends to a global section over $B$. We  claim  $\phi_{r}^\infty$ is non-trivial. Indeed, if it were trivial, it would mean that any limit section $\sigma^\infty_{r, t} (1\leq t\leq s_r)$ is a section of $\bigoplus_{k<r}\S_k$. We know there must be some $t$ such that $\sigma^\infty_{r, t}\neq 0$, but on the other hand by our construction it must be  $L^2$ orthogonal to $\sigma^\infty_{k, s}$ for all $k<r$ and $1\leq s\leq s_k$. Now by induction assumption $(a)_1, \cdots (a)_{r-1}$ and  our hypothesis on the global sections of each $\underline \E_k/\underline\E_{k-1}$, we know for all $k<r$, the global sections of $\underline \S_{k}\cong \underline \E_k /\underline \E_{k-1}$ with $\mu(\underline \S_{k})=\mu$ are given by $\{\underline\sigma_{k, t}: 1\leq t\leq s_k\}$. If $\phi^\infty_r$ becomes trivial, then as mentioned above, for some $k$, $\underline\S_{k}$ will have one extra section perpendicular to all sections in $\{\sigma_{k, t}: 1\leq t\leq s_k\}$ which is a contradiction. This is the place where the assumption in our main theorem needs to be made, i.e. $\underline\E_k /\underline \E_{k-1}$ is reflexive so that we know all the sections of $\underline \S_k$ comes from the limits.

Now it is clear from construction that $\phi_{r}^\infty$ descends to be over the projective space and its image must be contained in a simple HYM cone which is a direct summand $\S_r$ of $\V_r$ of holonomy $e^{-2\pi\sqrt{-1}\mu}$, hence descends to a map $\underline \phi_r^\infty$ from $\underline \E_r/\underline \E_{r-1}$ to $\underline \S_r$. Since $\underline \E_r/\underline \E_{r-1}$ is a stable reflexive sheaf, $\underline\phi_{r}^\infty$ must be an isomorphism onto a direct summand of $\underline \S_r$ which we still denote by $\underline\S_r$. This establishes $(a)_r$.  For any $x\in B\setminus \tilde\Sigma'$, for each $k=1, \cdots, r-1$, we choose $r_k$ sections $\gamma_{k, l}(1\leq l\leq r_k)$ in the span of $\{\sigma_{k, t}|k\leq r, 1\leq t\leq s_k\}$ (where $r_k$ is the rank of $\E_k/\E_{k-1}$), so that they generate an orthonormal basis of the fiber $(\bigoplus_{k\leq r}\S_k)_x$. Each $\gamma_{k, l}$ is the limit of the corresponding section $\gamma_{k, l}^j$ in the span of $\{\sigma_{k, t}^j|1\leq t\leq s_k\}$, and $\{\gamma_{k, l}^j\}$  are approximately orthonormal at $x$. This easily implies $(b)_r$.

\begin{rmk} \label{rmk3-15}
We point out that here we make crucial use of the hypothesis that $\Gr(\underline \E)$ is reflexive; in general it is only torsion free and $\sigma^\infty_{k, s}$ only generates a torsion free subsheaf of $\S_k$. This is a key technical difficulty in extending the technique developed here to prove more general results. In \cite{CS2} we will use a different approach, bypassing this difficulty, to study the case without the assumption that $\Gr(\underline\E)$ being reflexive. 
\end{rmk}

\

\subsubsection{General case} \label{S3-5}

Now we assume $\underline \E$ is a general holomorphic vector bundle over $\C\P^{n-1}$ such that $\Gr^{HNS}(\underline \E)$ is reflexive.  Compared to the semistable case treated in Section \ref{semistable}, the new difficulty lies in the construction of a ``good" comparison metric, especially when the Harder-Narasimhan filtration has singularities.

Let 
$$0=\underline \E_0\subset \underline \E_1\subset\cdots \underline\E_{m}=\underline \E$$ 
be the Harder-Narasimhan filtration of $\underline \E$, with $\mu_i=\mu(\underline\E_i/\underline\E_{i-1})$ strictly decreasing in $i$, and 
$$\underline \E_{i-1}=\underline\E_{i, 0}\subset \underline \E_{i, 1}\subset \cdots \underline \E_{i, q_i}=\underline\E_{i}$$
be  a Seshadri filtration of $\underline \E_i$. 

As in Section \ref{semistable},  by tensoring $\underline \E$ with $\O(p)$ for $p$ large we may assume each $\underline\E_i$, $\underline\E_{i, q}$ is generated by its global sections, and for all $i$ and $q\geq 1, q'\geq 1$,  we have a short exact sequence of the form 
$$0\rightarrow H^0(\C\P^{n-1}, \underline \E_{i, q-1})\rightarrow H^0(\C\P^{n-1}, \underline \E_{i, q})\rightarrow H^0(\C\P^{n-1}, \underline \E_{i, q}/\underline \E_{i, q-1})\rightarrow 0. $$

For $i=1, \cdots, m$, we  define 
$$HG_i:=\{s={\iota_B}_*\pi_B^*\underline s|\underline s\in H^0(\C\P^{n-1}, \underline \E_i)\}, $$
then we have
$$0=HG_0\subset HG_1\subset\cdots \subset HG_m. $$
The key property is

\begin{prop}\label{onesidedbound}
For any $s\in HG_i\setminus HG_{i-1}$, $d(s)\leq \mu_i$.
\end{prop}

\begin{rmk}
This is proved via analytic means. Below (Lemma \ref{twosidedbound}) we shall prove the equality indeed holds, but we need to make crucial use of the property of algebraic stability. 
\end{rmk}

Assuming this for the moment, we first finish the proof of Theorem \ref{main}. It is a slight modification of the proof in the semistable case in Section \ref{semistable}, and below we shall only outline the overall argument and point out the places where the change is necessary. We first prove 

\begin{lem} \label{twosidedbound}
For any $s\in HG_i\setminus HG_{i-1}$, $d(s)= \mu_i$.
\end{lem}

\begin{proof}[Proof of Lemma \ref{twosidedbound}]
By Proposition \ref{onesidedbound},  we know for any $s\in HG_i\setminus HG_{i-1}$, $d(s)\leq \mu_i$. In order to prove the equality holds, it suffices to show for all $i$ and $s\in HG_i$, $d(s)\geq \mu_i$. Choose a basis $\{\sigma_s\}$ of $HG_i$, and let  $M_j=\sup_s\{\|\sigma_s\|_j\}$. Consider the sequence of sections $\{\frac{1}{M_j}\sigma_s\}$ on $B_{2^{-j}}$. Passing to a subsequence, we may assume for each $s$ as $j\rightarrow\infty$ these converge to limit sections $\sigma^\infty_s$ on some tangent cone. By definition there is at least one $s$ with $\sigma_s^\infty\neq 0$, and the degree of all the non-zero $\sigma_s^\infty$'s must be all equal to the same number, say $\nu$. It is also clear that $d(\sigma_s)\geq \nu$ for all $s$, and we only need to show $\nu\geq \mu_i$. 

As in the discussion in the semistable case, we obtain a non-trivial homomorphism $\phi^\infty$ from $\E_i$ to $\E_\infty$. Moreover, this is  homogeneous and the image is in a simple HYM cone  $\S$ of holonomy $e^{-2\pi \sqrt{-1}\nu}$, which is a direct summand of $\E_\infty$. So it descends to a non-trivial homomorphism $\underline\phi^\infty: \underline \E_i\rightarrow \underline \S$ with $\mu(\underline \S)=\nu$. Hence it induces a non-trivial homomorphism from $\underline\E_{k, q}/\underline \E_{k, q-1}$ to $\underline \S$, for some $k\leq i$ and $1\leq q\leq q_k$. Since $(\underline \E_{k, q}/\underline \E_{k,q-1})^*\otimes \underline\S$ is polystable by Proposition \ref{prop2-4} we conclude $\nu\geq \mu_i$.
\end{proof}

Fix a vector space decomposition
\begin{equation*}
HG_m=\bigoplus_{i=1}^{m} V_i
\end{equation*}
where each $V_i$ is chosen to be complementary to $HG_{i-1}$ in $HG_i$ with repect to any fixed Hermitian metric. We now build an isomorphism between the graded sheaf $\bigoplus_{i=1}^m \bigoplus_{q=1}^{q_i} \E_{i, q}/\E_{i, q-1}$ and $\E_\infty$ by working inductively on both $i$ and $q$. 

When $i=1$ we can use exactly the same arguments as in the proof of semistable case,  in view of Lemma \ref{twosidedbound}. For the case $i>1$ we need
to replace Proposition \ref{prop3-11} with the following
\begin{prop}
For $\epsilon>0$ sufficiently small, we can find $j(\epsilon)$ so that for any $j\geq j(\epsilon)$ and for all $i$ and for any $s\in V_i$, we have 
\begin{equation*}
2^{\mu_i-\epsilon}\|s\|_j\leq \|s\|_{j-1} \leq 2^{\mu_i+\epsilon}\|s\|_j.
\end{equation*}
\end{prop}
The proof is exactly the same as that of Proposition \ref{prop3-11}: the key point is that by Lemma \ref{twosidedbound} we know $d(s)=\mu_i$ for all nonzero $s\in V_i$. Given this, we can do $L^2$ projections between any chosen basis of holomorphic sections of the same degree $\mu_i$ so that they will converge to a $L^2$ orthonormal set of homogeneous sections of the tangent cone with degree $\mu_i$. In this way, we can repeat the argument in the semistable case to finish the proof for the main theorem in general. 

Indeed,the case $i>1$ is proved in almost the same way as in the semistable case. Here we only point out the places that we need to be careful about when we build an isomorphism between $\underline\E_{i,q}/\underline\E_{i,q-1}$ and a direct summand of $\E_\infty$ inductively on $(i,q)$. Suppose the isomorphisms between $\underline\E_{s,t}/\underline \E_{s,t-1}$ and $\underline \S_{s,t}$ have been built for $(s,t)<(i,q)$, i.e. for $(s,t)$ satisfying either $s\leq i, t<q$ or $s<i$.  

For the factor $\underline\E_{i,q}/\underline\E_{i,q-1}$, we can perform orthogonal projection within $V_i\cong HG_i/HG_{i-1}$, and obtain corresponding limit map from $\E_{i, q}/\E_{i, q-1}$ onto a simply HYM cone direct summand $\S_{i, q}$ of $\E_\infty$. If there is no $l<k$ such that $\mu_l\equiv \mu_k (\mod \   \Z)$ then we know $\S_{k, q}$ must have different holonomy from $\S_{l, u}$ for all $l<k$. In particular, the image induced by the homogeneous sections will automatically avoid $\oplus_{(s,t)<(i,q)}\underline \S_{s,t}$. In this case we can follow exactly the same arguments as in the semistable case. If there are $l<k$ with $\mu_l\equiv \mu_i (mod \ \Z)$ then we need to enlarge $V_i$ by also including those sections of the form $\iota_*\pi^*s$ where $s\in H^0(\C\P^{n-1},  \underline \E_l\otimes \O(p))$ with $d(s)=\mu_k$, $l<k$ and some $p\in \Z$. Then we can perform orthogonal projection here, and the arguments go through as before. The reason that we need this extra consideration is due to the fact that in this case $\S_{k, q}$ and $\S_{l, u}$ have the same holonomy so can not be automatically separated. This finishes the proof of Theorem \ref{main}.

\

The rest of this subsection is devoted to the proof of Propsition \ref{onesidedbound}. When the Harder-Narasimhan filtration is given by sub-bundles, there is a direct construction. Since in general the filtration may have singularities, more delicate arguments are required. We emphasize here the construction holds for any holomorphic vector bundle $\underline \E$ and we do not need $Gr^{HNS}(\underline \E)$ to be reflexive. 

Let $\underline \Sigma\subset\C\P^{n-1}$ be the subset where the Harder-Narasimhan filtration is not given by sub-bundles. It is of complex codimension at least $2$. We will use the following notations 
\begin{itemize}
\item $\Sigma=\pi^{-1}(\underline \Sigma)$;
\item $\Sigma_{r}= \Sigma \cap S(r)$ where $S(r)=\{x\in B^*: |x|=r\}$;
\item $\underline\Sigma^s=\{x\in \C\P^{n-1}|d(x, \underline \Sigma)\leq s\}$, where the distance is measured by the fixed Fubini-Study metric on $\C\P^{n-1}$; 
\item $\Sigma^s=\pi^{-1}(\underline\Sigma^s)$;
\item $\Sigma_r^s=\{x\in S^{2n-1}(r)|d(x, \Sigma_r)<s\}$, where the distance is measured with respect to the round metric on $S^{2n-1}(r)$.
\end{itemize}

\begin{prop}\label{constructedmetric}
For any $0<\epsilon<<1$, there exists a smooth Hermitian metric $H_{\epsilon}$ on $\E|_{B^*}$ satisfying the following
\begin{enumerate}[(i).]
\item $|F_{(H_{\epsilon}, \bp_\E)}|\in L^{1+\delta}(B^*)$ for some $\delta>0$;
\item $\limsup_{r\rightarrow 0} r^{1-2n}\int_{S^{2n-1}(r)}r^2|\Lambda_{\omega}F_{(H_\epsilon,\bp_\E)}| \leq\epsilon$;
\item $|z|^2|\Lambda_{\omega}F_{(H_{\epsilon},\bp_\E)}(z)|\leq \epsilon$ for $z\notin \Sigma^{10^{-4}}$;
\item For any $s\in HG_i\setminus HG_{i-1}$, 
$$
\lim_{r\rightarrow 0} \frac{1}{2}\frac{\log \int_{B^*_r\setminus \Sigma^{10^{-3}}}|z|^\epsilon |s|_{H_\epsilon}^2}{\log r}-n=\mu_i+\frac{\epsilon}{2}.
$$
\end{enumerate}
\end{prop}

We will first prove Proposition \ref{onesidedbound} assuming Proposition \ref{constructedmetric},  and then prove Proposition \ref{constructedmetric}.

\

\begin{proof}[Proof of Proposition \ref{onesidedbound}]
For any $0<\epsilon\leq 1$, let $H_{\epsilon}$ be the metric given in Proposition \ref{constructedmetric}. 
Let $g=\log \Tr_{H} H_\epsilon$, and $f(z)=|z|^2|\Lambda_{\omega}F_{(H_{\epsilon},\bp_\E)}(z)|$. By Theorem \ref{thm2-24}, we have on $B^*$, 
$$\Delta g\geq -|z|^{-2} f(z), $$
So by items (i), (ii), (iii) in Proposition \ref{constructedmetric}, Lemma \ref{lem2-30}, and  item (2) in Lemma \ref{lem2-19} (replacing $|z|/2$ by $|z|/A$ for some big but fixed $A$), we see that there is a constant $C$ independent of $\epsilon$ such that for any $z\notin \Sigma^{10^{-3}}$, 
$$g(z)\leq C-\epsilon \log |z|. $$
In other words, we have 
$$
H \geq e^{-C}|z|^{\epsilon}H_\epsilon. 
$$
Now for any $s\in HG_i\setminus HG_{i-1}$, we have 
$$
\int_{B_r^*}|s|^2_{H} \geq  e^{-C} \int_{B_r^*\setminus \Sigma^{10^{-3}}} |z|^\epsilon |s|_{H_\epsilon}^2, 
$$
which, by item (iv) in Proposition \ref{constructedmetric}, implies that 
\begin{equation*}
\begin{aligned}
d(s) =\lim_{r\rightarrow 0}\frac{1}{2} \frac{\log \int_{B_r^*} |s|_H^2}{\log r}-n\leq \mu_i+\frac{\epsilon}{2}
\end{aligned}
\end{equation*}
Let $\epsilon\rightarrow 0$, we get $d(s)\leq\mu_i$. This finishes the proof of Proposition \ref{onesidedbound}.
\end{proof}

Before starting the proof of Proposition \ref{constructedmetric}, we need a lemma concerning the existence of a good cut-off function.

\begin{lem}\label{cut-off}
For any fixed $N>>1$ , there exists $R(N)\in (0, 10^{-7})$ which is decreasing with respect to $N$, and a constant $C=C(N)>0$ so that for any $R\in (0, R(N)]$, there exists a smooth function $\chi_R: B^* \rightarrow [0,1]$ such that the following holds on $B^*$
\begin{itemize}
\item $ \chi_R|_{\Sigma^{R_r}_r}\equiv 1$; 
\item $\chi_R|_{S^{2n-1}(r)\setminus\Sigma^{200R_r}_r}\equiv 0$; 
\item $|\nabla \chi_{R}| \leq C R_r^{-1} $; 
\item $|\nabla^2 \chi_{R}|\leq C R_r^{-2}$. 
\end{itemize}
Here $R_r=Rr^{N}$. 
\end{lem}

\begin{proof}
Let $\phi: \mathbb{R}\rightarrow [0,1]$ be a  smooth function so that $\int_{\mathbb{R}}\phi(t) dt=1$, $\text{Supp}(\phi)=(-2, 2)$ and $\phi(t)=1$ for $t$ in a neighborhood of $0$. For $s>0$ we denote 
$$\phi_s(t)=\phi(\frac{t}{s}).$$
Define $\psi_{r}: S^{2n-1}(r)\rightarrow\mathbb{R}$ by setting 
\begin{equation*}
\psi_r=
\begin{cases}
1, \ \ x\in \Sigma^{100 R_r}_r ; \\
0, \ \  x\notin \Sigma^{100 R_r}_r.
\end{cases}
\end{equation*}
Then we define $f_r: S^{2n-1}(r) \rightarrow [0,1]$ by
\begin{equation*}
f_r(x)=\frac{ \int_{S^{2n-1}(r)} \psi_{r}(y) \phi_{R_r}(d_{S^{2n-1}(r)}( x,y)) \dvol_{S^{2n-1}(r)}(y)}{\int_{S^{2n-1}(r)}\phi_{R_r}(d_{S^{2n-1}(r)}(x,y)) \dVol_{S^{2n-1}(r)}(y)}.
\end{equation*}
Since $\phi_{R_r}(d_{S^{2n-1}(r)}(x,y))= 1$ if $d_{S^{2n-1}(r)}(x,y)$ is small, we know $f_r$ is indeed smooth. It is also direct to see $f_r=0$ on $S^{2n-1}(r) \setminus\Sigma^{150 R_r}_r$, $f_r=1$ on $\Sigma^{50R_r}_r$. Furthermore we have on $S^{2n-1}(r)$, 
$$|\nabla f_r|\leq C R_r^{-1}$$ and $$|\nabla^2 f_r|\leq C R_r^{-2}$$ for some constant $C=C(N)$.  Finally we define the smooth function $\chi_R: B^*\rightarrow [0,1]$ by
\begin{equation*}
\chi_R(x)=\int_{\mathbb{R}} f_t(t\underline x)\cdot \frac{1}{R_r}\phi(\frac{1}{R_r}(r-t))dt
\end{equation*}
where $r=|x|$, and $\underline x=|x|^{-1}x$. Now we verify $\chi_R$ satisfies the desired properties for $0<R\leq R(N)$,  where  
$$R(N)=\min\{\frac{1}{2}[(\frac{4}{3})^{\frac{1}{N-1}}-1],  \frac{1}{2}(1-50^{-\frac{1}{N-1}})\}.$$
The choice of $R(N)$ comes from the discussion below.
\begin{itemize}
\item If $x\in S^{2n-1}(r)\setminus\Sigma_r^{200R_r}$, then $t\underline x \in S^{2n-1}(t)\setminus \Sigma_t^{200\frac{t R_r}{r}}$. If $|\frac{r-t}{R_r}|\leq 2$, i.e., $$(1-2\frac{R_r}{r}) r\leq t\leq (1+2\frac{R_r}{r})r,$$ then $$\Sigma_t^{150R_t}\subset \Sigma_t^{200 (1+2\frac{R_r}{r})^{-N+1} R_t}\subset\Sigma_t^{200t\frac{R_r}{r}}$$ for $0<R\leq \frac{1}{2}[(\frac{4}{3})^{\frac{1}{N-1}}-1]$, thus $$t\underline x \in S^{2n-1}(t)\setminus\Sigma_t^{150R_t}.$$ This implies $f_t(t\underline x)=0$, thus by definition $\chi_R(x)=0$. So we have 
$$\chi_R(x)|_{S^{2n-1}(r)\setminus\Sigma_r^{200R_r}}\equiv 0;$$
\item For $x\in\Sigma^{R_r}_r$, if $|\frac{r-t}{R_r}|\leq 2$, i.e., $$(1-2\frac{R_r}{r}) r\leq t\leq (1+2\frac{R_r}{r})r,$$ then $$\Sigma_t^{t\frac{R_r}{r}}\subset\Sigma_t^{(1-2\frac{R_r}{r})^{-N+1}R_t}\subset\Sigma^{50R_t}_t$$ for $0< R \leq \frac{1}{2}(1-50^{-\frac{1}{N-1}})$, thus $t\underline x \in \Sigma^{50R_t}_t$. This implies  $f_t(t\underline x)=1$, thus by definition $\chi_R(x)=1$. So we have 
$$\chi_R|_{\Sigma^{R_r}_r}\equiv 1;$$
\item Denote by $\partial_r$  the unit radial vector field and $r^{-1}\partial_{\theta}$ any unit vector field tangential to the sphere $S^{2n-1}(r)$. Then 
\begin{equation*}
\begin{aligned}
|r^{-1}\partial_\theta \chi_R|
&=|\int_{\mathbb{R}} r^{-1}\partial_\theta (f_t(t\underline x))\cdot \frac{1}{R_r}\phi(\frac{1}{R_r}(r-t))dt|
\\
&\leq C\int_{\mathbb{R}} R_t^{-1} \frac{1}{R_r}\phi(\frac{1}{R_r}(r-t))dt
\\
&\leq C\int_{|r-t|\leq 2R_r} R_t^{-1} R_r^{-1}dt
\\
&\leq C R_r^{-1} 
\end{aligned}
\end{equation*}
and 
\begin{equation*}
\begin{aligned}
|\partial_r \chi_R|
&=
\int_{\mathbb{R}} f_t(t\underline x)\cdot \partial_r(\frac{1}{R_r}\phi(\frac{1}{R_r}(r-t)))dt
\\
&\leq C\int_{|r-t|\leq 2R_r} R_{r}^{-2} dt
\\
&\leq C R_r^{-1}.
\end{aligned}
\end{equation*} 
So $|\nabla \chi_{R}|\leq CR_r^{-1}$;
\item Similarly one can show  $|\nabla^2 \chi_{R}|\leq C R_r^{-2}$.
\end{itemize}
This finishes the proof. 
\end{proof}

Now we  prove Proposition \ref{constructedmetric}.
\begin{proof}[Proof of Proposition \ref{constructedmetric}]
Let $K$ be the constant  given by Proposition \ref{prop2-11}. Denote by $X_0=\C\P^{n-1}\setminus \underline\Sigma$. For any $0<\epsilon'<<1$, let $\underline H_{\epsilon'}$ be the metric on $\underline\E$ satisfying the  properties listed  in Proposition \ref{prop2-11} with $\delta=10^{-4}$, i.e., 
\begin{enumerate}[(1).]
\item $\sup_i\int_{X_0}|\underline\beta_i|_{\omega_{FS}}^2\leq{\epsilon'}$;
\item $\sup_i\int_{X_0}|\Lambda_{\omega_{FS}}\bp_{\underline\E}\underline\beta_i|\leq{\epsilon'}$;
\item $\int_{X_0} |\sqrt{-1}\Lambda_{\omega_{FS}}F_{(\underline H_{\epsilon'},\bp_{\underline S})}-\psi^{\underline H_{\epsilon'}}|\leq{\epsilon'}$;
\item $\|\Lambda_{\omega_{FS}} F_{(\underline H_{\epsilon'},\bp_{\underline\E})}\|_{L^\infty(\C\P^{n-1})}\leq K$;
\item $|\sqrt{-1}\Lambda_{\omega_{FS}}F_{(\underline H_{\epsilon'},\bp_{\underline S})}-\psi^{\underline H_{\epsilon'}}|(z)\leq {\epsilon'}$ for $z\notin \underline\Sigma^{10^{-4}}$;
\item $\sup_i|\underline\beta_i(z)|_{\omega_{FS}} \leq{\epsilon'}$ for $z\notin\underline \Sigma^{10^{-4}}$;
\item $\sup_i|\Lambda_{\omega_{FS}}\p_{\underline\E}\underline\beta_i(z)|\leq{\epsilon'}$ for $z\notin\underline\Sigma^{10^{-4}}$;
\item $\sup_i|F_{(\underline H_{\epsilon'},\bp_{\underline\E})}|\leq K$ for $z\notin\underline\Sigma^{10^{-4}}$.
\end{enumerate}
Denote $H_{\epsilon'}=\pi^*\underline H_{\epsilon'}$. Let $\pi_i: \E \rightarrow \pi^*\underline \E_i$ be the orthogonal projection given by the metric $H_{\epsilon'}$. Fix $N>>\mu_1$, and let $R(N)$ be given by Lemma \ref{cut-off}. Let $R\in (0, R(N)]$ be determined later. Apply Lemma \ref{complexgauge} with $g=\sum^m_{i=1} f_i(\pi_i-\pi_{i-1})$ where 
$$f_i=(1-\chi_R)|z|^{\mu_i}+\chi_R$$ 
and $\chi_R$ is given by Lemma \ref{cut-off}. As in Lemma \ref{complexgauge} we can write 
\begin{equation*}
F_{(H_{\epsilon'}, g\cdot\bp_\E)}=T_0 +T_1+T_2.
\end{equation*}
In the following, we will estimate $T_i$ for $i=0,1,2$ separately and we also use $A\lesssim B$ to denote $A\leq C B$ for some constant $C=C(n,m,N,K,\mu_1)$. For simplicity we introduce one more notation (see Figure \ref{mainfig})
\begin{itemize}
\item $\Sigma'=\bigcup_{r\in (0, 1)} \Sigma_r^{200R_r}$, $\Sigma''=\bigcup_{r\in (0, 1)} \Sigma_r^{R_r}$.
\end{itemize}
\begin{figure}
\begin{tikzpicture}[scale=0.9]
\draw (5, 0) to [out = 90, in = 0] (0, 5);
\draw (0, 5) to [out=180, in=90] (-5, 0);
\draw (-5,0) to [out=270, in=180] (0, -5); 
\draw (0, -5) to [out=0, in=270] (5, 0);
\draw[very thin] (0, 0) to (4.25, 2.68);
\draw[very thin] (0, 0) to (4.25, -2.68);
\draw[very thin] (0, 0) to [out=0, in=235] (4.51, 2.11);
\draw[very thin] (0, 0) to [out=0, in=255] (4.74, 1.68);
\draw[very thin] (0, 0) to [out=0, in=-235] (4.51, -2.11);
\draw[very thin] (0, 0) to [out=0, in=105] (4.74, -1.68);
\draw[thick] (0,0) to (5, 0);
\node at (0, 0) {$\bullet$};
\node at (-2, 0) {$B^*\setminus \Sigma^{10^{-4}}$};
\draw[<-, dotted,  thick] (2, 1.34) to [out=120, in=90] (-2, 0.5);
\draw[<-, dotted,  thick] (2.12, -1.34) to [out=-120, in=-90] (-2, -0.5);
\node at (4.7, 0.5) {$\Sigma'$};
\node at (4.3, -0.2) {$\Sigma''$};
\draw[<-, dotted, thick] (4.2, 1.6) to [out=-45, in=110] (4.65, 0.8);
\draw[<-, dotted, thick] (4.2, -1.6) to [out=45, in=-80] (4.65, 0.2);
\node at (2.5, -3) {$\Sigma$};
\draw[<->, dotted,  thick] (4, 0.6) to [out=-45, in=45] (4, -0.6);
\draw[<-, dotted, thick] (3.5, 0) to  (2.5, -2.8);
\end{tikzpicture}
\caption{The cut-off}
\label{mainfig}
\end{figure}

\begin{enumerate}[(A).]
\item For $T_0= F_{(H_{\epsilon'},g\cdot\bp_S)}=F_{(H_{\epsilon'}, \bp_S)}-\sum_{i} \partial\bp\log (f_i^2) (\pi_i-\pi_{i-1})$. Since $f_i=(1-\chi_R)|z|^{\mu_i}+\chi_R$, we have 
\begin{enumerate}[(a).]
 \item For $z\notin\Sigma'$, 
\begin{equation*}
\begin{aligned}
T_0
&=\pi^*F_{(\underline H_{\epsilon'}, \bp_{\underline S})}-\sum_i \mu_i\p\bp \log |z|^2(\pi_i-\pi_{i-1}), \\
\Lambda_{\omega_0}T_0
&=\frac{\pi^*(\Lambda_{\omega_{FS}}F_{(\underline H_{\epsilon'},\bp_{\underline\S})}+\sqrt{-1}\psi^{\underline H_{\epsilon'}})}{r^2}
\end{aligned}
\end{equation*}
As a result, by Lemma \ref{lem2-16}, we have 
$$|\Lambda_{\omega_0}T_0|=\frac{|\sqrt{-1}\Lambda_{\omega_{FS}}F_{(\underline H_{\epsilon'},\bp_{\underline S})}-\psi^{\underline H_{\epsilon'}}|}{r^2}, $$
and by Corollary \ref{splittingcurvature} 
$$|T_0|_{\omega_0}\lesssim \frac{|F_{(\underline H_{\epsilon'}, \bp_{\underline \E})}|_{\omega_{FS}}+\sup_i|\underline \beta_i|^2_{\omega_{FS}}+1}{r^2}. $$
So 
$$|\Lambda_{\omega}T_0|\lesssim \frac{|\sqrt{-1}\Lambda_{\omega_{FS}}F_{(\underline H_{\epsilon'},\bp_{\underline S})}-\psi^{\underline H_{\epsilon'}}|}{r^2}+|F_{(\underline H_{\epsilon'}, \bp_{\underline \E})}|_{\omega_{FS}}+\sup_i|\underline \beta_i|^2_{\omega_{FS}}+1$$

\item Similarly, for $z\in\Sigma'$, using the fact that $\mu_i$ is strictly decreasing, by Corollary \ref{splittingcurvature} and Lemma \ref{cut-off}, we have 
$$|T_0|_{\omega_0}\lesssim \frac{|F_{(\underline H_{\epsilon'}, \bp_{\underline \E})}|_{\omega_{FS}}+\sup_i |\underline \beta_i|_{\omega_{FS}}^2+R^{-2}_rr^{-2\mu_1}}{r^2}, $$
and
\begin{equation*}
\begin{aligned}
|\Lambda_{\omega}T_0|\lesssim\frac{K+\sup_i|\underline\beta_i|_{{\omega_{FS}}}^2+R^{-2}_rr^{-2\mu_1}}{r^2}.
\end{aligned}
\end{equation*}

\end{enumerate}
Combining (a) and (b), and using Items (1) and (3) above,  we get

\begin{equation*}
\begin{aligned}
&r^{-(2n-1)}\int_{S^{2n-1}(r)}r^2|\Lambda_{\omega}T_0|\\
\lesssim& \int_{X_0} |\sqrt{-1}\Lambda_{\omega_{FS}}F_{(\underline H_{\epsilon'},\bp_{\underline S})}-\psi^{\underline H_{\epsilon'}}|+Cr^2\\
&+r^{-(2n-1)}\int_{\Sigma'\cap S^{2n-1}(r)}(K+\sup_i|\underline\beta_i|_{{\omega_{FS}}}^2+R^{-2}_rr^{-2\mu_1})\\
\lesssim & \epsilon'+R^2+Cr^2, 
\end{aligned}
\end{equation*}
where the last step we used Lemma \ref{cut-off} and the fact that $\underline \Sigma$ has complex codimension at least two. 
By Proposition \ref{prop2-15}, using $\beta_i=\bp_\E \pi_i$, we have
$$
|T_0|_{\omega} \in L^{1+\delta}(B^*)
$$
for some $\delta>0$. Also, for $z\notin  \Sigma^{10^{-4}}$, by Items (5), (6) and (8) above, 
$$r^2|\Lambda_{\omega}T_0|\lesssim {\epsilon'}+Cr^2. $$
\item For $T_1=-(g\cdot\bp_S)(g \beta g^{-1})^*
+(g\cdot \bp_S)^*(g \beta g^{-1})$, by Lemma \ref{complexgauge}, we have 
$$
\begin{aligned}
(g\cdot\bp_S)(g \beta g^{-1})^*=&\sum_{i<j}\frac{f_i}{f_j} (\pi_j-\pi_{j-1})(\bp_S \beta^*)(\pi_i-\pi_{i-1})\\
&-2\sum_{i<j}\bp(\frac{f_i}{f_j})\wedge(\pi_j-\pi_{j-1})(\p_{\E} \pi_i) (\pi_i-\pi_{i-1}).
\end{aligned}
$$
Plugging in $f_i=(1-\chi_R)|z|^{\mu_i}+\chi_R$, we have 
\begin{enumerate}[(a).]
\item If $z\notin\Sigma'$, then $\frac{f_i}{f_j}=|z|^{\mu_i-\mu_j}$ where $\mu_i>\mu_j$. As a result, 
$$
\begin{aligned}
|(g\cdot\bp_S)(g \beta g^{-1})^*|_{\omega}
&\lesssim\frac{\sup_i|\underline\beta_i|_{\omega_{FS}}^2+\sup_i|\p_{\underline\E}\underline\beta_i|_{\omega_{FS}}+r^{1+\mu'}\sup_i|\underline\beta_i|_{\omega_{FS}}}{r^2}\\
&\lesssim \frac{2\sup_i|\underline\beta_i|_{\omega_{FS}}^2+|F_{(\underline H_{\epsilon'},\bp_{\underline\E})}|_{\omega_{FS}}+r^{1+\mu'}\sup_i|\underline\beta_i|_{\omega_{FS}}}{r^2}, 
\end{aligned}
$$
where the second inequality follows from Equation (\ref{sf}). Here $\mu'=\min\{i<j: \mu_i-\mu_j\}$. So by the explicit formula above, we also have
$$|\Lambda_{\omega}(g\cdot\bp_S)(g \beta g^{-1})^*|\lesssim\frac{\sup_i|\underline\beta_i|_{\omega_{FS}}^2+\sup_i|\Lambda_{\omega_{FS}}\p_{\underline\E}\underline\beta_i|+r^{1+\mu'}\sup_i|\underline\beta_i|_{\omega_{FS}}}{r^2}$$
\item If $z\in \Sigma'$, using $|\bp \frac{f_i}{f_j}|\lesssim R_r^{-1}r^{-\mu_1-1}$ by Lemma \ref{cut-off}, we get
$$|\Lambda_{\omega}(g\cdot\bp_S)(g \beta g^{-1})^*|\lesssim\frac{\sup_i|\underline\beta_i|_{\omega_{FS}}^2+\sup_i|\Lambda_{\omega_{FS}}\p_{\underline\E}\underline\beta_i|_{\omega_{FS}}+R_r^{-2}r^{-2\mu_1}}{r^2}$$
and
$$
\begin{aligned}
|(g\cdot\bp_S)(g \beta g^{-1})^*|_{\omega}
&\lesssim\frac{\sup_i|\underline\beta_i|_{\omega_{FS}}^2+\sup_i|\p_{\underline\E}\underline\beta_i|_{\omega_{FS}}+R_r^{-2} r^{-2\mu_1}}{r^2}\\
&\lesssim \frac{2\sup_i|\underline\beta_i|_{\omega_{FS}}^2+|F_{(\underline H,\bp_{\underline\E})}|_{\omega_{FS}}+R_r^{-2} r^{-2\mu_1}}{r^2}
\end{aligned}
$$
\end{enumerate} 
Now combining the estimates in (a) and (b), and using Items (1) and (2), we have 
\begin{equation*}
\begin{aligned}
&r^{-(2n-1)}\int_{S^{2n-1}(r)}r^2|\Lambda_{\omega}T_1|\\
&\lesssim\int_{X_0}\sup_i|\underline\beta_i|_{\omega_{FS}}^2+\sup_i|\Lambda_{\omega_{FS}}\p_{\underline\E}\underline\beta_i|_{\omega_{FS}}+r^{-(2n-1)}\int_{\Sigma'} R_r^{-2}r^{-2\mu_1}\\
&\leq \epsilon'+R^2. 
\end{aligned}
\end{equation*}
Similar to the estimate for $T_0$, using $\beta_i=\bp_\E \pi_i$, by Proposition \ref{prop2-15}, we have 
$$
|T_1|\in L^{1+\delta}(B^*).
$$
 Also, for $z\notin \Sigma^{10^{-4}}$, by Items (6) and (7), 
$$r^2|\Lambda_{\omega}T_1|\leq\sup_i|\underline\beta_i|_{\omega_{FS}}^2+\sup_i|\Lambda_{\omega_{FS}}\p_{\underline\E}\underline\beta_i|+\sup_i r^{1+\mu'}|\underline\beta_i|_{\omega_{FS}}\lesssim \epsilon'.$$
\item  For $T_2=-
g \beta g^{-1}\wedge(g \beta g^{-1})^*
-
(g \beta g^{-1})^*\wedge g \beta g^{-1}$, we have 
\begin{equation}
|T_2|_{\omega}\leq 2|g \beta g^{-1}|_{\omega}^2\lesssim (\sup_{i<j} |\frac{f_i}{f_j}|^2)\sup_i|\beta_i|_{\omega}^2\lesssim \frac{\sup_i |\underline \beta_i|_{\omega_{FS}}^2}{r^2}.
\end{equation}
 Here, the second inequality follows from 
\begin{equation*}
\begin{aligned}
g \beta g^{-1}
&=-\sum_{i,j,k} \frac{f_i}{f_j}(\pi_i-\pi_{i-1})(\pi_k-\pi_{k-1})\bp_\E \pi_k (\pi_j-\pi_{j-1})\\
&=-\sum_{i,j}\frac{f_i}{f_j}(\pi_i-\pi_{i-1})\bp_\E \pi_i (\pi_j-\pi_{j-1})\\
&=-\sum_{i<j}\frac{f_i}{f_j}(\pi_i-\pi_{i-1})\bp_\E \pi_i (\pi_j-\pi_{j-1})
\end{aligned}
\end{equation*}
The last equality follows from $\bp_\E \pi_i \cdot \pi_i=0$ ($\pi_i$'s are all weakly holomorphic). As a result, we have 
$$
r^{-(2n-1)}\int_{S^{2n-1}(r)}r^2|\Lambda_{\omega} T_2|\lesssim {\epsilon'}
$$
and 
$$
|T_2|_{\omega}\in L^{1+\delta}(B^*)
$$
for some $\delta>0$ as $|T_0|$ and $|T_1|$. Also, for $z\notin  \Sigma^{10^{-4}}$, $r^2|\Lambda_{\omega}T_2|\leq{\epsilon'}$.
\end{enumerate}
Now combining (A), (B), (C), we have
$$
r^{-(2n-1)}\int_{S^{2n-1}(r)}r^2|
\Lambda_{\omega} F_{(H_{\epsilon'},g\cdot\bp_\E)}|\lesssim \epsilon'+R^2+cr^2, 
$$
and for $z\notin \Sigma^{10^{-4}}$, 
$$r^2|\Lambda_{\omega}F_{(H_{\epsilon', g\cdot\bp_\E)}}|\lesssim\epsilon'. $$
 Since $|F_{(H_{\epsilon'},g\cdot\bp_\E)}|\leq |T_0|+|T_1|+|T_2|$, we have $|F_{(H_{\epsilon'},g\cdot\bp_\E)}|\in L^{1+\delta}(B^*)$ for some $\delta>0$. For any $0<{\epsilon}<<1$, choose ${\epsilon'}$ and $R$ small so that $\epsilon'+R^2<<\epsilon$ and let $H_\epsilon=H_{\epsilon'}(g\cdot, g\cdot)$, i.e. 
\begin{equation}
H_\epsilon=\sum_i (1-\chi_R)|z|^{2\mu_i}H_{\epsilon'}((\pi_i-\pi_{i-1})\cdot,(\pi_i-\pi_{i-1})\cdot)+\chi_R H_{\epsilon'}.
\end{equation}
The calculation above shows that $H_\epsilon$ satisfies $(i)$, $(ii)$ and $(iii)$. It suffices to verify that $H_\epsilon$ satisfies $(iv)$. For any $s\in HG_i\setminus HG_{i-1}$, we have $(\pi_i-\pi_{i-1})s\neq 0$, so 
\begin{equation*}
\begin{aligned}
&\int_{B_r^*\setminus  \Sigma^{10^{-3}}}|z|^{\epsilon}|s|^2_{H_\epsilon}\\
=&\int_{B_r^*\setminus  \Sigma^{10^{-3}}} \sum_{j\leq i} |z|^{\epsilon+2\mu_j}H_{\epsilon'}((\pi_j-\pi_{j-1})s,(\pi_j-\pi_{j-1})s )\\
=&\sum_{j\leq i} a_j r^{\epsilon+2\mu_i+2n}
\end{aligned}
\end{equation*}
where $a_i\neq 0$. Thus by taking limit  $r\rightarrow 0$, we have 
$$
\lim_{r\rightarrow 0} \frac{1}{2}\frac{\log \int_{B^*_r\setminus \Sigma^{10^{-3}}}|z|^{\epsilon}|s|_{H_\epsilon}^2}{\log r}-n=\mu_i+\frac{\epsilon}{2}.
$$
This finishes the proof.
\end{proof}

 \subsection{Proof of Theorem \ref{thm2}} \label{Section3-3}
Going back to the general setting in the introduction, we let $A$ be an admissible Hermtian-Yang-Mills connection on $B$ with vanishing Einstein constant, and with an isolated singularity at $0$. Let $\E$ be the corresponding reflexive sheaf. In this section, we shall use the following notations
\begin{itemize}
\item $p: \widehat{B}\rightarrow B$ denotes the blow-up of $B$ at $0$.  We can identify $\widehat B$ naturally with an open neighborhood of the zero section in the total space of the line bundle $\mathcal O(-1)\rightarrow \C\P^{n-1}$; 
\item $i:  p^{-1}(0)\simeq \C\P^{n-1}\rightarrow\widehat{B}$ denotes the obvious inclusion map;
\item $\phi: \widehat B\rightarrow \mathbb{\C\P}^{n-1}$ denotes  the restriction of the projection map $\mathcal{O}(-1)\rightarrow\mathbb{\C\P}^{n-1}$;
\end{itemize}

 We first introduce the following definition, which has been informally mentioned in the Introduction.

\begin{defi}
An algebraic tangent cone of $\E$ at $0$ is a coherent sheaf on $\C\P^{n-1}$ which is given by the restriction of a reflexive sheaf $\F$ on $\widehat B$, such that $\F|_{\widehat B\setminus p^{-1}(0)}$ is isomorphic to $p^*(\E|_{B\setminus \{0\}})$.
\end{defi}

For the convenience of the reader we digress to discuss the notion of restriction of reflexive coherent analytic sheaves. The corresponding theory in the category of algebraic geometry is well-known. For a coherent sheaf $\F$ on a complex manifold $X$ and a smooth divisor $D$, we denote by $\F|_D=i_D^*\F$ to be the pull-back of $\F$ to  $D$ as a coherent sheaf, i.e. $\F|_D=i^{-1}_D\F\otimes_{i^{-1}_D\O_X} \O_D$ where $i_D: D \hookrightarrow X$ denotes the inclusion map and $i_D^{-1}$ denotes the inverse image functor.

\begin{lem} \label{lem3-24}
If $\F$ is reflexive then $\F|_D$ is torsion free. 
\end{lem}
\begin{proof}
It suffices to prove the stalk of $\F|_D$ at any point $p$ is torsion free. For this purpose we can work in the local holomorphic coordinates $z_1, z_2, \cdots, z_n$ centered $p$ and assume $D$ is locally given by $\{z_1=0\}$. Since $\F$ is reflexive we can find a local short exact sequence in a neighborhood $U$ of $p$ of the form 
$$0\rightarrow \F\rightarrow \O^{n_1}_U\xrightarrow {\phi}\O_U^{n_2}. $$This  can be achieved, for example, by first choosing a locally free resolution of $\F^*$ and then taking dual. Suppose there exists $s\in (\F|_D)_{p}$ with $f\cdot s=0$. By definition we can write $s=[\eta]$ for an element $\eta$ of $\F_p$. Then $f\cdot s=0$ implies that there is a nonzero element $\lambda \in \F_p$ such that $f\eta=z_1\lambda$. Using the above short exact sequence we can view both $\eta$ and $\lambda$ as elements of $(\O_{X}^{n_1})_p$, which implies that $\eta=z_1\eta'$ for some $\eta'$ in $(\O_X^{n_1})_p$. Since $\phi(\eta)=\phi(z_1\eta')=z_1\phi(\eta')=0$, we know $\phi(\eta')=0$ i.e. $\eta'\in \F_p$ which implies $\eta=z_1\eta'$ and so $s=0$. This finishes the proof. 
\end{proof}

\begin{rmk}
Another proof (pointed to us by the referee) can be done by noting $Im(\phi)$ is torsion free and so $Tor_1(\O_D, Im(\phi))=0$. This will imply that the restriction map $\F|_D \rightarrow \O_D^{n_1}$ is still injective. In particular, $\F|_D$ is torsion free. 
\end{rmk}

\begin{rmk}\label{rmk3.26}
Using this we can give an alternative description of $\F|_D$, as the subsheaf of $\O_{D}^{n_1}$ generated by the restriction of local holomorphic sections of $\F$, viewed as sections of $\O_{X}^{n_1}$. Indeed, if we denote the latter sheaf by $\F'$, then there is an obvious surjective homomorphism $\psi$ from $\F|_D$ to $\F'$. Notice since the quotient $\O_X^{n_1}/\F$ is torsion free, we know $\F$ is locally free   outside a codimension two complex analytic subvariety of $X$ and the map $\F\rightarrow \O_X^{n_1}$ realizes $\F$ as a sub-bundle of the trivial bundle. So outside a divisor in $D$ the restriction $\F|_D$ is locally free and $\psi$ is an isomorphism. Thus  the kernel of $\psi$ is necessarily torsion and the above Lemma implies it has to be zero. This description is often helpful when dealing with explicit examples. For example, if $\F$ is the reflexive sheaf on $\C^3$ defined as the kernel of the map $\O^{\oplus 3}_{\C^3}\rightarrow\O_{\C^3}$ given by dot product with $(z_1, z_2, z_3)$, and $D$ is the divisor $\C^2=\{z_1=0\}$, then $\F|_D$ is isomorphic to $\I_0\oplus \O_{\C^2}$, where $\I_0$ is the ideal sheaf of the origin in $\C^2$. Indeed, it is easy to see using the exactness of Koszul complex that, $\F$ is the subsheaf of $\O_{\C^3}^{\oplus 3}$ globally generated by three holomorphic sections $(z_2, -z_1, 0)$, $(z_3, 0, -z_1)$, and $(0, z_3, -z_2)$.  Restricting to $D=\C^2$ these become sections of $\O_{\C^2}^{\oplus 3}$ given by $(z_2, 0, 0)$, $(z_3, 0, 0)$, and $(0, z_3, -z_2)$. The first two sections generate the sheaf $\I_0$ and the last section generates a sheaf isomorphic to $\O_{\C^2}$. 
\end{rmk}
\begin{rmk}
Lemma \ref{lem3-24} is not true if $\F$ is only torsion-free. For example, it is easy to see that for the ideal sheaf $\I_0$ of the origin in $\C^2$, the restriction to a line $\C$ through the origin indeed has torsion. 
\end{rmk}

Lemma \ref{lem3-24} implies that an algebraic tangent cone is always torsion free. It is also easy to see an algebraic tangent cone always exists. For instance, one way to obtain an algebraic tangent cone is by first taking $(p^*\E^*)^*$, and then restrict to $p^{-1}(0)$. We will show how to calculate this algebraic tangent cone by examples in Section \ref{section3-4}.  

\

As mentioned in the Introduction, the motivation for introducing the notion  of an algebraic tangent cone is to extend Theorem \ref{main} to the general case of non-homogeneous singularities (we will see examples of these in Section \ref{section3-4}), and as a first step we want to find an object playing the role of $\underline \E$ which serves as a good homogeneous approximation of the singularity. Notice even though an algebraic tangent cone always exists, it is far from being unique since it involves  extending a coherent sheaf across a  complex co-dimension one subset. In \cite{CS3} we will make a systematic algebro-geometric study on the set of algebraic tangent cones. In particular the notion of an \emph{optimal} algebraic tangent cone is introduced. Here we briefly summarize the key points and the interested readers are referred to \cite{CS3} for details.
Given an algebraic tangent cone $\underline{\widehat\E}$, and let $0=\underline{\E}_0\subset \underline{\E}_1\subset \cdots \underline{\E}_m=\underline{\widehat\E}$ be its Harder-Narasimhan filtration, then $\underline{\widehat{\E}}$ is said to be optimal if the difference of slopes $\mu(\underline \E_m /\underline \E_{m-1})-\mu(\underline{\E}_1)\in [0, 1)$. Obviously a semistable algebraic tangent cone is always optimal, but semistability is not always achievable and being optimal is the next best to hope for. In \cite{CS3} it is proved that an optimal tangent cone always exists and it is unique up to certain natural transforms (which includes in particular twisting by some $\O(k)$); more crucially, the isomorphism class of the corresponding torsion-free sheaf $\iota_*\pi^*(\Gr^{HNS}(\underline {\widehat \E})) $  (notice there are notational differences between this paper and \cite{CS3}) on $\C^n$ does not depend on the specific choice of the optimal algebraic tangent cone. With these preparations we can now state the following conjecture which gives a fairly complete picture for describing the tangent cones of admissible Hermitian-Yang-Mills connections in terms of the algebraic geometry of the underlying reflexive sheaf. 

\begin{conj} \label{conj1-3}
Given any reflexive sheaf $\E$ on $B$ with an isolated singularity at $0$, for any admissible Hermitian-Yang-Mills connection $A$ on $\E$ there is a unique tangent cone $A_\infty$ at $0$, whose corresponding sheaf $\E_\infty$ is isomorphic to $ \iota_*\pi^*(\Gr^{HNS}(\underline {\widehat \E}))^{**}$ where $\underline {\widehat \E}$ is any optimal algebraic tangent cone of $\E$ at $0$, and $A_\infty$ is gauge equivalent to the natural Hermitian-Yang-Mills cone that is induced by the admissible Hermitian-Yang-Mills connection on  $(\Gr^{HNS}(\underline{\widehat \E}))^{**}$. 
\end{conj} 

\begin{rmk}
\begin{itemize}
\item In \cite{CS3} it is also proved that when $\E={(\iota_B)}_*\pi_B^*\underline\E $ is homogeneous, for any optimal algebraic tangent cone $\underline{\widehat\E}$, $ \iota_*\pi^*(\Gr^{HNS}(\underline {\widehat\E}))$ is isomorphic to  $ \iota_*\pi^*(\Gr^{HNS}(\underline \E))$. Hence Theorem \ref{main} confirms Conjecture \ref{conj1-3} in the homogeneous case under the extra assumption that $Gr^{HNS}(\underline\E)$ is reflexive. Also Theorem \ref{thm2} confirms Conjecture \ref{conj1-3} in the case when there is an algebraic tangent cone which is locally free and stable (hence is optimal). 
\item In \cite{CS2} we shall prove Conjecture \ref{conj1-3} in the homogeneous case without extra assumptions, and we also strengthen the conjecture by relating the analytic bubbling information and the non-reflexive locus of  $ \iota_*\pi^*(\Gr^{HNS}(\underline{\widehat \E}))$.
\item One  expects that a similar statement holds even if $\E$ has a non-isolated singularity. Notice the existence of tangent cones for admissible HYM connections is already known in general (c.f. Section 5 in \cite{Tian}). Also the algebro-geometric results in \cite{CS3} are proved for general singularities. The proof of Conjecture \ref{conj1-3} in the general case of non-isolated singularities will suffice for understanding the singularities of global admissible Hermitian-Yang-Mills connections on compact K\"ahler manifolds. One may also ask whether any reflexive sheaf over a ball $B$ admits an admissible Hermitian-Yang-Mills connection, and it is reasonable to expect the answer is yes. In the case of isolated singularities, one may try to extend the result of Donaldson \cite{Donaldson3} to solve a Dirichlet boundary value problem. The key point is that the maximum principle reduces the key $C^0$ estimate to the boundary, and this replaces the use of the slope stability in the compact case.  We leave this to future study. 
\end{itemize}

\end{rmk}

From now on we assume that  there is an algebraic tangent cone $\underline{\widehat \E}$ which is locally free (i.e. defines a holomorphic vector bundle) on $\C\P^{n-1}$.  Denote by $\widehat \E$ the reflexive sheaf on $\widehat B$ that restricts to $\underline{\widehat \E}$ on $p^{-1}(0)$ and is isomorphic to the pull-back of $\E$ outside $p^{-1}(0)$. By the discussion in Remark \ref{rmk3.26} we see that $\widehat\E$ itself is also locally free. 

\begin{prop}\label{holomorphicsectioninverse}
For $k$ large, the natural map 
$$r: H^{0}(\widehat{B}, \widehat{\E}\otimes\phi^*\O(k)) \rightarrow H^{0}(\C\P^{n-1}, \underline{\widehat{\E}}\otimes\mathcal{O}(k))$$ 
is surjective. 
\end{prop}
\begin{proof}
By the vanishing theorem of A. Fujiki (see Theorem $N'$ in \cite{Fujiki}), we know $H^1(\hat{B}, \widehat{\E} \otimes \phi^* \O(k+1))$ when $k$ is large. In particular, we know $r$ is surjective.
\end{proof}

Now we fix $k$ large given by the above Proposition, replace $\widehat{\E}$ by $\widehat{\E}\otimes \phi^*\mathcal O(k)$, and assume $r: H^0(\widehat B, \widehat\E)\rightarrow H^0(\C\P^{n-1}, \underline{\widehat \E})$ is surjective. We may assume $\widehat{\underline \E}$ is globally generated on $\C\P^{n-1}$. Notice since $\E$ is reflexive, there is a natural map $\phi_*: H^0(\widehat B, \widehat \E)\rightarrow H^0(B, \E)$. Denote by $HG$ the image of $\phi_*$.

\begin{prop} \label{prop3-29}
Suppose $\underline{\widehat \E}$ is a semistable vector bundle on $\C\P^{n-1}$, then for any $s\in HG\setminus\{0\}$, we have $d(s)=\mu(\underline{\widehat \E})$. 

\end{prop}

\begin{lem}\label{lem3-31}
There exists a constant $C>0$ such that for any $\epsilon>0$ there is a smooth Hermitian metric $H_\epsilon$ on $\E$ over $B^*$ with the following properties
\begin{enumerate}[(1).]
\item $\int_{B^*} |F_{(H_\epsilon,\bp_\E)}|^{2}<\infty$; 
\item $|z|^2|\Lambda_{\omega}F_{(H_\epsilon,\bp_\E)}(z)|\leq \epsilon+C|z|$ for all $z\in B^*$; 
\item For all $s\in HG\setminus\{0\}$, 
$$\frac{1}{2}\lim_{r\rightarrow 0} \frac{\log \int_{B_r^*} |s|_{H_\epsilon}^2}{ \log r}-n=\mu(\underline{\widehat \E}).$$
\item Given a holomorphic section $s$ of $\widehat{\E}$ over $\widehat{B}$, if $s(0,[z])=0$ for some $z\in B^*$, then viewed as a section of $\E$ over $B^*$.
$$
|s(t.z)|_{H_\epsilon}\leq C |tz|^{\mu+1}
$$
for some constant $C>0$ and for any $t\in \C^*$ with $|t|$ small.

\end{enumerate}
\end{lem}
Assuming this, applying Corollary \ref{cor2-26} we obtain that 
\begin{equation} \label{eqn3-13}
C|z|^{\epsilon} H_\epsilon \leq H\leq C|z|^{-\epsilon}H_\epsilon,  
\end{equation}
and Proposition \ref{prop3-29} follows easily. 

\begin{proof}[Proof of Lemma \ref{lem3-31}]
As in Section \ref{semistable}, for any $\epsilon>0$ we can find a Hermitian metric $\underline H_\epsilon$ on $\underline{\widehat \E}$ such that $|\sqrt{-1}\Lambda_{\omega_{FS}}F_{\underline A_\epsilon}-\mu\Id|_{L^\infty}<\epsilon$ with $\mu=\mu(\underline{\widehat \E})$. Pulling back to $\widehat B$ by the map $\phi$, we get a Hermitian metric $H_\epsilon'$ on $\E':=\phi^*(\underline{\widehat \E})$. Now by our assumption we know that $\widehat\E$ is also a vector bundle and it is isomorphic to $\E'$ as smooth complex vector bundles. Fixing any smooth isomorphism between these two which restricts to the natural identity map on $\widehat{\underline \E}$ over the exceptional divisor $\C\P^{n-1}$, we may then view $H_\epsilon'$ naturally as a Hermitian metric on $\widehat\E$ too. Through this isomorphism we write $\beta=\bp_{\widehat\E}-\bp_{\E'}$. In particular, the tangential component of the restriction of $\beta$ to $\C\P^{n-1}$ is zero. A direct computation shows 

\begin{itemize}
\item $|\beta|_{\pi^*\omega}\leq C$;
\item $|\p_{\E'}\beta|_{\pi^*\omega}\leq C|z|^{-1}$.
\end{itemize}
So 
$$|F_{( H_\epsilon', \bp_{\widehat\E})}-F_{(H_\epsilon', \bp_{\E'})}|_{\pi^*\omega}\leq C|z|^{-1}. $$
Now let $H_\epsilon=|z|^{2\mu}H_\epsilon'$ and using the map $p$ we obtain a corresponding Hermitian metric on $\E|_{B^*}$, which we still denote by $H_\epsilon$. Then it is clear that (1) and (2) hold. (3) follows from the fact that there exists $C$ independent of $r$ so that 
$$
C^{-1} r^{2n-1+2\mu}\leq \int_{\p B_r} |s|^2_{H_\epsilon} \leq C r^{2n-1+2\mu}
$$
for $r$ small. (4) follows similarly. 
\end{proof}

Now we prove Theorem \ref{thm2}. The idea is similar to that has been previously used in Section \ref{semistable}. Let $(\E_\infty, A_\infty)$ be a tangent cone of $A$ at $0$.  We can build a homogeneous homomorphism $\tau: {\iota_B}_*\pi_B^*\underline{\widehat \E}\rightarrow \E_\infty$ as follows. Fix a subspace $V$ of $H^0(\widehat B, \widehat \E)$ such that $r: V\rightarrow H^0(\C\P^{n-1}, \widehat{\underline \E})$ is an isomorphism and we identify $V$ with a subspace of $H^0(B, \E)$ using the map $\phi_*$.  Choose a basis $s_i$ of $H^0(\C\P^{n-1}, \widehat{\underline \E})$ and correspondingly a basis $\sigma_i$ of $V$. By Proposition \ref{prop3-29} and by passing to a subsequence we may assume $\sigma_i$ converges to homogeneous holomorphic sections $\sigma_{i, \infty}$ of $\E_\infty$. Let $M_j=\sup_{i}\|\sigma_i\|_{j}$, and let $\sigma'_{i, \infty}$ be the limit of $\frac{1}{M_j}\sigma_{i}$. Then $\sigma'_{i, \infty}$ is either zero or homogeneous of degree $\mu$ and there is at least one $i$ such that $\sigma'_{i, \infty}$ is non-zero. 

 For any $z\in B^*$, and any $\eta$ on the fiber $\pi^*\underline{\widehat \E}|_{z}$, we may write $\eta=\sum_i a_i\pi^*s_i(\pi(z))$. Then we define $\tau(\eta)$ to be $\sum_i a_i \sigma'_{i, \infty}(z)$. To see this is well-defined suppose a section $s=\sum_i a_is_i\in H^0(\C\P^{n-1}, \underline{\widehat \E})$ vanishes at $\pi(z)$, then we need to show the corresponding limit section $\sum_i a_i \sigma'_{i, \infty}$ vanishes at $z$. This follows from (\ref{eqn3-13}): let $\sigma=\sum_i a_i \sigma_i$, by (4) in Lemma \ref{lem3-31} $|\sigma(tz)|_{H_\epsilon}\leq C|tz|^{\mu+1}$, hence 
$$|\sigma(tz)|_{H}\leq C|tz|^{\mu+1-\epsilon/2}. $$

On the other hand since $d(\sigma_i)=\mu$ for all $i$ we have
$$M_j \geq C2^{-j(\mu+\epsilon/2)}. $$
If we have chosen a priori that $\epsilon$ is sufficiently small then we know the corresponding limit of $\frac{1}{M_j}|\sigma(2^{-j}z)|$ is zero.
 
Now by analogy with the previous discussion it is easy to see $\tau$ is indeed a non-trivial homogeneous homomorphism. As before using the stability of $\underline{\widehat \E}$ one can conclude that $\E_\infty$ is a simple HYM cone with holonomy $e^{-2\pi\sqrt{-1}\mu}$, and $\tau$ induces an isomorphism between $\underline{\widehat \E}$ and $\underline\E_\infty$. This finishes the proof of Theorem \ref{thm2}. 

\subsection{Discussion and examples}\label{section3-4}
 Let $\E$ be a reflexive sheaf defined over the $n$-dimensional ball $B\subset \C^n$ with a (not necessarily isolated) singularity at $0$. 

\begin{defi} \label{defi3-31}
We say $0$ is a \emph{homogeneous} singularity of $\E$  if there is a reflexive sheaf $\underline \E$ over $\C\P^{n-1}$ such that $\E$ is isomorphic to ${\iota_U}_*\pi_U^*\underline \E$ on some neighborhood $U$ of $0$.
\end{defi}

 We briefly recall the notion of \emph{Fitting invariants}, following \cite{DE}.  Choose a finitely generated free presentation of the stalk $\E_0$ 
$$ \F\xrightarrow{\phi} \G\rightarrow \E_0\rightarrow 0,  $$
we define the $j$-th \emph{Fitting ideal} $\text{Fitt}_j(\E, 0)$ of $\E$ at $0$ to be the ideal of $\O_0$ given by the image of the $\O_0$-module homomorphism 
$$\Lambda^{b-j}\phi: (\Lambda^{b-j} \G)^*\otimes (\Lambda^{b-j}\F)\rightarrow \O_0,  $$
where $b=rank(\G)$.
If we identify $\F$ with $\O_0^{\oplus a}$ and $\G$ with $\O_0^{\oplus b}$, and represent $\phi$ by a $\O_0$-valued matrix, then $\text{Fitt}_j(\E, 0)$ is the ideal of $\O_0$ generated by all the $(b-j)\times (b-j)$ minors of the matrix. We make the convention that $\text{Fitt}_{j}(\E, 0)=\O_0$ if $j\geq b$. It is not hard to show (see for example \cite{DE}, Chapter 20) that for all $j$, $\text{Fitt}_j(\E, 0)$ is a well-defined invariant of the stalk $\E_0$, i.e. it does not depend on the choice of the particular presentation. 

The following is pointed out to us by Professor Jason Starr.

\begin{prop}\label{prop3.34}
If $\E$ is homogeneous at $0$ then all the corresponding Fitting ideals $\text{Fitt}_j(\E, 0)$ must be homogeneous ideals of $\O_0$. 
\end{prop}
\begin{rmk}
Here we say an ideal $\I$ of $\O_0$ is \emph{homogeneous} if it is generated by homogeneous polynomials; it is not hard to see that for a homogeneous ideal $\I$, if a function $f$ belongs to $\I$, then all the homogeneous components in the Taylor expansion of $f$ at $0$ also belong to $\I$. 
\end{rmk}
\begin{proof}
Since the claimed property only depends on the local structure of $\E$ near $0$, we may assume without loss of generality that on $B\setminus\{0\}$,  $\E$ is isomorphic to ${\iota_B}_*\pi_B^*\underline \E$ for some reflexive sheaf $\underline \E$ on $\C\P^{n-1}$. Let $l_0$ be the smallest $l$ such that $H^0(\C\P^{n-1}, \underline\E(l))\neq 0$, and choose $l_1$ such that the maps $H^0(\C\P^{n-1}, \underline \E\otimes \O(l))\otimes H^0(\C\P^{n-1}, \O(1))\rightarrow H^0(\C\P^{n-1}, \underline \E\otimes \O(l+1))$ are surjective for all $l\geq l_1$. Choosing a basis of $H^0(\C\P^{n-1}, \underline \E\otimes \O(l))$ for all $l\in[l_0, l_1]$ we obtain a surjective homomorphism 
$$\underline \phi: \underline \F:= \bigoplus_{l_0\leq l\leq l_1} \O(-l)^{\oplus n_l}\rightarrow \underline \E, $$
where $n_l=\dim H^0(\C\P^{n-1}, \underline \E\otimes \O(l))$. Pulling-back to $B\setminus \{0\}$ and pushing forward to $B$, we obtain the corresponding map 
$$\phi: \F\simeq \bigoplus_{l_0\leq l\leq l_1}\O_{B}^{\oplus n_l}\rightarrow \E $$
We claim $\phi$ is surjective at $0$. To see this we first notice that by definition $\phi$ is surjective on $B\setminus \{0\}$,  so the sheaf $\E/\text{Im}(\phi)$ is a torsion sheaf supported at the origin, hence there is an $m\geq 1$ such that $\text{Im}(\phi)$ contains $\I_{0}^m \E$, where $\I_0$ is the ideal sheaf of $0$. Similar to the proof of Lemma \ref{lem3-4} (notice we did not use the HYM condition there), we know any local section $s$ of $\E$ can be written as a Taylor series 
$s=\sum_{j\geq l_0} \pi^*s_j$, where $s_j$ is a holomorphic section of $\underline{\E}\otimes \O(j)$, and we have used the natural identification $\pi^*(\O(-1))\simeq \O_{B\setminus \{0\}}$. Now by our choice of $l_1$ and $m$ it follows that $\pi^*s_j$ is a section of $\text{Im}(\phi)$ for $j\leq l_1+m$, and $\sum_{j\geq l_1+m}\pi^*s_j$ defines a germ of a section $\text{Im}(\phi)$ in a neighborhood of $0$. This proves the claim.  

By analogy with the previous discussion, we get a locally free presentation of $\underline\E$
$$\underline\G\xrightarrow{\underline \psi}\underline \F\xrightarrow{\underline \phi}\underline\E\rightarrow 0, $$
where $\underline\G$ is also given by a direct sum of line bundles on $\C\P^{n-1}$, and the map $\psi$ is then represented by a matrix of homogeneous polynomials. Hence it also induces a corresponding locally free presentation of  $\E$
$$\G\xrightarrow{\psi}\F\xrightarrow{\phi}\E\rightarrow 0. $$
Then the conclusion follows directly.
\end{proof}

Using Proposition \ref{prop3.34}, one can easily find explicit examples of  reflexive sheaves with non-homogeneous singularities.
Now we consider the case $n=3$. Let $\E_f$ be given by the short exact sequence
\begin{equation}\label{reflexivesheafexample}
0\rightarrow\O_{B}\xrightarrow f \O_{B}^{\oplus 3}\rightarrow \E_f \rightarrow 0,
\end{equation} 
where $f=(f_1, f_2, f_3)$ is a triple of holomorphic functions defined over $B$. We assume $0$ is an isolated zero of $f$. Then $\E_f$ is a rank two reflexive sheaf in a neighborhood of $0$,  by the Remark after Example $1.1.13$ on Page $77$, \cite{CMH}.  $\E_f$ has an isolated singularity at $0$ and it follows from definition that 
$\text{Fitt}_2(\E_f, 0)$ is the ideal of $\O_0$ generated by $f_1, f_2, f_3$. So if $f_1, f_2, f_3$ do not generate a homogeneous ideal then the corresponding $\E_f$ is not homogeneous.

As mentioned before, one possible algebraic tangent cone is given by $(p^*\E_f^*)^{*}\otimes \O_D$, which we denote by $\widehat{\underline{\E_f}}$. Similarly, we define $\widehat{\E^*_f}=(p^*\E_f)^{*}\otimes \O_D$ We will show that $\widehat{\underline {\E_f}}$ can be explicitly calculated under suitable assumption on $f$. 
\begin{lem}\label{lem3.36}
$\E_f\cong\E_f^*$ in a neighborhood of $0$. 
\end{lem}
\begin{proof}
This is actually true for any rank $2$ reflexive sheaf $\F$ over  $B$ where $B$ is the unit ball centered at $0$ in $\C^n$. Indeed, we know $\det(\F)$ is a holomorphic line bundle over $B$ which has to be trivial. Let $\Theta$ be a global trivialization of $\det(\F)$ over $B$. Then one can naturally define an isomorphism $\F^* \rightarrow \F$ outside $Sing(\F)$ by 
$$
v\in \F^*(x) \longmapsto i_v \Theta\in \F(x)
$$
where $i_v$ denotes the contraction with $v$. By using $\mathcal{H}om(\F^*, \F)$ and $\mathcal{H}om(\F, \F^*)$ are both reflexive, we know that the map above can be extended to be an isomorphism between $\F^*$ and $\F$. 
\end{proof}

Using Lemma \ref{lem3.36} and the fact that the pull-back functor is right-exact (See for example Page 7 in \cite{F}), we obtain 
$$\O_{\hat B}\xrightarrow{p^*f}\O_{\hat B}^{\oplus 3}\rightarrow p^*\E_f^*\rightarrow 0. $$
Taking dual we get
$$0\rightarrow (p^*\E_f^*)^*\rightarrow\O_{\hat B}^{\oplus 3}\xrightarrow{(p^*f)^t}\O_{\hat B},$$
where $(p^*f)^t$ denotes the transpose of $p^*f$.

Given $f=(f_1,f_2,f_3)$, we define a new triple $\widehat f=(\widehat f_1, \widehat f_2, \widehat f_3)$ as follows. Let $g_i$ be the homogeneous part of $f_i$ which has the lowest degree in the Taylor expansion of $f_i$ at $0$ and let  
$d$ be the smallest degree among the degrees of $g_1, g_2, g_3$. Then we define $\widehat f_i=g_i$ if the degree of $g_i$ is $d$ and $\widehat f_i=0$ otherwise.  Denote by $D=\C\P^{2}$ the exceptional divisor of the map $p: \widehat B\rightarrow B$. Then we can naturally view $p^*f_i$, $i=1, 2, 3$ as  holomorphic sections of $\O_{\hat B}(-d\cdot D)$, and view $\hat f_i$ as sections of $\O_{\C\P^2}(d)$ over $D=\C\P^2$.

\begin{lem} \label{lem3-36}
Suppose the common zero set of $\hat{f}_1, \hat{f}_2,\widehat{f}_3$ consists of finitely many points in $D$, then $\underline{\widehat{\E_f}}\otimes \O_{\C\P^2}(d)=\underline{\widehat{\E_f^*}}\otimes \O_{\C\P^2}(d)$ is the subsheaf of $\O_{\C\P^2}(d)^{\oplus 3}$ generated by the global sections  $(\widehat{f}_2, -\widehat{f}_1,0), (\widehat{f}_3,0,-\widehat{f}_1)$, and $(0,\widehat{f}_3, -\widehat{f}_2)$.  
\end{lem}

\begin{proof}
By definition $\widehat{\E_f^*}$ is given by the kernel of the map $p^*f:\O_{\widehat B}^{\oplus 3}\rightarrow \O_{\widehat B}$. We work in the local coordinate chart $\{z_1, w_2, \cdots, w_n\}$ of $\widehat B$ so that the map $p$ is given by $(z_1, w_2, \cdots, w_n)\mapsto (z_1, z_1w_2, \cdots, z_1w_n)$. Observe the restriction of the function $z_1^{-d}p^*f_i$ to $U\cap D$ equals $l_1^{-\otimes d}\cdot \widehat f_i$, where $l_1$ is the local trivialization of $\O_{\C\P^2}(d)$ on $U$ induced by $z_1$ and as above we view $\widehat f_i$ as a holomorphic section of $\O_{\C\P^2}(d)$ over $D$. So from the assumption  we know on $U$, the common zero set of $z_1^{-d}p^*f_1, z_1^{-d}p^*f_2, z_1^{-d}p^*f_3$ consists of finitely many points in $D$. Then by Page 687 in \cite{GH} it follows from the exactness of the Koszul complex of $(z_1^{-d}p^*f_1, z_1^{-d}p^*f_2, z_1^{-d}p^*f_3)$ that  
$\widehat{\E_f^*}$ on $U$ is the subsheaf of $\O_U^{\oplus 3}$ generated by the sections $z_1^{-d}(p^*f_2, -p^*f_1, 0)$, $z_1^{-d}(p^*f_3, 0, -p^*f_1)$, and $z_1^{-d}(0, p^*f_3, -p^*f_2)$. By Remark \ref{rmk3.26} we know $\underline{\widehat{\E_f^*}}$ on $U\cap D$ is the subsheaf of $\O_{U\cap D}^{\oplus 3}$ generated by $l_1^{-\otimes d}(\widehat{f}_2, -\widehat{f}_1,0)$, $l_1^{-\otimes d}(\widehat{f}_3,0,-\widehat{f}_1)$, and $l_1^{-\otimes d}(0,\widehat{f}_3, -\widehat{f}_2)$. Tensoring with $l_1^{\otimes d}$ it follows that $\underline{\widehat{\E_f^*}}\otimes \O_{\C\P^2}(d)$ on $U$ is the subsheaf of $\O_{\C\P^2}(d)^{\oplus 3}$ generated by $(\widehat{f}_2, -\widehat{f}_1,0), (\widehat{f}_3,0,-\widehat{f}_1)$, and $(0,\widehat{f}_3, -\widehat{f}_2)$. From this the conclusion follows. 
\end{proof}
\

\begin{cor}\label{cor3.38}
If the common zero set of $\widehat{f}_1, \widehat{f}_2,\widehat{f}_3$ is empty in $D$, then $\underline{\widehat{\E}_f}$ is a stable vector bundle on $\C\P^2$.
\end{cor}

\begin{proof}
Since slope stability is preserved under taking dual, it suffices to show that $(\underline{\widehat{\E}_f})^*$ is stable. 
 By Lemma \ref{lem3-36} and the assumption, we know $\underline{\widehat{\E_f}}$ is locally free and we have the exact sequence of locally free sheaves 
$$
0\rightarrow\underline{\widehat{\E_f}}\otimes \O_{\C\P^2}(d)\rightarrow \O_{\C\P^2}(d)^{\oplus 3}\xrightarrow{(\widehat f_1, \widehat f_2, \widehat f_3)} \O_{\C\P^2}(2d)\rightarrow 0$$
We first assume $d=2k$ for some $k\in \mathbb{Z}_{+}$. Then we know $c_{1}((\underline{\widehat{\E}_f})^*\otimes \O_{\C\P^2}(-k))=0$ and by taking the dual of the previous sequence we get the exact sequence 
$$
0\rightarrow \O_{\C\P^2}(-3k) \rightarrow (\O_{\C\P^2}(-k))^{\oplus 3} \rightarrow (\underline{\widehat{\E}_f})^*\otimes \O_{\C\P^2}(-k) \rightarrow 0.
$$
By Kodaira vanishing theorem we know $H^1(\C\P^2, \O_{\C\P^2}(-3k))=0$, so we obtain $H^0(\C\P^2, (\underline{\widehat{\E}_f})^*\otimes \O_{\C\P^2}(-k))=0$ and the stability follows (see page 84 in \cite{CMH}). When $k$ is odd, the arguments are similar.
\end{proof}

We finish this section with two examples.
\begin{enumerate}
\item[] \textbf{Example 1.} $f=(z_1^2-z_1z_2z_3, z_2^2-z_3^3, z_3^2-z_1^3)$. The ideal $\text{Fitt}_2(\E_f, 0)$ is not homogeneous, for otherwise  the polynomials $z_1^2, z_2^2$ and $z_3^2$ must belong to the ideal generated by $f_1, f_2, f_3$, and it is easy to see this is impossible. Since $f$ here satisfies the assumption in Corollary \ref{cor3.38}, the  algebraic tangent cone  $\underline{\widehat {\E_f}}$ defined above is a stable bundle on $\C\P^2$. So our Theorem \ref{thm2} applies here, yielding that any admissible Hermitian-Yang-Mills connection on the germ of $\E_f$ at $0$ has a unique tangent cone which is a simple  HYM cone defined by the Hermitian-Einstein metric on $\underline{\widehat {\E_f}}$.

\item[] \textbf{Example 2.} Let $\E\rightarrow \C\P^{3}$ be given by the following exact sequence
\begin{equation}\label{globalexample}
0\rightarrow\O_{\C\P^{3}} \xrightarrow{s}\O_{\C\P^3}(3)^{\oplus 3}\rightarrow\E\rightarrow 0
\end{equation}
where $s=(z_0 z_1^2-z_1z_2z_3, z_0z_2^2-z_3^3, z_0z_3^2-z_1^3)\in H^0(\C\P^{3}, \O(3)^{\oplus 3})$. We know by Bezout's theorem $\text{Sing}(\E)=\{Z\in\C\P^{3}: s(Z)=0\}$ which consists of $27$ points (counted with multiplicities). Since $c_1(\E(-5))=-1$ and $H^0(\C\P^3, \E(-5))=0$, $\E$ is a stable reflexive sheaf. 

So by Theorem \ref{DUYBS} we know $\E$ admits an admissible Hermitian-Einstein metric. By choosing local coordinate and doing direct computations, $\Sing(\E)=\{Z_1, Z_2, \cdots, Z_{13}\}$ which is set of zeroes of $s$, where $Z_1=[1:0:0:0]$ is a zero of $s$ with multiplicity 8, and locally around $Z_1$ the sheaf $\E$ is modeled exactly by Example 1; $Z_2=[0:0:1:0]$ is also a zero of $s$ with multiplicity 8, whose local model is more complicated; all the other $Z_i$'s are simple zeroes of $s$ locally around which $\E$ is homogeneous and is isomorphic to the pull-back of the tangent bundle of $\C\P^2$. So using our results in this paper we know the tangent cones of the admissible Hermitian-Yang-Mills connection at $Z_i$ for $i\neq 2$. Indeed, we take $Z_1$ for example. Near $Z_1$, one can choose local coordinate given by $(\tau_1, \tau_2, \tau_3)=(\frac{z_1}{z_0}, \frac{z_2}{z_0}, \frac{z_3}{z_0})$ and $\frac{1}{z^3_0}$ to be a trivialization for $\O(3)$. In particular, near $Z_1$, $\E$ is modelled as follows 
$$
0\rightarrow \O_{B} \xrightarrow{(\tau_1^2-\tau_1\tau_2\tau_3, \tau_2^2-\tau_3^3, \tau_3^2-\tau_1^3)} \O_B^{\oplus 3} \rightarrow \E \rightarrow 0
$$
which is exactly given by Example 1. Other points are similar. 
\end{enumerate}

\end{document}